\documentclass{amsart}

\usepackage{amsmath,amsthm,verbatim,amssymb,amsfonts,amscd,graphicx}
\usepackage{xcolor,slashed}
\usepackage{soul}
\usepackage{graphics}
\usepackage{array}
\usepackage{hhline}

\usepackage{euscript, enumerate}
\usepackage{url,hyperref}

\newcolumntype{M}[1]{>{\centering\arraybackslash}m{#1}}
\newcolumntype{N}{@{}m{0pt}@{}}
\usepackage{etoolbox}

\makeatletter
\newcommand{\changeoperator}[1]{%
  \csletcs{#1@saved}{#1@}%
  \csdef{#1@}{\changed@operator{#1}}%
}
\newcommand{\changed@operator}[1]{%
  \mathop{%
    \mathchoice{\textstyle\csuse{#1@saved}}
               {\csuse{#1@saved}}
               {\csuse{#1@saved}}
               {\csuse{#1@saved}}%
  }%
}
\makeatother

\changeoperator{sum}
\changeoperator{prod}

\newcommand{\e}{\varepsilon}
\newcommand{\eps}{\varepsilon}

\newcommand{\be}{\begin{equation}}
\newcommand{\ba}{\begin{aligned}}
\newcommand{\bee}{\begin{equation*}}
\newcommand{\ee}{\end{equation}}
\newcommand{\ea}{\end{aligned}}
\newcommand{\eee}{\end{equation*}}
\newcommand{\bea}{\begin{equation} \begin{aligned} }
\newcommand{\eea}{\end{aligned}\end{equation} }

\newtheorem{theorem}{Theorem}[section]
\newtheorem{proposition}[theorem]{Proposition}
\newtheorem{lemma}[theorem]{Lemma}
\newtheorem{corollary}[theorem]{Corollary}

\newtheorem{question}[theorem]{Question}
\newtheorem{example}[theorem]{Example}

\theoremstyle{remark}

\newtheorem{remark}[theorem]{Remark}

\theoremstyle{definition}
\newtheorem{definition}[theorem]{Definition}

\numberwithin{equation}{section}

\title{Thom's gradient conjecture for nonlinear evolution equations}

\author{Beomjun Choi}
\address{Department of Mathematical Sciences, Korea Advanced Institute of Science and Technology (KAIST), Daejeon, 34141, Republic of Korea}
\email{bchoi@kaist.ac.kr}

\author{Pei-Ken Hung}
\address{Department of Mathematics, University of Illinois Urbana-Champaign, IL 61801, USA}
\email{pkhung@illinois.edu}

\thanks{This research was supported by the National Research Foundation of Korea (NRF) grants Nos. 2022R1C1C1013511 and RS-2023-00219980, the POSTECH Basic Science Research Institute Fund (NRF grant No. 2021R1A6A1A10042944), and Samsung Science \& Technology Foundation grant SSTF-BA2302-02.}

\begin{document}

\begin{abstract} 
R. Thom's gradient conjecture states that if a gradient flow of an analytic function converges to a limit, it does so along a unique limiting direction. In this paper, we extend and settle this conjecture in the context of infinite dimensional problems. Building on the foundational works of {\L}ojasiewicz, L. Simon, and the resolution of the conjecture for finite dimensional cases by Kurdyka-Mostowski-Parusinski, we focus on nonlinear evolutions on Riemannian manifolds as studied by L. Simon. This framework includes geometric PDEs such as minimal surface, harmonic map, mean curvature flow, and normalized Yamabe flow. Our main result not only confirms the uniqueness of the limiting direction but also characterizes the rate of convergence and possible limiting directions for both classical and infinite dimensional settings.
\end{abstract}

\maketitle

\section{Introduction}
Gradient flow, as an optimization process, is ubiquitous in various fields across analysis, geometry, mathematical physics, modeling, and machine learning. Given a potential function $f$ defined on a Euclidean space, the gradient flow is the ODE:
\begin{equation}\label{equ:gradient}
\dot{x}(t)=-\nabla f(x(t)).
\end{equation}
It is a classical problem to understand the large-time asymptotics for $x(t)$. Suppose there exists a sequence of times $t_i\to \infty$ such that $\lim_{i\to\infty}x(t_i)=x_\infty$. A natural question is if the convergence of  $x(t)$ to $x_\infty$ holds along \textbf{any} sequence of times diverging to infinity, or equivalently if $x(t)$ converges to $x_\infty$ as $t$ goes to infinity. For smooth potential functions, if no further assumption is made, the answer is negative and there are counterexamples known as Mexican hats (or goat's tracks).\footnote{The potential $f$ and explicit solution are presented in  \cite{doi:10.1137/040605266}, but these examples are previously known, for instance, in  \cite{MR0669541}.} Nevertheless, the answer is positive provided the potential function is real analytic. This is a fundamental result due to {\L}ojasiewicz \cite{Loj,lojasiewicz1965ensembles,Loj3} and is referred to as {\L}ojasiewicz's theorem.

\bigskip

In \cite{S0,S}, Leon Simon made a crucial observation that {\L}ojasiewicz gradient inequality can be generalized to a wide range of variational problems. The equations under consideration have the forms
\begin{equation}\label{equ:main}
{\ddot u}-m {\dot u}+\mathcal{M}_\Sigma u=N_1(u),  
\end{equation} for a non-zero constant $m$, and\begin{equation}\label{equ:parabolic}
{\dot u}-\mathcal{M}_\Sigma u=N_2(u).
\end{equation}
These two types of equations include notable geometric problems such as minimal surfaces and mean curvature flows. We refer readers to Section~\ref{sec:2} for the precise definitions of \eqref{equ:main} and \eqref{equ:parabolic} and to Section~\ref{sec:B} for more examples covered by these equations. Simon \cite{S0,S} developed the {\L}ojasiewicz-Simon gradient inequality and proved that a solution to \eqref{equ:main} or \eqref{equ:parabolic} has a unique limit point if it exists. Equivalently, suppose there exists a sequence of times $t_i\to\infty$ such that $\lim_{i\to\infty} u(t_i)=u_\infty$, then $\lim_{t\to\infty}u(t)=u_\infty$. We will refer to this result as Simon's theorem. In the context of minimal surfaces, Simon's theorem implies the uniqueness of the tangent cone provided one tangent cone has a smooth and compact link. {Since then, the \L ojasiewicz-Simon inequality and Simon's theorem have been central methods for studying the asymptotic convergence and uniqueness of blowups for equations in geometry \cite{MR2495079,MR3921325,CCR15,MR4332673,MR3211041,CM2015,MR4218365,lee2023uniqueness}, mathematical physics \cite{MR4199212,MR2020554}, and PDEs with gradient-like structures \cite{MR4101738,MR4129355,MR1800136,MR3926130}. The uniqueness of limits has made profound implications in structural theory about singular set: in the mean curvature flow, for example, see this in \cite{MR3381496,MR3489702}.} We note that {\L}ojasiewicz-Simon theory has been generalized to include the presence of constraints \cite{MR4129083,MR4701428} and to wave equations \cite{MR1616964,MR1714129}.

\bigskip

Simon's theorem provides the first-order asymptotics of solutions  to \eqref{equ:main} and \eqref{equ:parabolic}. A natural subsequent question is to investigate the next-order asymptotics. We consider the following question:
\begin{question}\label{que:1}
Find a function $r(t)$ in time and a function $u^{(1)}_\infty$ in space such that $$u(t)-u_\infty=r(t)u^{(1)}_\infty+o(r(t)).$$ 
\end{question}
\noindent The functions $r(t)$ and $u^{(1)}_\infty$ describe the {\it rate} and the {\it direction of convergence}, respectively. The next-order behavior becomes a crucial starting point for further analysis. First, there has been significant progress on the classification of ancient solutions in recent years. Building upon new methods, such as obtaining rotational symmetry \cite{MR3127066} or sharp higher-order asymptotics \cite{angenent2019unique_JDG}, non-collapsed ancient flows were classified for 3d Ricci flow and the mean curvature flow in $\mathbb{R}^3$ \cite{brendle2019uniqueness,ADS,MR4176064,MR4323639}. See  \cite{MR4448681} for the classification under low-entropy assumption and the proof of mean-convex neighborhood conjecture in $\mathbb{R}^3$. For higher dimensions, we refer \cite{MR4310889,MR4438354,MR4629007,CDDHS,MR4542417,ChoiMantoulidis} and references therein. We also mention \cite{SS} where the sharp next-order asymptotics were employed to classify minimal surfaces. For the singularity formation in parabolic problem, the next-order asymptotics at a singularity gives structural results on the singularity set in a neighborhood \cite{SunXue,Gang}. As the uniqueness of blow-ups implies the second differentiability of arrival time for the mean curvature flow \cite{CM2015,CM2018}, the next-order asymptotics and the convergence rate have a strong relation to further regularity of arrival time \cite{MR2200259,MR2383930}. 

In the context of finite-dimensional analytic gradient flows \eqref{equ:gradient}, the next-order asymptotics is closely related to the gradient conjecture of Ren\'e Thom {\cite{MR1067383,MR0713046,Loj2}}. Thom conjectured that the secant $ \frac {x(t)-x_\infty }{|x(t)-x_\infty|}$ converges to a limit when $t\to \infty$. Equivalently, the conjecture concerns the direction of convergence $x^{(1)}_\infty$. Kurdyka-Mostowski-Parusi\'nski \cite{KMP} settled this conjecture by showing a stronger result that the secant has a finite length on the unit sphere. However, a stronger conjecture by Arnold \cite{MR2078115}, concerning the convergence of the gradient $\frac{\dot x(t)}{|\dot x(t)|}$, remains open. {Note both conjectures may fail to hold for non-analytic potential functions \cite{MR4453693}.} On the other hand, Colding-Minicozzi \cite{MR3957399} showed that both conjectures hold for $C^2$-solutions to the level set equation for the mean curvature flow. They also conjectured that the principle that solutions of degenerate equations behave as though they are analytic should hold in greater generality. See also the recent improvement by Lee-Zhu \cite{lee2024arnold}, where the $C^2$-assumption was removed for solutions in $\mathbb{R}^3$. Thom's conjecture was considered in Hilbert spaces under functional analytic assumptions \cite{Schabrun2021}. 

\bigskip

In this paper, we settle Question~\ref{que:1} for the elliptic equations \eqref{equ:main} in Theorems~\ref{thm:general_e} and \ref{thm:general_exponential_e}, and for the parabolic equations \eqref{equ:parabolic} in Theorems~\ref{thm:general_p} and \ref{thm:general_exponential_p}. The theorems are applicable to  ODEs (see Remark \ref{rmk:cases} (2)) and it is noteworthy that identifying the convergence rate $r(t)$ is novel even for finite-dimensional gradient flows \eqref{equ:gradient}. In the following paragraphs, we discuss our main results in more detail.

By re-centering, we may and will assume $u_{\infty}=0$. The behavior of the solution's asymptotics differs significantly depending on whether it decays exponentially, \textit{i.e.} $\Vert u(t)\Vert_{L^2} \lesssim e^{-\delta t}$ for some $\delta>0$, or not. Our main contribution is to establish the direction and the rate of convergence for slowly (non-exponentially) decaying solutions. {Let $u(t)$ be a solution to \eqref{equ:main} or \eqref{equ:parabolic} that slowly converges to zero. In Theorems \ref{thm:general_p} and \ref{thm:general_e}, we show that there exist $r(t)=(\alpha_0\ell_*(\ell_*-2)t)^{-1/(\ell_*-2)}$ for some $\alpha_0>0$ and $\ell_*\geq 3$, and $u^{(1)}_\infty$ with $\|u^{(1)}_\infty\|_{L^2}=1$ such that
\begin{align*}
u(t)=r(t)u^{(1)}_\infty+o(r(t)).
\end{align*}
To explain the roles of $\alpha_0$ and $\ell_*$, let us consider the gradient flow \eqref{equ:gradient} with $f$ given by a homogeneous polynomial of degree $p\geq 3$ and set $\hat{f}$ to be the restriction of $f$ on the unit sphere. Suppose a unit vector $x^{(1)}_\infty$  is a critical point of $\hat{f}$ with $\beta_0=\hat{f}(x^{(1)}_\infty)>0$. Then it is straightforward to check that 
\begin{equation}\label{equ:ansatz}
x(t)=(\beta_0p(p-2)t)^{-1/(p-2)}x_\infty^{(1)}
\end{equation}
is a solution to \eqref{equ:gradient}. Hence $\alpha_0$ and $\ell_*$ could be viewed as generalizations of critical values and degrees respectively. Investigating possible $u^{(1)}_\infty$ and $r(t)$ for slowly decaying solutions, we provide a partial characterization which is closely related to the Adams-Simon condition. In \cite{AS}, Adams-Simon discovered a sufficient condition (later called the Adams-Simon positivity condition or simply $AS_p$ condition) which allows them to construct a slowly converging solution to \eqref{equ:main} starting with an ansatz of type \eqref{equ:ansatz}. Here, $p\in \mathbb{N}_{\ge 3}$ is so-called the order of integrability and it is an integer solely determined by the energy functional. Later, Carlotto-Chodosh-Rubinstein \cite{CCR15} found explicit examples of critical points of normalized Yamabe functional with the $AS_p$ condition and constructed normalized Yamabe flows converging slowly. As questioned in \cite[p.1533]{CCR15}, it is of interest to understand whether such a slowly converging solution must satisfy the Adams-Simon positivity condition and whether the next-order asymptotics follows the ansatz of type \eqref{equ:ansatz}. Our partial characterization shows that $\ell_*\ge p$ and the Adams-Simon `non-negativity' condition is necessary for slowly converging solutions to exist.}

The next-order asymptotics of fast (exponentially) decaying solutions are discussed in Theorems \ref{thm:general_exponential_p} and \ref{thm:general_exponential_e}. In this case, \eqref{equ:main} and \eqref{equ:parabolic} are well-approximated by the linearized equations. The possible asymptotics are then determined by the non-zero eigenvalues and eigendirections of $\mathcal{L}_\Sigma$, the linearization of $\mathcal{M}_\Sigma$. Note also that $-\mathcal{L}_\Sigma$ is the second variation of the energy functional. 
 Results of this kind have been anticipated by experts and have been obtained for specific problems, such as those in \cite{SS} and \cite{ChoiMantoulidis}. Nevertheless, we could not find proper literature covering the general forms \eqref{equ:main} and \eqref{equ:parabolic}, so we provide a proof in Section~\ref{sec:exp}. The linearization of the elliptic equation \eqref{equ:main} is the second-order equation ${\ddot u}-m{\dot u}+\mathcal{L}_\Sigma u=0$. If $\mathcal{L}_\Sigma$ has an eigenvalue larger than $4^{-1}m^2$, there exists an underdamped solution which oscillates while decaying exponentially as $t$ goes to infinity or minus infinity. This type of solutions has been constructed for minimal graphs over Simons cones in $\mathbb{R}^4$ and $\mathbb{R}^6$: see \cite[Remark 1.21]{MR3838575} and \cite{BGG} for more details. For this reason, the original form of Thom's gradient conjecture is not true in the elliptic problem. Moreover, if $4^{-1}m^2$ is an eigenvalue of $\mathcal{L}_\Sigma$ and $m<0$, then a resonance possibly resulting in a solution that decays at a rate $t e^{2^{-1}m t}$. For instance, this possibility was explored for minimal graphs over stable (but not strictly stable) minimal cones in \cite{HS}. 

\bigskip
            
This paper is organized as follows. In Section \ref{sec:2}, we introduce the notation, condition, spectral property of linearized operator $\mathcal{L}_\Sigma$, and main theorems. In Section \ref{sec:1storderode}, we set up the elliptic problem \eqref{equ:main} as a first order ODE system on function spaces. In Section \ref{sec:exp}, Theorem \ref{thm:general_exponential_p} and Theorem \ref{thm:general_exponential_e} for exponentially decaying solutions to parabolic and elliptic equations are proved, respectively. In Section \ref{sec:elliptic}, we show in Proposition~\ref{prop-neutral-dynamics} that slowly converging solutions to \eqref{equ:main} are governed by a finite-dimensional gradient flow with a small perturbation. The parabolic analogue, Proposition~\ref{prop-neutral-dynamics-p1}, is proved in Section~\ref{sec:slowparabolic}. In Section~\ref{sec:7}, we complete the proofs for Theorems~\ref{thm:general_p} and \ref{thm:general_e} by analyzing a finite-dimensional gradient flow with a small perturbation, using a version of the {\L}ojasiewicz argument motivated by \cite{KMP}. In Section~\ref{sec:B}, we provide some notable examples, including (rescaled) mean curvature flow, to which the main results apply. Appendix~\ref{sec:A} contains auxiliary tools we need in the paper.  

\begin{table}
\begin{center}
\begin{tabular}{M{4cm}|M{3.8cm}|M{3.8cm} N}

 & parabolic PDE &  gradient flow  &  \\[3pt]  
 & $\dot{u}-\mathcal{M}_\Sigma u=N_2(u)$ &$\dot{x}(t)=-\nabla f(x(t))$ &\\ [3pt]
 \hhline{=|=|=}
\raisebox{-4pt}{uniqueness of limit} &  \raisebox{-4pt}{Simon's Theorem}  &   \raisebox{-4pt}{\L ojasiewicz's Theorem} & \\[10pt] 
\hline 
\raisebox{-4pt}{direction of convergence} &\raisebox{-4pt}{Theorems~\ref{thm:general_p} and \ref{thm:general_exponential_p}} & \raisebox{-4pt}{Thom's conjecture }& \\ [10pt] 
\hline 
\raisebox{-4pt}{rate of convergence} & \raisebox{-4pt}{Theorems~\ref{thm:general_p} and \ref{thm:general_exponential_p}} & \raisebox{-4pt}{Theorem~\ref{thm:gradflow}} & \\ [10pt] 
\end{tabular}
\end{center}
\caption{comparison between results for \eqref{equ:parabolic} and \eqref{equ:gradient}}
\end{table}

\section{Preliminary}\label{sec:2}
Let us introduce our setting to study \eqref{equ:main} and \eqref{equ:parabolic} and state the main results. Let $\Sigma$ be a closed $n$-dimensional Riemannian manifold and $d\mu$ be a smooth volume form on $\Sigma$ which is mutually absolutely continuous with respect to the volume form induced by the metric. Let $\mathbf{V}\to \Sigma$ be a smooth vector bundle equipped with a smooth inner product $\left\langle\cdot,\cdot\right\rangle$. For $u\in\mathbf{V}$, we write $|u|^2:=\left\langle u,u \right\rangle$. For $\omega\in\Sigma$, we denote by $\mathbf{V}_\omega$ the fiber over $\omega$. Let $\slashed{\nabla}$ be a connection on $\mathbf{V}$ which is compatible with the inner product.

We denote by $L^2(\Sigma;\mathbf{V})$ the space of $L^2$-sections of $\mathbf{V}$ with respect to $d\mu$. Namely, a section $u$ belongs to $L^2(\Sigma;\mathbf{V})$ provided
\begin{align*}
\|u\|^2_{L^2}:=\int_\Sigma |u|^2\, d\mu <\infty.
\end{align*}
For $\ell \in\mathbb{N}_0$, we denote by $H^\ell(\Sigma;\mathbf{V})$ the collection of sections that satisfy
\begin{align*}
\|u\|^2_{H^\ell}:= \sum_{i=0}^\ell \int_\Sigma \big |\slashed{\nabla}^i u  \big|^2\, d\mu  <\infty.
\end{align*}
For $-\infty< a<b\leq \infty$, we define $Q_{a,b}:=\Sigma\times [a,b)$ and equip $Q_{a,b}$ with the product metric. We denote by $\widetilde{\mathbf{V}}$ the pull back bundle of $\mathbf{V}$ through the projection $  Q_{a,b}\to \Sigma$. For $u\in C^s(Q_{a,b};\widetilde{\mathbf{V}})$ and $t\in [a,b)$, $u(t)\in C^s(\Sigma;\mathbf{V})$ is defined through $u(t)(\omega):=u(\omega,t)$. Its $H^\ell$-norm is denoted by
\begin{equation*}
\|u\|_{H^\ell}(t):= \|u(t) \|_{H^\ell}.
\end{equation*} 
Also, it is convenient to define 
\begin{align*}
\|u\|_{C^s}(t):=\sup_{\omega\in\Sigma}\sup_{k+\ell\leq s}\left| \frac{\partial^k}{\partial t^k}\slashed{\nabla}^\ell u(\omega,t) \right|.
\end{align*}
Here, note that {\it the norms of time derivatives} are also included  in the definition of $\Vert \cdot \Vert _{C^s}  (t)$.
Let ${\dot u}$ and ${\ddot u}$ denote $\frac{\partial u }{\partial t}$ and $\frac{\partial^2 u }{\partial t^2}$, respectively.

\vspace*{0.3cm}

Our assumption on $\mathcal{M}_\Sigma$ is almost identical to the one in \cite{S0}. The difference is that the volume form, $d\mu$, is not necessarily the one induced from the Riemannian metric. Let $\mathcal{F}_\Sigma$ be a functional on $u\in C^{1}(\Sigma;\mathbf{V})$
\begin{align}\label{eq-F}
\mathcal{F}_\Sigma(u) :=\int_{\Sigma} F(\omega,u,\slashed{\nabla}u)\, d\mu.
\end{align}
Here the integrand $F$ satisfies the following assumptions:
\begin{enumerate}[$(1)$]
\item $F=F(\omega,z,p)$ is a smooth function defined on an open
neighborhood of the zero section in
$\mathbf V\oplus(T^*\Sigma\otimes\mathbf V)$.
\item $F(\omega,\cdot,\cdot)$ is real analytic in the fiber variables
$(z,p)$, uniformly in $\omega\in\Sigma$, for $|z|+|p|$ sufficiently small. cf.~\cite[pp.~528--529]{S0}
\item $F$ satisfies the Legendre-Hadamard ellipticity condition
\begin{align*}
D^2_pF(\omega,0,0)[\eta\otimes\xi,\eta\otimes\xi]\geq c|\eta|^2|\xi|^2,
\end{align*}
for $c>0$ independent of $\omega\in\Sigma$, $\eta\in T^{*}_\omega\Sigma$ and $\xi\in \mathbf{V}_\omega$.
\end{enumerate}
$\mathcal{M}_{\Sigma}$ is defined to be the negative Euler-Lagrange operator of $\mathcal{F}_\Sigma$. Namely, for any $\zeta\in C^{\infty}(\Sigma;\mathbf{V})$,
\begin{align}\label{eq-M}
\int_\Sigma \left\langle \mathcal{M}_\Sigma( u),\zeta \right\rangle \, d\mu=\left.-\frac{d}{ds}\mathcal{F}_\Sigma(u+s\zeta)\right|_{s=0}.
\end{align}
We further assume $0$ is a critical point of $\mathcal{F}_\Sigma$. Namely,  $\mathcal{M}_\Sigma(0)=0$. We denote by $\mathcal{L}_{\Sigma}$ the linearization of $\mathcal{M}_\Sigma$ at $0$. It is clear that the difference between $\mathcal{M}_{\Sigma} (u)$ and $\mathcal{L}_{\Sigma} u$ is a quadratic term of the form 
\begin{equation}\label{equ:quasilinear}
 \mathcal{M}_\Sigma u-\mathcal{L}_{\Sigma} u =\sum_{j=0}^2 c_j\cdot\slashed{\nabla}^j u,
\end{equation} 
where $c_j=c_j(\omega, u,\slashed{\nabla}u)$ are smooth with $c_j(\omega,0,0)=0$. The Legendre-Hadamard condition implies $\mathcal{L}_{\Sigma}$ is elliptic. There exists a constant $C>0$ such that 
\begin{align}\label{equ:H1}
\int_\Sigma -\left\langle \mathcal{L}_{\Sigma} u,u \right\rangle+C|u|^2 \, d\mu \asymp  \|u\|^2_{H^1}.
\end{align}
Here, and in the rest of the paper, we use the notation $\asymp$ to indicate that two non-negative quantities are comparable. In other words, we write $A\asymp B$ provided $C^{-1}B\le A\leq CB$ for some $C>0$ independent of $A$ and $B$. $\mathcal{L}_{\Sigma}$ is self-adjoint with respect to $d\mu$. This can be seen from
\begin{align} \label{equ:hessianF}
\int_{\Sigma} \left\langle \mathcal{L}_{\Sigma} u, v \right\rangle\, d\mu=-\frac{\partial^2}{\partial s_1\partial s_2}\mathcal{F}_\Sigma (s_1u+s_2v)\bigg|_{s_1=s_2=0}. 
\end{align}

 We assume $N_1(u)$ and $N_2(u)$ in \eqref{equ:main} and \eqref{equ:parabolic} are of the form
\begin{equation}\label{equ:N}
\begin{split}
N_1(u)=&a_1\cdot D {\dot u}+a_2\cdot {\dot u}+a_3\cdot\mathcal{M}_\Sigma (u),\\
 N_2(u)= & b_1\cdot\mathcal{M}_\Sigma (u). 
\end{split}
\end{equation}
Here $D=\left\{ \frac{\partial}{\partial t},\slashed{\nabla} \right\}$; $a_1,a_3$ depend smoothly on $(\omega,u,Du)$ with $a_1(\omega,0,0)\equiv 0$ and $a_3(\omega,0,0)\equiv 0$; 
$a_2$ depends smoothly on $(\omega,u,Du,D^2u)$ and has the form
\[
a_2=c(\omega,u,Du)\cdot D^2u+d(\omega,u,Du),
\]
where $c$ and $d$ are smooth and $d(\omega,0,0)\equiv0$.\footnote{
Together with the other assumptions, this makes
\eqref{equ:main} a quasilinear system that is uniformly elliptic
for $|u|+|Du|$ sufficiently small; cf.~\cite[(1.9)]{S0}.}; $b_1$ depends smoothly on $(\omega,u,\slashed{\nabla}u)$ with $b_1(\omega,0,0)\equiv 0$. The structure \eqref{equ:N} allows us to apply results in \cite{S0} and is slightly more general than the one in \cite{S,AS}.

{In view of the elliptic energy estimate and \eqref{equ:hessianF},  $\mathcal{L}_\Sigma -\lambda I$, for large $\lambda<\infty$, has a well-defined inverse which is a compact self-adjoint operator on $L^2(\Sigma;\mathbf{V})$. By the spectral decomposition, there exist non-increasing eigenvalues $\lambda_1 \ge \lambda_2 \ge \ldots$ of $\mathcal{L}_\Sigma$ and corresponding eigensections $\varphi_1$, $\varphi_2$, $\ldots$ which form a complete orthonormal basis of $L^2(\Sigma;\mathbf{V})$. Note that $\lambda_i\to -\infty$.}
We denote by $\Pi^T$ and $\Pi^\perp$ the orthogonal projection of $L^2(\Sigma;\mathbf{V})$ to the $0$-eigenspace $\ker \mathcal{L}_{\Sigma}$ and its orthogonal complement $\left(\ker \mathcal{L}_{\Sigma}\right)^\perp$ respectively.
 The following Lyapunov--Schmidt reduction is crucial for this paper.
\begin{proposition} [Lyapunov--Schmidt reduction]\label{pro:AS} 
Let $B_\rho$ be the open ball of radius $\rho$ in $L^2(\Sigma;\mathbf{V})$. There exist $\rho>0$ and a map \[H: \ker \mathcal{L}_{\Sigma} \cap B_\rho \rightarrow C^\infty (\Sigma;\mathbf{V} )\cap (\ker \mathcal{L}_{\Sigma}) ^\perp \] such that the following statements hold. First, for any $k\geq 2$ and $\alpha\in (0,1)$, $H$ is an analytic map from $\ker \mathcal{L}_{\Sigma} \cap B_\rho$ to $C^{k,\alpha}(\Sigma;\mathbf{V})$. Second,  
\[H(0)=0,\, DH(0)=0. \]
Lastly,
\[\begin{cases} \begin{aligned} \Pi ^\perp   \mathcal{M}_{\Sigma} (v+H(v) ) &= 0,\\ 
\Pi ^T \mathcal{M}_{\Sigma} (v+H(v) )& = -\nabla f (v).
\end{aligned} \end{cases}\]
Here $f:\ker \mathcal{L}_{\Sigma}  \cap B_\rho \to \mathbb{R}$ is given by $f(v):= \mathcal{F}_\Sigma (v+H(v)).$ 
The gradient $\nabla f(v)\in \ker\mathcal{L}_{\Sigma}$ is defined by
\[
\int_{\Sigma}\langle \nabla f(v),w\rangle\,d\mu
=
\left.\frac{d}{ds} f(v+sw)\right|_{s=0}
\]
for all $w\in \ker\mathcal{L}_{\Sigma}$.

\end{proposition}
\begin{proof}
This is essentially proved in \cite[\S 3.12]{S1}; we recall the argument for completeness. Define
\[
\mathcal{N}:=\Pi^T+\mathcal{M}_{\Sigma},
\]
viewed as a map from $C^{2,\alpha}(\Sigma;\mathbf{V})$ to
$C^{0,\alpha}(\Sigma;\mathbf{V})$. By the inverse function theorem, $\mathcal{N}$ is invertible in a neighborhood of the origin. We denote its local inverse by $\mathcal{N}^{-1}$.\footnote{In \cite{S1}, $\mathcal{N}^{-1}$ is denoted by $\Psi$.}

Applying the (finite dimensional) inverse function theorem to the map $A:\ker \mathcal{L}_{\Sigma} \to \ker \mathcal{L}_{\Sigma}$ defined by $$ w \mapsto  \Pi ^T \mathcal{N}^{-1}(w)$$near the origin and then defining $H(v):=\mathcal{N}^{-1}(A^{-1}(v))- v$, one obtains a map
\[
H:\ker \mathcal{L}_{\Sigma}\cap B_\rho
\longrightarrow
(\ker \mathcal{L}_{\Sigma})^\perp
\]
such that, for each $v\in \ker \mathcal{L}_{\Sigma}\cap B_\rho$, the section $v+H(v)$ lies in the Lyapunov--Schmidt slice; equivalently, $v+H(v)=\mathcal{N}^{-1}(w)$ for some $w\in\ker\mathcal{L}_{\Sigma}$. Since $\mathcal{N}^{-1}(0)=0$ and $D\mathcal{N}^{-1}(0)$ restricts to the identity on $\ker\mathcal{L}_{\Sigma}$ at the origin, we have
\[
H(0)=0,\qquad DH(0)=0.
\]
The smoothness of $H(v)$ follows from elliptic regularity, while the analyticity of $H$ follows from the analyticity of $\mathcal{F}_{\Sigma}$. The identity
\[
\Pi^\perp \mathcal{M}_{\Sigma}(v+H(v))=0
\]
is proved in \cite[\S 3.12(iv)]{S1}.

It remains to identify the tangential component. For $w\in\ker\mathcal{L}_{\Sigma}$, using the first variation formula for $\mathcal{F}_\Sigma$, we compute
\begin{align*}
\int_{\Sigma}\langle \nabla f(v),w\rangle\,d\mu
&=
\left.\frac{d}{ds} f(v+sw)\right|_{s=0} \\
&=
-\int_{\Sigma}
\left\langle
\mathcal{M}_{\Sigma}(v+H(v)),
w+DH(v)[w]
\right\rangle\,d\mu \\
&=
-\int_{\Sigma}
\left\langle
\mathcal{M}_{\Sigma}(v+H(v)),
w
\right\rangle\,d\mu.
\end{align*}
In the last equality we used $DH(v)[w]\in(\ker\mathcal{L}_{\Sigma})^\perp$ and $\Pi^\perp \mathcal{M}_{\Sigma}(v+H(v))=0$.
Since this holds for all $w\in\ker\mathcal{L}_{\Sigma}$, it follows that
\[
\Pi^T \mathcal{M}_{\Sigma}(v+H(v))=-\nabla f(v).
\]
This completes the proof.
\end{proof}
For solutions $u(t)$ whose trajectories are asymptotic to $\ker \mathcal{L}_\Sigma$, the function $f$ plays a crucial role and we will call it \textbf{reduced functional}. The reduced functional $f$ is real analytic with $\nabla f(0)=0$ and $\nabla^2 f(0)=0$. Let $J$ be the dimension of $\ker \mathcal{L}_{\Sigma}$ and $\ker \mathcal{L}_{\Sigma}=\textup{span}\{\varphi_{\iota+1},\varphi_{\iota+2},\dots, \varphi_{\iota+J}\}$ .
We may view $f$ as an analytic function defined on an open ball in $\mathbb{R}^J$
through the identification
\begin{align}\label{equ:identification}
(x_1,x_2,\dots, x_J)\mapsto \sum_{j=1}^J x_j\varphi_{\iota+j}.
\end{align}
Under this identification, we shall also regard $\nabla f$ as an
$\mathbb{R}^J$-valued map. More precisely, if
\[
v=\sum_{j=1}^J x_j\varphi_{\iota+j}
\qquad\text{and}\qquad
\bar f(x):=f(v),
\]
then
\[
\partial_{x_j}\bar f(x)
=
\left\langle \nabla f(v),\varphi_{\iota+j}\right\rangle_{L^2_\mu}.
\]
By a slight abuse of notation, we will write $f(x)$ and $\nabla f(x)$ for
$\bar f(x)$ and $\nabla\bar f(x)$, respectively, viewing the latter as a vector
in $\mathbb{R}^J$.

Let us review the integrability condition. The kernel $\ker \mathcal{L}_{\Sigma}$ is called {\it integrable} if for any $v\in\ker \mathcal{L}_{\Sigma}$, there exists a family $\{v_s\}_{s\in (0,1)}\subset C^2(\Sigma;\mathbf{V})$ such that $v_s\to 0$ in $C^2(\Sigma;\mathbf{V})$, $\mathcal{M}_\Sigma (v_s)\equiv 0$ and $\lim_{s\to 0} v_s/s=v$ in $L^2(\Sigma;\mathbf{V})$. It is well-known \cite[Lemma 1]{AS} that $\ker \mathcal{L}_{\Sigma}$ is integrable if and only if the reduced functional $f$ is a constant. Moreover, if integrability condition is satisfied, any decaying solution to \eqref{equ:main} or \eqref{equ:parabolic} decays exponentially \cite{MR0607893,AS,CCR15}. For this reason, whenever non-exponentially decaying solution is considered, the integrability condition should necessarily \textbf{fail}. Namely, the reduced functional $f$ is not a constant function. In particular, there exists an integer $p\geq 3$ such that
\begin{equation}\label{def:fp}
f=f(0)+\sum_{j\geq p} f_j,
\end{equation}
where $f_j$ are homogeneous polynomials with degree $j$ and $f_p\not\equiv 0$. Following \cite{CCR15}, we call this integer the order of integrability.

As suggested  in \eqref{equ:ansatz}, the gradient flow of $f_p$ has a dominant role in the asymptotic behavior. Let $\hat{f}_p$ be the restriction of $f_p$ on $\{ w\in \ker \mathcal{L}_{\Sigma}  \, :\, \Vert w\Vert _{L^2}=1\} \approx \mathbb{S}^{J-1}$. Consider the critical points of $\hat{f}_p$:
\begin{equation}\label{def:C}
\mathbf{C}:=\left\{ w\in \ker \mathcal{L}_{\Sigma}\, :\, \Vert w\Vert_{L^2}=1\ \textup{and}\ w\ \textup{is a critical point of}\ \hat{f}_p \right\}.
\end{equation}
If $w\in\mathbf{C}$ satisfies $ f_p(w)>0$, then $x(t)=   [ p(p-2)f_p(w) t ] ^{-\frac{1}{p-2}} w$ becomes a radial solution to the flow $x'= -\nabla f_p(x)$. The higher order asymptotics in Theorem \ref{thm:general_e} and \ref{thm:general_p} will be modeled on such solutions.

\begin{definition}[Adams-Simon conditions. \textit{c.f.} (4.1) in \cite{AS}] We say $\Sigma$ satisfies the Adams-Simon non-negativity condition for the parabolic problem  \eqref{equ:parabolic} if there exists $w\in\mathbf{C}$ such that  ${f}_p(w)\geq 0$. 
We say $\Sigma$ satisfies the Adams-Simon non-negativity condition for the elliptic problem \eqref{equ:main} if there exists $w\in\mathbf{C}$ such that  $m^{-1} {f}_p(w)\geq 0$. 

In both equations, the Adams-Simon positivity conditions are defined similarly by requiring that {there exists a critical point of positive critical value.} 
\end{definition}

Let us state main results concerning the asymptotic behavior and the convergence rate of decaying solutions. {We begin with the parabolic equation~\eqref{equ:parabolic}. The behavior can be naturally divided into mutually exclusive cases according to the convergence rates, or equivalently to the subspaces of $L^2(\Sigma;\mathbf{V})$ to which solutions are asymptotic. We state them separately. 
\begin{theorem}[slow decay in parabolic equation]\label{thm:general_p} {There exists a finite set $\mathcal{Z}=\mathcal{Z}(\mathcal{F}_\Sigma) \subset \mathbb{Q}_{\ge3}\times \mathbb{R}_{>0}$\footnote{$\mathbb{Q}_{\ge3}:=\{ x\in \mathbb{Q}\,:\, x\ge3\}$ and $\mathbb{R}_{>0}:=\{x\in \mathbb{R}\,:\, x>0\}$.} which has the following significance: For a given solution  $u\in C^\infty(Q_{0,\infty},\widetilde{ \mathbf{V}})$ to the equation \eqref{equ:parabolic}, if $\|u\|_{H^{n+4}}(t)=o(1)$ but it does not decay exponentially as $t\to \infty$, then there exists $(\ell_*,\alpha_0)\in \mathcal{Z}$ such that}  
\begin{align*}
\lim_{t\to\infty} t^{1/(\ell_* -2)}\Vert u(t)\Vert_{L^2}=(\alpha_0\ell_*(\ell_*-2))^{-1/(\ell_*-2)}.  
\end{align*}
Moreover,
\begin{align*}
\lim_{t\to \infty  }{u(t)}/{\Vert u(t)\Vert _{L^2}}=w\ \textup{in}\ C^\infty(\Sigma;\mathbf{V}),
\end{align*}
for some $w \in \mathbf{C}\subset \ker \mathcal{L}_\Sigma $ with $ \hat f_p (w)\geq 0 $. In particular, $\Sigma$ satisfies the Adams-Simon non-negativity condition for \eqref{equ:parabolic}. 
\end{theorem}

{
The definitions of $\mathcal{Z}$ and the corresponding one for the elliptic problem in Theorem~\ref{thm:general_e} are postponed to Section~\ref{sec:7}. Indeed, we prove a stronger result that $\ell_*$ cannot be smaller than the order of integrability $p \geq 3$. See \cite{CCR15} for examples of slowly decaying solutions with $\ell_* = p \geq 4$. According to our proof, a slow decay solution with $\ell_*>p$ may exist only if there is $w\in \mathbf{C}$ with $\hat f_p(w)=0$. For a heuristic explanation of $\ell_*$ and $\alpha_0$, we refer readers to the paragraph around \eqref{equ:ansatz}.}

\begin{theorem}[fast decay in parabolic equation]\label{thm:general_exponential_p}
Let $u\in C^\infty(Q_{0,\infty},\widetilde{ \mathbf{V}})$ be a non-zero solution to \eqref{equ:parabolic} with $\|u\|_{H^{n+4}}(t)=O(e^{-\varepsilon t})$ for some $\varepsilon>0$ as $t\to \infty$. Then there exists a negative eigenvalue $\lambda$ of $\mathcal{L}_{\Sigma}$ such that 
\begin{align*}
\lim_{t\to \infty} e^{-\lambda t} \Vert u(t)\Vert _{L^2}\in (0,\infty).
\end{align*}
Moreover, for some eigensection $w$ with $\mathcal{L}_{\Sigma} w =\lambda w$, 
\begin{align*}
\lim_{t\to \infty  }{u(t)}/{\Vert u(t)\Vert _{L^2}}=w\ \textup{in}\ C^\infty(\Sigma;\mathbf{V}).
\end{align*} 
\end{theorem}
 In the parabolic case, the linearized equation around the zero section, $\partial_t u = \mathcal{L}_{\Sigma}u$, has solutions (Jacobi fields) $e^{\lambda_i t} \varphi_i$ for $i\ge 0$ and they model the exponential decay in Theorem \ref{thm:general_exponential_p}. Two theorems have the following consequence. \begin{corollary}[Thom's gradient conjecture for parabolic system]\label{cor:parabolic_thom}  
Let $u$ be a non-zero decaying solution to the parabolic equation \eqref{equ:parabolic}. Then the secant has a unique limit\[ \lim_{t\to \infty} u(t)/ \Vert u(t)\Vert_{L^2} = w \text { in } C^\infty (\Sigma, \mathbf{V}) .\]
The limit $w$ is an eigensection $\mathcal{L}_{\Sigma}w=\lambda w$ for some non-positive eigenvalue $\lambda$. 

\end{corollary}

\bigskip

Let us discuss the elliptic equation \eqref{equ:main}. In this case, the linearized equation  $\ddot u -m \dot u  +\mathcal{L}_\Sigma u =0$ has solutions $e^{\gamma_i^+ t}\varphi_i$ and $e^{\gamma_i^- t}\varphi_i$ where $\gamma_i^\pm  =2^{-1}m\pm\sqrt{4^{-1}m^2-\lambda_i}$. According to $\gamma_i^\pm$, we separate the set of indices $\mathbb{N}$ into four indices.
\begin{equation}\label{def:I123}
\begin{split}
I_1:= &\{i\in\mathbb{N}\, :\, \lambda_i>4^{-1}m^2\},\ 
I_2:=  \{i\in\mathbb{N}\, :\, \lambda_i=4^{-1}m^2\},\\ 
I_3:= & \{i\in\mathbb{N}\, :\, \lambda_i=0\},\
I_4:= \mathbb{N}\setminus (I_1\cup I_2 \cup I_3).
\end{split}
\end{equation} 
$I_1$,$I_2$,$I_3$ correspond to the indices whose characteristic equations have complex roots, repeated roots, and (one) zero root, respectively. Note that $I_1$,$I_2$,$I_3$ are (possibly empty) finite sets and $\ker \mathcal{L}_{\Sigma}$ is spanned by $\{\varphi_i\}_{i\in I_3}$.
As $J=\dim \ker \mathcal{L}_\Sigma$, this implies $I_3=\{\iota+ 1,\iota+2,\dots, \iota+J\}$ for some $\iota\in\mathbb{N}_0$.

\begin{theorem} [slow decay in elliptic equation]  \label{thm:general_e} {There exists a finite set $\mathcal{Z}=\mathcal{Z}(\mathcal{F}_\Sigma,m ) \subset \mathbb{Q}_{\ge 3}\times \mathbb{R}_{>0}$ with the following significance: For a given solution  $u\in C^\infty(Q_{0,\infty},\widetilde{ \mathbf{V}})$  to the equation \eqref{equ:main}, if  $\|u\|_{C^1}(t)=o(1)$ but it  does not decay exponentially as $t\to\infty$, then there is $(\ell_*,\alpha_0 )\in \mathcal{Z}$ such that }
\begin{align*}
\lim_{t\to\infty} t^{1/(\ell_* -2)}\Vert u(t)\Vert_{L^2}=(\alpha_0\ell_*(\ell_*-2))^{-1/(\ell_*-2)}.  
\end{align*}
Moreover,
\begin{align*}
\lim_{t\to \infty  }{u(t)}/{\Vert u(t)\Vert _{L^2}}=w\ \textup{in}\ C^\infty(\Sigma;\mathbf{V}),
\end{align*}
for some $w \in \mathbf{C}\subset \ker \mathcal{L}_{\Sigma}$ with $m^{-1} \hat f_p (w)\geq 0 $. In particular, $\Sigma$ satisfies the Adams-Simon non-negativity condition for \eqref{equ:main}.
\end{theorem}

\begin{remark}
As \cite{S0,S}, the condition $m\neq 0$ is crucial. For slow decay, an intermediate step is to show $\ddot u$ is negligible, letting us to treat the equation roughly as a gradient flow. This idea was used in constructing such a solution \cite{AS}.\end{remark}

\begin{theorem}[fast decay in elliptic equation]\label{thm:general_exponential_e}
Let $u\in C^\infty(Q_{0,\infty},\widetilde{ \mathbf{V}})$ be a non-zero solution to \eqref{equ:main} with $\|u\|_{C^1}(t)=O(e^{-\varepsilon t})$ for some $\varepsilon>0$ as $t\to \infty$. Then one of the following alternatives holds: 
\begin{enumerate}[$(1)$]
\item  There exists an eigenvalue $\lambda<4^{-1}m^2$ of $\mathcal{L}_{\Sigma}$ such that 
\begin{align*}
&\gamma^+<0\ \textup{and}\  \lim_{t\to \infty} e^{-\gamma^+t} \Vert u(t)\Vert _{L^2}\in (0,\infty), \text{ or}\\ 
&\gamma^-<0\ \textup{and}\ \lim_{t\to \infty} e^{-\gamma^-t} \Vert u(t)\Vert _{L^2}\in (0,\infty) .
\end{align*}
Here $\gamma^\pm=2^{-1}m\pm\sqrt{4^{-1}m^2-\lambda}. $ 
Moreover, for some eigensection $w$ with $\mathcal{L}_{\Sigma} w =\lambda w$,  
\begin{align*}
\lim_{t\to \infty  }{u(t)}/{\Vert u(t)\Vert _{L^2}}=w\ \textup{in}\ C^\infty(\Sigma;\mathbf{V}).
\end{align*} 

\item There holds $m<0$ and $\lim_{t\to \infty}  t^{-1}e^{-2^{-1}m t} \Vert u(t)\Vert _{L^2}\in (0,\infty)$.
Moreover, for some eigensection $w$ with $\mathcal{L}_{\Sigma} w =4^{-1}m^2 w$,  
\begin{align*}
\lim_{t\to \infty  }{u(t)}/{\Vert u(t)\Vert _{L^2}}=w\ \textup{in}\ C^\infty(\Sigma;\mathbf{V}).
\end{align*}
\item There holds $m<0$ and $\limsup_{t\to \infty} e^{-2^{-1}m t} \Vert u(t)\Vert _{L^2} \in (0,\infty)$.
Moreover, there exist ${d_i}\in\mathbb{C}$ for $i\in I_1$ and $c_{i}\in\mathbb{R}$ for $i\in  I_2$ such that
\begin{align*}
\lim_{t\to \infty  }\bigg ( e^{-2^{-1}mt}u(t) -\sum_{i\in I_1}\textup{Re} \left({d_i}e^{\mathbf{i}\beta_i t}  \right)\varphi_i-\sum_{i\in I_2}c_i \varphi_i\bigg ) =0\   \textup{in}\ C^\infty(\Sigma;\mathbf{V}).
\end{align*}
Here $\beta_i=\sqrt{\lambda_i-4^{-1}m^2}$ and $\mathbf{i}=\sqrt{-1}$. 
\end{enumerate}

\end{theorem}

Theorem \ref{thm:general_p} and \ref{thm:general_e} provide a necessary condition for the existence of a solution that does not decay exponentially. As previously mentioned, if $\ker  \mathcal{L}_\Sigma$ is integrable (or equivalently, $f$ is constant around the origin), then every decaying solution converges exponentially. For the parabolic equation \eqref{equ:parabolic}, Theorem \ref{thm:general_p} shows the same result holds if the $\hat{f}_p$ has only negative critical values. A corresponding result holds for the elliptic equation \eqref{equ:main} with $\hat{f}_p$ replaced by $m^{-1}\hat{f}_p$. Heuristically, this non-existence is motivated by a simple observation that one-dimensional gradient flow $\dot x=-(a x^p)'$ has no decaying solution if $a<0$ and $p$ is an even integer. Indeed, this phenomenon is not exclusive to gradient flows. In \cite{choi2021translating,choisunclassification}, slowly decaying solutions are ruled out in the context of sub-critical Gauss curvature flow and curve shortening flow for similar reasons: for a slow decaying solution to exist, the modulus of the projection onto the kernel satisfies an ODE of the form $\dot x \approx - (ax^4)'$ some $a<0$ and this readily gives a contradiction. To our knowledge, those equations in \cite{choi2021translating,choisunclassification} are not written in the form \eqref{equ:main} or \eqref{equ:parabolic}.

\bigskip
\begin{remark}\label{rmk:cases}
The main results can be generalized to cover other cases.

\begin{enumerate}[$(1)$]
\item  A solution is called \textit{ancient} if it is defined for $t \in (-\infty, 0]$. The main theorems also provide corresponding results on the next-order asymptotics of ancient solutions that decay to zero as $t$ approaches minus infinity. In the elliptic case \eqref{equ:main}, simply performing a change of variable $t \mapsto -t$ is enough: the equation transforms to another equation of the same form \eqref{equ:main} with a different sign in $m$. For the parabolic case \eqref{equ:parabolic}, however, the same change $t\mapsto -t$ transforms the gradient descending flow to a gradient ascending flow. The finite-dimensional problem  does not distinguish between two: the same result holds after trivial changes in the signs within the statement, such as using positive eigenvalues instead of negative eigenvalues, and non-positive critical values instead of non-negative critical values. For the parabolic PDE \eqref{equ:parabolic}, the difference occurs because of the parabolic regularity (Lemma~\ref{lem:regularity}). Since the parabolic regularity (Lemma~\ref{lem:regularity}) actually holds backward-in-time, with minor modifications, the proofs of Theorems~\ref{thm:general_p} and \ref{thm:general_exponential_p} hold \textit{mutatis mutandis} for ancient solutions.

\item   The theorems are new even for finite-dimensional problems. The main results apply for gradient systems on $\mathbb{R}^n$ by choosing $\Sigma$ to be a point, $d\mu$ a Dirac mass, and $V$ a trivial $\mathbb{R}^n$ bundle. Suppressing $N_1$ and $N_2$, the equations \eqref{equ:parabolic} and \eqref{equ:main} become $\dot x +\nabla f(x)=0$ and $\ddot x -m \dot x -\nabla f(x)=0$ on $x(t)\in \mathbb{R}^n$ for analytic function $f:\mathbb{R}^n\to \mathbb{R}$, respectively. Note the second order system when $m<0$ describes Newton's second law for a unit-mass-particle under a potential force $\nabla f$ and a frictional force proportional to negative velocity, so-called the heavy ball ODE in the optimization and the machine learning literature. Thom's {conjecture} for (finite-dimensional) gradient flow was proved by \cite{KMP}. Our contribution lies in establishing the exact polynomial convergence rate.

\item  Due to volume normalization, the normalized Yamabe flow (NYF) involves a non-local term,  the total scalar curvature, in its speed. Nevertheless, our proof for Theorem~\ref{thm:general_p} can be modified to cover the NYF. Specifically, all parts of the proof, except for Lemma~\ref{lem:error_p}, remain the same. Lemma~\ref{lem:error_p} can be easily established with the explicit forms of the NYF.

 \item We expect that the arguments also extend to certain non-local analytic
functionals, for example when the energy contains a quadratic interaction
\[
    \frac12\iint_{\Sigma\times\Sigma}
   W(x,y) (Au)(x)(Au)(y)\,d\mu(x)d\mu(y),
\]
with \(A\) a linear differential operator, provided that the resulting
gradient map retains the spectral and regularity properties required in our
arguments. We do not pursue this extension here.

\end{enumerate}
\end{remark}}

 Since Simon \cite{S0}, the {\L}ojasiewicz--Simon inequality has been generalized in abstract functional-analytic settings; see \cite{CHILL2003572,MR4129355} and references therein. Although the precise formulations depend on the underlying Banach or Hilbert spaces and on the notion of the gradient map, a common theme is that analyticity of the energy, together with a Fredholm property of the Hessian at the critical point, yields a {\L}ojasiewicz--Simon inequality near that point, and hence convergence of gradient trajectories to a single critical point. It is then natural to ask whether the same kind of functional-analytic assumptions also lead to a proof of Thom's conjecture. The following example suggests that some additional structural assumption would be needed for Thom's conjecture.
\begin{example}[the infinite staircase]
Let $V$ be a Hilbert space with a countable orthonormal basis $\{e_i\}_{i=1}^\infty$, and define $\mathcal{E}:V\to\mathbb{R}$ by
\[
\mathcal{E}\Bigl(\sum_{i=1}^\infty u_i e_i\Bigr)
:= \frac12 \sum_{i=1}^\infty (1+i^{-1})u_i^2.
\]
For such a quadratic $\mathcal{E}$, it is immediate that $\mathcal{E}$ is analytic, that $0$ is a critical point of $\mathcal{E}$, and that the Hessian $\mathcal{E}''(0):V\to V$ is Fredholm. The gradient flow of $\mathcal{E}$, starting from $u_0=\sum_{i=1}^\infty i^{-1} e_i  \in V$ has an explicit representation
\[u(t) = \sum_{i=1}^\infty \frac{1}{i} e^{-(1+i^{-1}) t}e_i .\]
Indeed, $u(t)/\Vert u(t)\Vert$ does not converge in $V$: suppose $u(t)/\Vert u(t)\Vert$ converges to $\sum_{i=1}^\infty \theta_i e_i$ some $\sum_{i=1}^\infty \theta_i^2 =1$. Consider the first nonzero component, say $\theta_{i_0}$. For all $i> i_0$,
\[\frac{\theta_i}{\theta_{i_0}}=\lim_{t\to\infty} \frac{\left\langle u(t),e_i \right\rangle }{\left\langle u(t),e_{i_0} \right\rangle}=\lim_{t\to\infty} \frac{i_0}{i}e^{ (1+i_0^{-1})t -(1+i^{-1}) t}=\infty,\]
which is a contradiction. 

This example shows that analyticity of the energy and Fredholmness of the Hessian do not by themselves guarantee convergence of the secant direction.
\end{example}

\bigskip

We finish this section with two lemmas related to the Lyapunov-Schmidt reduction map $H$ given in Proposition~\ref{pro:AS}. We decompose $u$ as follows: let $u^T=\Pi^T u$ and define $\tilde{u}^\perp$ by
\begin{align}\label{def:ut}
u=u^T+H(u^T)+\tilde{u}^\perp.
\end{align}
Since $u^T\in\ker \mathcal{L}_{\Sigma}$, we may write $u^T$ as a linear combination of $\{\varphi_{\iota+j}\}_{1\leq j\leq  J}$
\begin{equation}\label{def:xj}
u^T= \sum_{j=1}^J x_j\varphi_{\iota+j}.
\end{equation}

\begin{lemma}\label{lem:MSigma}
Let $\rho>0$ be the constant given in Proposition~\ref{pro:AS} and $u\in C^{2,\alpha}(\Sigma;\mathbf{V})$ be a section with $\| u^T \|_{L^2}\leq 2^{-1}\rho$ and $\|\tilde{u}^\perp\|_{C^{2,\alpha}(\Sigma)}\leq M<\infty$. Let $x=(x_1,x_2,\dots,x_J)$ be given by \eqref{def:xj}. Then for any $\alpha\in (0,1)$, there exists a positive constant $C=C(\mathcal{M}_\Sigma,M,\alpha)$ such that
\begin{equation} 
\left| \mathcal{M}_\Sigma (u)+\nabla f(x)-\mathcal{L}_{\Sigma} \tilde{u}^\perp \right|\leq C \Vert u \Vert_{C^{2,\alpha}(\Sigma)}\Vert \tilde{u}^\perp  \Vert_{C^{2,\alpha}(\Sigma)} .
\end{equation} 
\begin{proof}
In the proof we use $C$ to represent a positive constant that depends on $\mathcal{M}_\Sigma$, $M$, $\alpha$, and its value may vary from one line to another.
Let $\bar{\mathcal{L}}_\Sigma$ be the linearization of $\mathcal{M}_\Sigma$ at $u^T+H(u^T)$. From Proposition~\ref{pro:AS} (Lyapunov-Schmidt reduction) and \eqref{def:ut}, $\left| \mathcal{M}_\Sigma (u)+\nabla f(x)-\mathcal{L}_{\Sigma} \tilde{u}^\perp \right|$ becomes
\begin{align*}
 &\left|    \mathcal{M}_{\Sigma}(u^T +H(u^T)+ \tilde u^\perp)-\mathcal{M}_{\Sigma} (u^T +H(u^T)) -\mathcal{L}_{\Sigma} \tilde{u}^\perp\right|\\
 &\leq |\bar{\mathcal{L}}_\Sigma  \tilde{u}^\perp -\mathcal{L}_{\Sigma} \tilde{u}^\perp  |+ |\mathcal{M}_{\Sigma}(u^T +H(u^T)+ \tilde u^\perp)-\mathcal{M}_{\Sigma} (u^T +H(u^T))-\bar{\mathcal{L}}_\Sigma  \tilde{u}^\perp   |. 
\end{align*}
Since $\mathcal{M}_\Sigma: C^{2,\alpha}(\Sigma;\mathbf{V})\to C^{\alpha}(\Sigma;\mathbf{V})$ is analytic, the above is bounded by  
\begin{align*}
C  \Vert u^T+H(u^T)\Vert_{C^{2,\alpha}(\Sigma)}\Vert \tilde u^\perp \Vert_{C^{2,\alpha}(\Sigma)}+C\Vert \tilde u^\perp \Vert^2_{C^{2,\alpha}(\Sigma)}
\end{align*}
From \eqref{def:xj} and the analyticity of $H$,   $\Vert u^T+H(u^T)\Vert_{C^{2,\alpha}(\Sigma)}\leq C |x|\leq C \Vert u \Vert_{C^{2,\alpha}(\Sigma)}.$ 	Together with \eqref{def:ut}, $\Vert \tilde{u}^\perp \Vert_{C^{2,\alpha}(\Sigma)} \leq C  \Vert u \Vert_{C^{2,\alpha}(\Sigma)}.$ Hence the assertion holds.

\end{proof}
\end{lemma}
 
\begin{lemma}[boundedness of decomposition]\label{lem:utup} 
Let $\rho>0$ be the constant from Proposition~\ref{pro:AS} and $u\in C^\infty (Q_{a,b};\widetilde{\mathbf{V}})$ be a smooth section with $\|u\|_{C^s}(t) \le M <\infty $. If $\|u^T\|_{L^2}(t)\leq 2^{-1}\rho$, then there is $C=C(\mathcal{M}_\Sigma, M,s)<\infty $ such that
\begin{align*}
\|u^T\|_{C^s}(t) +\|H(u^T)\|_{C^s}(t) + \|\tilde{u}^\perp\|_{C^s}(t) \le C  \|u\|_{C^s}(t).
\end{align*}
\begin{proof}
In the proof, the constant $C$ may vary from one line to another. Let $x(t)=(x_1(t),\dots, x_J(t))$ be the coefficients given by \eqref{def:xj} for $u^T(t)$. Then
\begin{align*}
\|u^T\|_{C^s}(t)\leq C \sum_{k=0}^s \left| \tfrac{d^k}{dt^k} x(t) \right|\leq C   \|u\|_{C^s}(t).
\end{align*}
By \eqref{equ:identification}, let us view $H$ as a map from an open ball in $\mathbb{R}^J$ to $C^\infty(\Sigma;\mathbf{V})$. We abuse the notation and write $H(x(t))$ for $H(u^T)$. For $1\leq k\leq s$, 
\begin{align*}
\frac{\partial^k}{\partial t^k} H(x(t))=\sum_{i=1}^k\sum_{\substack{k_1+\dots+k_i=k\\k_j\geq 1}}  \frac{k!}{i! k_1! \cdots k_i! } D^iH(x(t)) \left[  \tfrac{d^{k_1} x (t)}{dt^{k_1} },  \dots  ,\tfrac{d^{k_i} x (t)}{dt^{k_i} } \right].
\end{align*}

Let $\mathcal{B}_{i,\ell,\alpha}= \mathcal{L}(\left(\mathbb{R}^J\right)^{\otimes i}, C^{\ell,\alpha}(\Sigma;\mathbf{V}))$ be the Banach space of bounded linear maps from $\left(\mathbb{R}^J\right)^{\otimes i}$ to $ C^{\ell,\alpha}(\Sigma;\mathbf{V})$ equipped with the operator norm. From Proposition~\ref{pro:AS}, $D^iH$ is an analytic map from $B_\rho(0)\subset \mathbb{R}^J$ to $\mathcal{B}_{i,\ell,\alpha}$. In particular, the operator norm of $D^iH$ is bounded in $B_{\rho/2}(0)$. Therefore, provided $|x(t)|\leq \rho/2$,
\begin{align*}
\left| \slashed{\nabla}^\ell\left[ D^iH(x(t))\left [ \tfrac{d^{k_1} x (t)}{dt^{k_1} },  \dots ,  \tfrac{d^{k_i} x (t)}{dt^{k_i} } \right ]\right]\right|\leq C \left|\tfrac{d^{k_1} x (t)}{dt^{k_1} } \right|  \dots  \left| \tfrac{d^{k_i} x (t)}{dt^{k_i} } 	\right| 
 \leq C   \|u\|_{C^s}(t).
\end{align*} 
We used the assumption $\|u\|_{C^s}(t)\le M$ in the second inequality. This ensures $ \|H(u^T)\|_{C^s}(t) \leq C  \|u\|_{C^s}(t).$ Then because of \eqref{def:ut}, $ \|\tilde{u}^\perp\|_{C^s}(t) \leq C \|u\|_{C^s}(t)$ holds.  
\end{proof}
\end{lemma}

\section{First order ODE system} \label{sec:1storderode}
We transform \eqref{equ:main} into a first order ODE system.  Let $u\in C^2(Q_{0,\infty},\widetilde{\mathbf{V}})$ be a solution to \eqref{equ:main}. By setting ${E}_1 (u)=N_1(u)- \mathcal{M}_\Sigma u+\mathcal{L}_{\Sigma} u$, there holds
\bea\label{equ:minimalODE}  {\ddot u}- m  {\dot u}+   \mathcal{L}_{\Sigma} u = E_1(u).\eea
From \eqref{equ:quasilinear} and \eqref{equ:N}, the error term $E_1(u)$ has the structure
\begin{equation}\label{equ:Estr}
 {E}_1(u)=a_1\cdot D{\dot u}(t)+a_2\cdot {\dot u}(t)+\sum_{j=0}^2 a_{4,j}\cdot \slashed{\nabla}^j u,
\end{equation} 
where $a_{4,j}=a_{4,j}(\omega, u,\slashed{\nabla}u,{\dot u})$ are smooth with $a_{4,j}(\omega,0, 0,0)=0$. Let us begin to set up some notion.
\begin{definition}\label{def:mathcal}
Let $\mathbf{L}:H^2(\Sigma;\mathbf{V})\times H^1(\Sigma;\mathbf{V}) \to H^1(\Sigma;\mathbf{V})\times L^2(\Sigma;\mathbf{V})$ be the linear map given by 
\begin{align*}
\mathbf{L}(v,w) := \left( \tfrac{m}{2} v+
w , -\mathcal{L}_{\Sigma}v+\tfrac{m^2}{4} v+\tfrac{m}{2} w
\right).
\end{align*}
\end{definition}
\begin{definition}\label{def:qE}
For a section $u\in C^2(Q_{0,\infty},\widetilde{\mathbf{V}})$, define
\begin{align*}
 q(u):= (u, {\dot u} - \tfrac{m}{2}u),\ \mathcal{E}(u) := (0,E_1(u)). 
\end{align*} 
If there is no confusion, we often suppress the argument and write $q$, $\mathcal{E}$ to denote $q(u)$ and $\mathcal{E}(u)$, respectively. Moreover,  $q(t)$ and $\mathcal{E}(t)$ denote the restriction of $q(u)$ and $\mathcal{E}(u)$ on $\{t\}\times\Sigma$, respectively. We also suppress the argument $t$ if no confusion is made.  
\end{definition}

The equation \eqref{equ:minimalODE} can be rewritten as
\begin{align}\label{eq-vectoru}
 q'=\mathbf{L}q+\mathcal{E} .
\end{align}
We now identify the eigenvalues and eigensections of $\mathbf{L}$. For $i\in\mathbb{N}$, let 
$$\gamma^\pm _i := \tfrac{m}{2} \pm \sqrt {\tfrac{m^2}{4}   -\lambda_i }.$$
Recall the indices $\cup_{i=1}^4 I_i=\mathbb{N} $ in \eqref{def:I123}. For $i\in I_1$, we write the complex eigenvalue $\gamma^\pm_i = \tfrac m2 \pm \mathbf{i} \beta_i$ using  
\[\beta_i = \sqrt {\lambda_i- \tfrac{ m^2 }4}>0.\]
We define the eigenvectors 
\begin{align*}
& \psi_i ^\pm :=   \big ( \tfrac{m}{m-2\gamma^\pm_i }\,\varphi_i ,  -\tfrac{m}{2}  \varphi_i\big ) &&\text{ for }i\in I_3\cup I_4,\\
&\psi_{i,1}:= (0, \tfrac{m}{\sqrt2}\varphi_i),\ \psi_{i,2}:= (\tfrac{m}{\sqrt {2} \beta_i}\varphi_i ,0) &&\text{ for }i\in I_1,\\
&\psi_{i,3}:=(0,\tfrac{m}{\sqrt 2}\varphi_i),\ \psi_{i,4}:=(\tfrac{m}{\sqrt 2}\varphi_i,0) && \text{ for } i \in I_2.
\end{align*} From $\mathcal{L}_{\Sigma}\varphi_i=\lambda_i\varphi_i$, we directly check that 
\begin{equation}\begin{aligned}\label{equ:L1}
&
 \mathbf{L}\psi^{\pm}_i=\gamma^\pm_i\psi^\pm_i &&\text{ for }i\in I_3\cup I_4,\\
&
 \mathbf{L}\psi_{i,1}=\tfrac{m}{2} \psi_{i,1}+\beta_i\psi_{i,2},\ \mathbf{L}\psi_{i,2}=\tfrac{m}{2} \psi_{i,2}-\beta_i\psi_{i,1} &&\text{ for }i\in I_1,\\
&\mathbf{L}\psi_{i,3}=\tfrac{m}{2} \psi_{i,3}+\psi_{i,4},\ \mathbf{L}\psi_{i,4}=\tfrac{m}{2} \psi_{i,4} && \text{ for } i \in I_2.
\end{aligned}\end{equation}
\begin{definition}[inner product]\label{def:G}
Let $G:( H^1(\Sigma;\mathbf{V})\times H^0(\Sigma;\mathbf{V})  )^2\to\mathbb{R}$ be the symmetric bilinear form defined by 
\begin{align*}
&\tfrac{m^2}{2}\,G((v_1,w_1);(v_2,w_2)):=   \int_{\Sigma}\big( -\left\langle \mathcal{L}_{\Sigma} v_1-\tfrac{m^2}{4}v_1 ,v_2 \right\rangle+\left\langle w_1 ,w_2 \right\rangle \big)\, d\mu\\
 &+  \sum_{i\in I_1}  2  \beta^2_i \int_\Sigma \left\langle v_1, \varphi_i \right\rangle d\mu \int_\Sigma \left\langle v_2, \varphi_i \right\rangle d\mu  +\sum_{i\in I_2}\int_\Sigma \left\langle v_1, \varphi_i \right\rangle d\mu \int_\Sigma \left\langle v_2, \varphi_i \right\rangle d\mu . 
\end{align*}, for $v_1,v_2\in H^1(\Sigma;\mathbf{V})$, the term 
\[
-\int_{\Sigma} \left\langle \mathcal{L}_{\Sigma} v_1,v_2 \right\rangle\, d\mu
\]
is understood through integration by parts.

We denote $\|(v,w)\|^2_G:= G((v,w);(v,w))$. The positive-definiteness of $G$ will be justified in Lemma~\ref{lem:G}. $\mathbf{L}^\dagger$ denotes the formal adjoint of $\mathbf{L}$ with respect to $G$:
\begin{align*}
G(\mathbf{L}(v_1,w_1);(v_2,w_2))=G((v_1,w_1);\mathbf{L}^\dagger(v_2,w_2)).
\end{align*}

\end{definition}

We define the collection of vectors 
\begin{equation}\label{def:B}
 \mathcal{B}:= \{\psi_{i,1}, \psi_{i,2}\}_{i\in I_1}\cup\{\psi_{i,3}, \psi_{i,4}\}_{i\in I_2} \cup\{\psi^+_i,\psi^-_i\}_{i\in I_3\cup I_4}. 
\end{equation}
\begin{lemma}[spectral decomposition of  $\mathbf{L}$]\label{lem:G}
There hold

\begin{enumerate}[$(1)$]
\item   The bilinear form $G$ is equivalent to the standard inner product
\begin{equation*}
  \|(v,w)\|_G \asymp  \Vert (v,w)\Vert_{H^1\times H^0}. 
\end{equation*}
\item  The collection $\mathcal{B}$ forms a complete $G$-orthonormal basis.

\item  $\mathbf{L}^\dagger$, the adjoint operator of $\mathbf{L}$ with respect to $G$, acts on $\mathcal{B}$ as follows: 
\begin{align*}
 &\mathbf{L}^\dagger \psi^\pm_i=\gamma^\pm_i\psi^\pm_i  && \text{for }i\in I_3\cup I_4,
\\ 
 & \mathbf{L}^\dagger \psi_{i,1}=\tfrac{m}{2}\psi_{i,1}-\beta_i\psi_{i,2} ,\ \mathbf{L}^\dagger \psi_{i,2}=\tfrac{m}{2}\psi_{i,2}+\beta_i\psi_{i,1} &&\text{for }i\in I_1,
\\ 
& \mathbf{L}^\dagger \psi_{i,3}=\tfrac{m}{2}\psi_{i,3},\ \mathbf{L}^\dagger \psi_{i,4}=\tfrac{m}{2}\psi_{i,4}+\psi_{i,3} &&\text{for }i\in I_2.
\end{align*}

\end{enumerate}

\end{lemma}
\begin{proof}
We start to prove (1). By expressing $v$ as $v=\sum_{i=1}^\infty a_i \varphi_i$, \begin{align*}
G((v,0);(v,0))=2m^{-2}\bigg (\sum_{i\notin I_2 } \left| 4^{-1}m^2-\lambda_i \right| a_i^2+\sum_{i\in I_2 } a_i^2 \bigg ) \geq c \sum_{i=1}^\infty a_i^2=c\|v\|^2_{L^2}, 
\end{align*} 
for some $c=c(\mathcal{M}_\Sigma,m)>0 $. In view of \eqref{equ:H1}, this implies (1). It is straightforward to check that vectors in $\mathcal{B}$ are $G$-orthonormal. Suppose $(v,w)\in H^1(\Sigma;\mathbf{V})\times H^0(\Sigma;\mathbf{V}) $ is $G$-orthogonal to every vector in $\mathcal{B}$. Note that for all $j\in \mathbb{N}$, $(0,\varphi_j)$ and $(\varphi_j,0)$ lie in the linear span of $\mathcal{B}$. This implies that 
$
0=2^{-1}m^{2} G((v,w);(0,\varphi_j))= \int_{\Sigma}\left\langle w,\varphi_j \right\rangle \, d\mu,
$
and that
$
0=2^{-1}m^{2} G((v,w);(\varphi_j,0 ))= C_j  \int_{\Sigma}\left\langle v,\varphi_j \right\rangle \, d\mu.
$
Here $C_j=1$ for $j\in I_2$ and $C_j=\left| 4^{-1}m^2-\lambda_j \right|$ for $j\notin I_2$. Therefore, $v=w=0$ and $\mathcal{B}$ is complete. Lastly, (3) follow from \eqref{equ:L1} and the $G$-orthogonality of $\mathcal{B}$.
\end{proof}

\begin{lemma}[equivalence of $\mathbf{L}$ and angular derivative]\label{lem:Gequ}
For each $\ell\in \mathbb{N}_0$, 
\begin{equation}\label{equ:multiL}
 \sum_{j=0}^\ell \Vert \mathbf{L}^j (v,w) \Vert_{G} \asymp   \Vert (v,w)\Vert_{H^{\ell+1}\times H^\ell}.
\end{equation}
\end{lemma}
\begin{proof}
We use an induction argument. The assertion for $\ell=0$ follows directly from Lemma \ref{lem:G} (1). Now we assume \eqref{equ:multiL} holds for $\ell$ and prove it for $\ell+1$. From the induction hypothesis, 
\begin{align}\label{equ:GappH}
\sum_{j=0}^{\ell+1} \Vert \mathbf{L}^j (v,w) \Vert_{G}\asymp  \Vert  (v,w)  \Vert_{H^1\times H^0}+\Vert   \mathbf{L}(v,w)  \Vert_{H^{\ell+1}\times H^{\ell}}
\end{align}
In view of Definition~\ref{def:mathcal},  $ \Vert   \mathbf{L}(v,w)  \Vert_{H^{\ell+1}\times H^{\ell}}\leq C \Vert (v,w)\Vert_{H^{\ell+2}\times H^{\ell+1}}$, and thus
\begin{equation}\label{equ:GleqH}
\sum_{j=0}^{\ell+1} \Vert \mathbf{L}^j (v,w) \Vert_{G}\leq C \Vert (v,w)\Vert_{H^{\ell+2}\times H^{\ell+1}}.
\end{equation}
To obtain the inequality in the other direction, we use
\begin{align*}
&\Vert w \Vert_{H^{\ell+1}}   \leq C    \Vert  2^{-1}m v+
w  \Vert_{H^{\ell+1}} +  C\Vert v \Vert_{H^{\ell+1}},\\
&\Vert v\Vert_{H^{\ell+2}}    \leq C  \Vert -\mathcal{L}_{\Sigma}v+ 4^{-1}m^2 v+2^{-1} m  w  \Vert_{  H^{\ell}}+ C  \Vert  (v,w) \Vert_{H^{\ell+1}\times H^{\ell}}.
\end{align*} 
Combining Definition~\ref{def:mathcal}, \eqref{equ:GappH} and the induction hypothesis, \begin{equation}\label{equ:GgeqH}
\Vert (v,w)\Vert_{H^{\ell+2}\times H^{\ell+1}}  \leq C   \sum_{j=0}^{\ell+1} \Vert \mathbf{L}^j (v,w) \Vert_{G}.
\end{equation} 
The assertion then follows from \eqref{equ:GleqH} and \eqref{equ:GgeqH}.
\end{proof}
\begin{definition}\label{def:qEkl} 

For given \(k,\ell\in\mathbb{N}_0\) and section \(u\in C^{k+\ell+2}(Q_{0,\infty},\widetilde{\mathbf{V}})\), we define
\[
q^{(k,\ell)}(u):= \partial_t^k\mathbf{L}^\ell q(u),\ \mathcal{E}^{(k,\ell)}(u):= \partial_t^k\mathbf{L}^\ell \mathcal{E}(u). 
\]
Here, \(q(u)\) and \(\mathcal{E}(u)\) are as defined in Definition~\ref{def:qE}. These are often abbreviated as \(q^{(k,\ell)}\) and \(\mathcal{E}^{(k,\ell)}\) respectively. Additionally, \(q^{(k,\ell)}(t)\) and \(\mathcal{E}^{(k,\ell)}(t)\) denote the restrictions of \(q^{(k,\ell)}(u)\) and \(\mathcal{E}^{(k,\ell)}(u)\) to \(\{t\}\times\Sigma\), respectively.
\end{definition}

\begin{corollary}\label{cor:qE} Fix $s\in\mathbb{N}_0$. Then for all $u\in C^{s+2}(Q_{0,\infty},\widetilde{\mathbf{V}})$ and $t\in (0,\infty)$, 
\begin{equation}\label{equ:qkl}
\sum_{k+\ell \le s } \Vert q^{(k,\ell)}(u) \Vert_{G}(t)  \asymp  \sum_{k+\ell \le s+1 }\Vert \partial_t^k \slashed{\nabla}^{\ell  }u\Vert_{L^2}(t).
\end{equation}
Moreover, suppose $u$ satisfies $\|u\|_{C^{s+2}}(t)=o(1)$. Then
\begin{equation}\label{equ:Ekl}
\sum_{k+\ell \le s } \Vert \mathcal{E} ^{(k,\ell)}(u)\Vert_{G}(t)  = o(1)\sum_{k+\ell \le s } \Vert q^{(k,\ell)}(u) \Vert_{G}(t)
\end{equation}
\end{corollary}
\begin{proof}
From Definition~\ref{def:qE} and Lemma~\ref{lem:Gequ},
\begin{align*}
\sum_{k+\ell \le s } \Vert q^{(k,\ell)}(u) \Vert_{G}(t)&\asymp  \sum_{k=0}^s \Vert  (\partial_t^k u,\partial_t^{k+1} u-2^{-1}m \partial_t^k u)  \Vert_{H^{s-k+1}\times H^{s-k} }(t)\\
&\asymp \sum_{k+\ell \le s+1 }\Vert \partial_t^k \slashed{\nabla}^{\ell  }u\Vert_{L^2}(t). 
\end{align*}
This gives \eqref{equ:qkl}. Similarly, from Lemma~\ref{lem:Gequ}  and Definition~\ref{def:qE}, 
\begin{align*}
\sum_{k+\ell \le s } \Vert \mathcal{E}^{(k,\ell)}(u) \Vert_{G}(t)&\asymp  \sum_{k=0}^s \Vert \partial_t^k  E_1(u) \Vert_{ H^{s-k} }(t).  
\end{align*}
From \eqref{equ:Estr} and the assumption, 
\begin{align*}
\sum_{k=0}^s \Vert \partial_t^k  E_1(u) \Vert_{ H^{s-k} }(t)=o(1)\sum_{k+\ell \le s+1 }\Vert \partial_t^k \slashed{\nabla}^{\ell  }u\Vert_{L^2}(t). 
\end{align*}
Then \eqref{equ:Ekl} follows from \eqref{equ:qkl}.
\end{proof}

Let us project the equation \eqref{eq-vectoru} onto vectors in $\mathcal{B}$. Let 
\begin{equation}\label{def:coef_exp}
\begin{split}
&\xi_{i,1}(t):= G(q(t),\psi_{i,1}),\ \xi_{i,2}(t):=G(q(t),\psi_{i,2})\ \textup{for}\ i\in I_1,\\ 
&\xi_{i,3}(t):= G(q(t),\psi_{i,3}),\ \xi_{i,4}(t):=G(q(t),\psi_{i,4})\ \textup{for}\ i\in I_2,\\ 
&\xi^\pm_{i}(t):= G(q(t),\psi^\pm_{i})\ \textup{for}\ i\in I_3\cup I_4.
\end{split}
\end{equation}
Also, let
\begin{equation}\label{def:error_exp}
\begin{split}
&\mathcal{E}_{i,1}(t):= G(\mathcal{E}(t),\psi_{i,1}),\ \mathcal{E}_{i,2}(t):= G(\mathcal{E}(t),\psi_{i,2})\ \textup{for}\ i\in I_1,\\
&\mathcal{E}_{i,3}(t):= G(\mathcal{E}(t),\psi_{i,3}),\ \mathcal{E}_{i,4}(t):= G(\mathcal{E}(t),\psi_{i,4})\ \textup{for}\ i\in I_2,\\
&\mathcal{E}^\pm_{i}(t):= G(\mathcal{E}(t),\psi^\pm_{i})\ \textup{for}\ i\in I_3\cup I_4.
\end{split}
\end{equation}
The equation \eqref{eq-vectoru} becomes the following system: 
\begin{align}\label{equ:odegroup1_exp}
&\tfrac{d}{dt} \xi_{i,1}-\tfrac m2  \xi_{i,1}+\beta_i\xi_{i,2}=\mathcal{E}_{i,1}, && \tfrac{d}{dt} \xi_{i,2}-\tfrac m2 \xi_{i,2}-\beta_i\xi_{i,1}=\mathcal{E}_{i,2}, 
\\ \label{equ:odegroupJ_exp}
&\tfrac{d}{dt} \xi_{i,3}-\tfrac m2 \xi_{i,3} =\mathcal{E}_{i,3}, && 
\tfrac{d}{dt} \xi_{i,4}-\tfrac m2\xi_{i,4}- \xi_{i,3}=\mathcal{E}_{i,4},
\\
\label{equ:odegroup_exp}
&\tfrac{d}{dt} \xi^\pm_{i}-\gamma^\pm_i \xi_{i }^\pm  =\mathcal{E}^\pm_{i },
\end{align}
where the index $i$ runs $i\in I_1$ for \eqref{equ:odegroup1_exp}, $i\in I_2$ for \eqref{equ:odegroupJ_exp}, $i\in I_3\cup I_4$ for \eqref{equ:odegroup_exp}, respectively. 

\section{Fast decaying solutions}\label{sec:exp}
In this section, we prove the asymptotics of exponentially decaying solutions for elliptic problems as in Theorem~\ref{thm:general_exponential_e}. The proof for the parabolic problems in  Theorem~\ref{thm:general_exponential_p} follows {\it mutatis mutandis} after mostly simple notational changes. (See Remark \ref{remark:parabolicfast}). For the same reason, we assume $m<0$  throughout this section. The case $m>0$ is simpler as the alternatives (2),(3) in Theorem~\ref{thm:general_exponential_e} do not appear. We begin with a unique continuation property at infinity. Though we closely follow the argument in \cite{Strehlke}, we include the proof for readers' convenience.  
{\begin{proposition}\label{pro:infinitydecay}
Let $u\in C^\infty(Q_{0,\infty},\widetilde{ \mathbf{V}})$ be a solution to \eqref{equ:main} that satisfies $\|u\|_{C^1}(t)=O(e^{ \gamma t})$ as $t\to \infty$ for all $\gamma<0$. Then $u\equiv 0$.
\end{proposition}
\begin{proof}
By the elliptic regularity, Lemma~\ref{lem:regularity_e},  $\|u\|_{C^s}(t)=O(e^{\gamma t} )$ for all $s\in\mathbb{N}$ and $\gamma <0$. Let $q=q(u)$ be given in Definition~\ref{def:qE}. Take $\gamma<2^{-1}m-1$. Let $\Pi_\gamma q$ be the projection of $q$ onto the eigenspace of $\mathbf{L}$ whose eigenvalues are less than or equal to  $\gamma $. In terms of the coefficients introduced in \eqref{def:coef_exp}, 
\begin{align*}
(\Pi_\gamma q)(t)=\sum_{i:\gamma^-_i \le  \gamma} \xi_i^-(t)\psi_i^-.
\end{align*}
We claim that for any $t_2 \ge t_1 $, 
\bea \label{eq-mainprop41} e^{-\gamma (t_2 -t_1 )} \Vert  q(t_2) \Vert_G \le \Vert\Pi_\gamma   q(t_1 )\Vert_G + \int_{t_1 }^\infty e^{ -\gamma (s-t_1 )} \Vert \mathcal{E}(s)\Vert_G \, ds.  \eea Suppose \eqref{eq-mainprop41} is true at the moment. From \eqref{equ:Estr}, there exists a uniform constant $C<\infty$ such that if $\Vert u\Vert_{C^1}(t)\le 1$, then there holds
\begin{align*}
\Vert \mathcal{E}(t)\Vert _{G } \le C  \Vert q(t)\Vert _{G}\Vert u\Vert _{C^2}(t).  
\end{align*}
Define $M_\gamma(t)=\sup_{\tau \geq t} e^{-\gamma(\tau-t)} \Vert q(\tau) \Vert_G $. Suppose $t_1$ is large enough such that $\Vert u\Vert _{C^{100}}(s )\le 1$ for $s\ge t_1$. Then from \eqref{eq-mainprop41}, 
\begin{align*}
M_\gamma(t_1) \le \Vert \Pi_\gamma q(t_1)\Vert_G + C  \int_{t_1}^\infty  \Vert u\Vert_{C^2}(s) \, ds \cdot M_\gamma(t_1).
\end{align*}
From the exponential decay assumption we may choose large $t_1$ so that $C \int _{t_1}^\infty \Vert u\Vert_{C^2}(s) ds\le 1/2$. Hence
\begin{align}\label{eq-4777}
M_\gamma(t_1)  \le 2\Vert \Pi_\gamma  q(t_1)\Vert_G.
\end{align}
It is clear that $M_\gamma(t_1)$ is non-increasing in $\gamma$. Therefore,  
\begin{align*}
M_\gamma(t_1)\leq \limsup_{\gamma'\to -\infty } M_{\gamma'}(t_1)\leq 2\limsup_{\gamma'\to -\infty }  \Vert \Pi_{\gamma'}  q(t_1)\Vert_G=0.
\end{align*} 
We conclude $q(t)=0$ (and thus $u(t)=0$) for $t\ge t_1$. Next, we show $q(t)=0$ up to $t=0$. Suppose on the contrary $t_0=\inf \{ t_1 \, :\, q(t)=0 \text{ for } t\ge t_1\}>0$. By the smoothness of $u(\omega,t)$, we may find a small $\e>0$ such that $\Vert u\Vert_{C^1}(t) \le 1$ for $t\in [t_0-\e, t_0]$ and 
\[C\int_{t_0-\e}^\infty \Vert u\Vert _{C^2}(s)\, ds =C\int_{t_0-\e}^{t_0}  \Vert u\Vert _{C^2}(s)\, ds \le \frac12 ,  \]
which implies that \eqref{eq-4777} holds for $t_1=t_0-\e$. This gives a contradiction and proves the statement. 

It remains to show \eqref{eq-mainprop41}. Define non-negative functions $X_\pm(t)$ by
\begin{align*}
X^2_+ = &\sum_{i\in I_1} |\xi_{i,1} |^2+|\xi_{i,2} |^2+\sum_{i\in I_2} |\xi_{i,3} |^2+|\xi_{i,4} |^2 + \sum |\xi^+_i |^2+\sum_{i:\gamma^-_i>\gamma} |\xi^-_i |^2,\\
X^2_- = &\sum_{i:\gamma^-_i\leq \gamma} |\xi^-_i |^2.
\end{align*}
From \eqref{equ:odegroup1_exp}-\eqref{equ:odegroup_exp}, $\frac12 (X_+^2)'$ becomes  
\begin{equation} \label{eq-15661} \begin{aligned} 
 &\frac m2 \bigg (\sum_{i\in I_1} |\xi_{i,1} |^2+|\xi_{i,2} |^2+\sum_{i\in I_2} |\xi_{i,3} |^2+|\xi_{i,4} |^2 \bigg )+ \sum \gamma^+_i |\xi^+_i |^2+\sum_{i:\gamma^-_i>\gamma} \gamma^-_i |\xi^-_i |^2\\
 &+\sum_{i\in I_2}\xi_{i,3} \xi_{i,4} +\sum \xi^+_i \mathcal{E}^+_i +\sum_{i:\gamma^-_i>\gamma} \xi^-_i \mathcal{E}^-_i +\sum_{i\in I_1}  \xi_{i,1} \mathcal{E}_{i,1}  +\xi_{i,2} \mathcal{E}_{i,2} \\
&+\sum_{i\in I_2}  \xi_{i,3} \mathcal{E}_{i,3}  +\xi_{i,4} \mathcal{E}_{i,4} .\end{aligned}
\end{equation}
From $\left| \sum_{i\in I_2}\xi_{i,3} \xi_{i,4}  \right|\leq  2^{-1}\sum_{i\in I_2}|\xi_{i,3} |^2+|\xi_{i,4} |^2$ and $\gamma<2^{-1}m-1$, 
\begin{align}
 \tfrac12 (X_+^2)' \geq \gamma {X_+^2}  +Z_+ ,
\end{align}
where \begin{align*}
Z_+: =&\sum \xi^+_i \mathcal{E}^+_i +\sum_{i:\gamma^-_i>\gamma} \xi^-_i \mathcal{E}^-_i +\sum_{i\in I_1}  \xi_{i,1} \mathcal{E}_{i,1}  +\xi_{i,2} \mathcal{E}_{i,2} +\sum_{i\in I_2}  \xi_{i,3} \mathcal{E}_{i,3}  +\xi_{i,4} \mathcal{E}_{i,4}.
\end{align*}
 Let us define $Y_+(t)$ by 
 \[Y_+(t):= \frac{Z_+(t)}{X_+(t)} 1_{\{X_+>0\}}(t).\]
 Since $|Z_+| \le |X_+|\Vert \mathcal{E}\Vert _{G}$, $Y_+$ is a bounded measurable function for $t\ge t_1$. 
By Lemma~\ref{lem:sqrt-lipschitz},   
\begin{equation}\label{equ:42} X_+' \ge \gamma X_+ + Y_+ \end{equation}
in the integral sense.
  Similarly, if we define 
\[Z_-:=\sum_{i:\gamma^-_i \le \gamma} \xi^-_i \mathcal{E}^-_i  ,\quad  \text{ and } \quad Y_-(t):= \frac{Z_-(t)}{X_-(t)} 1_{\{X_->0\}}(t),\] then, in view of Lemma~\ref{lem:sqrt-lipschitz},   
\begin{align}\label{equ:43}
X_-' \leq \gamma {X_-} +Y_-
\end{align}
in the integral sense.
For $t_1 \le t_2$, integrating \eqref{equ:42} and \eqref{equ:43},   
\begin{align}\label{equ:45}
e^{-\gamma (t_2-t_1)}X_+(t_2)\leq & \int_{t_1}^{\infty} e^{-\gamma (s-t_1)} |Y_+(s)|\, ds, \text{ and} \\ 
\label{equ:44}
e^{-\gamma (t_2-t_1)}X_-(t_2)\leq &   X_-(t_1)+\int_{t_1}^{\infty} e^{-\gamma (s-t_1)} |Y_-(s)|\, ds, \text{ respectively}.
\end{align} 
In \eqref{equ:45}, the decay assumption $\lim_{s\to \infty}e^{-\gamma (s-t_1)}X_+(s)=0$ was used.
Now \eqref{eq-mainprop41} follows as $X_+^2  +X_-^2 =\Vert q  \Vert^2_G $, $X_- =\Vert \Pi_\gamma q  \Vert_G$ and $ Y_+^2 +Y_-^2 \leq  \Vert \mathcal{E}  \Vert^2_G $.
\end{proof}

\begin{remark} \label{remark:parabolicfast}
A similar argument applies for the parabolic equation. The only difference is that to control the error term $E_2(u)$ in \eqref{equ:linear_p}, one needs to differentiate the equation. See Lemma~\ref{lem:error_p}.
\end{remark}
} 
Let $u$ be an exponentially decaying solution so that  $\|u\|_{C^1}(t)=O(e^{-2\varepsilon_0 t})$ for some $\varepsilon_0>0$. We further assume $u$ is not identically zero. Define 
\begin{align*}
\gamma_*:=\inf \{\gamma<0\, :\, \|u\|_{C^1}(t)=O(e^{ \gamma t})\}.
\end{align*}
 Proposition~\ref{pro:infinitydecay} shows $\gamma_*$ is a finite negative number. From the elliptic regularity, Lemma~\ref{lem:regularity_e}, $\|u\|_{C^s}(t)=O(e^{(\gamma_*+\varepsilon) t})$ for all $s\in\mathbb{N}$ and $\varepsilon>0$. The quadratic nature of $E_1(u)$ (see \eqref{equ:Estr}), in particular, implies  
\begin{equation}\label{equ:error_exp}
\begin{split} \Vert(0, E_1(u))\Vert^2_{G}  
=O(e^{2(\gamma_*-\varepsilon_0)t}).
\end{split}
\end{equation} Here we used $|E_1(u)| \le C e^{-2\varepsilon_0t} e^{(\gamma_*+\varepsilon_0)t} = Ce^{(\gamma_*-\varepsilon_0)t}$.

\begin{lemma}\label{lem:gamma_star}
There holds $\gamma_*\in \{\gamma^+_i,\gamma^-_i\}_{i\in I_3\cup I_4}\cup\{\tfrac m2 \}$.
\end{lemma}
\begin{proof}
Suppose the assertion fails. This implies there exists $\varepsilon_1\in (0,\varepsilon_0)$ such that there is no element of $\{\gamma^+_i,\gamma^-_i\}_{i\in I_3\cup I_4}\cup\{2^{-1}m\}$ in the interval $[\gamma_*-2\varepsilon_1,\gamma_*+2\varepsilon_1]$. We show this leads to a contradiction for the case $2^{-1}m< \gamma_*$. The argument for the case $\gamma_*<2^{-1}m$ is similar. Define a non-negative function $X_+(t)$ by
\begin{align*}
 X^2_+  =\sum_{i:\gamma^+_i>\gamma_*} |\xi_i^+ |^2+\sum_{i:\gamma^-_i>\gamma_*} |\xi_i^- |^2.
\end{align*}
Set  $Y_+(t_0)=0$ if $X_+(t_0)=0$, and otherwise define $Y_+(t)$ by $
X_+ Y_+  = \sum_{i:\gamma^+_i>\gamma_*}  \xi_i^+ \mathcal{E}^+_i  +\sum_{i:\gamma^-_i>\gamma_*} \xi_i^- \mathcal{E}^-_i .$
From \eqref{equ:odegroup_exp},

\begin{align*}
\tfrac12( X^2_+)'  
 \geq &(\gamma_*+\varepsilon_1) X^2_+   +X_+ Y_+ .
\end{align*}
Proceeding as in \eqref{eq-15661}--\eqref{equ:42} and applying Lemma~\ref{lem:sqrt-lipschitz}, we have
\begin{align*}
X_+'\geq (\gamma_*+\varepsilon_1)X_++Y_+,
\end{align*} in the integral sense.

By the Cauchy-Schwarz and \eqref{equ:error_exp}, $|Y_+(t)|=O(e^{(\gamma_*-\varepsilon_0)t})$. Utilizing $\lim_{t\to\infty} e^{-(\gamma_*+\varepsilon_1)t}X_+(t)=0$, we integrate $\frac{d}{dt}\left(e^{-(\gamma_*+\varepsilon_1)t}X_+(t)  \right) \geq  e^{-(\gamma_*+\varepsilon_1)t}Y_+(t)$ and obtain 
\begin{align}\label{equ:X+}
X_+(t)\leq e^{(\gamma_*+\varepsilon_1)t}\int_t^\infty e^{-(\gamma_*+\varepsilon_1)\tau }|Y_+(\tau)|\, d\tau=O(e^{(\gamma_*-\varepsilon_0)t}).
\end{align}
Let us define $X_-(t)$ by 
\begin{align*}
 X^2_-  =\sum_{i:\gamma^+_i<\gamma_*} |\xi_i^+ |^2+\sum_{i:\gamma^-_i<\gamma_*} |\xi_i^- |^2+\sum_{i\in I_1} |\xi_{i,1} |^2+|\xi_{i,2} |^2+\sum_{i\in I_2} |\xi_{i,3} |^2+\varepsilon_1^2 |\xi_{i,4} |^2.
\end{align*}
From \eqref{equ:odegroup1_exp}-\eqref{equ:odegroup_exp},  $\tfrac 12 (X^2_-)' $ becomes
\begin{align*}
&\sum_{i:\gamma^+_i<\gamma_*} \gamma^+_i|\xi_i^+ |^2+\sum_{i:\gamma^-_i<\gamma_*} \gamma^-_i|\xi_i^- |^2+2^{-1}m\bigg(\sum_{i\in I_1}|\xi_{i,1} |^2+|\xi_{i,2} |^2+\sum_{i\in I_2} |\xi_{i,3} |^2+\varepsilon_1^2 |\xi_{i,4} |^2\bigg) \\
 &+\varepsilon_1^2 \sum_{i\in I_2}\xi_{i,3} \xi_{i,4} + \sum_{i:\gamma^+_i<\gamma_*}  \xi_i^+ \mathcal{E}^+_i  +\sum_{i:\gamma^-_i<\gamma_*} \xi_i^- \mathcal{E}^-_i +\sum_{i\in I_1} \xi_{i,1} \mathcal{E}_{i,1}  + \xi_{i,2} \mathcal{E}_{i,2} +\sum_{i\in I_2} \xi_{i,3} \mathcal{E}_{i,3}  + \varepsilon_1^2 \xi_{i,4} \mathcal{E}_{i,4} .
\end{align*}
From $\varepsilon_1^2 \sum_{i\in I_2}|\xi_{i,3} \xi_{i,4} |\leq 2^{-1}\varepsilon_1  \sum_{i\in I_2}|\xi_{i,3} |^2+\varepsilon^2_1|\xi_{i,4} |^2$,
 
\begin{align*}
 \frac{1}{2}(X^2_-)' \leq (\gamma_*-\varepsilon_1) X^2_-   +X_- Y_- .
\end{align*}
Here $Y_-(t_0)=0$ if $X_-(t_0)=0$ and is otherwise given by
\begin{align*}
X_- Y_- =&\sum_{i:\gamma^+_i<\gamma_*}  \xi_i^+ \mathcal{E}^+_i  +\sum_{i:\gamma^-_i<\gamma_*} \xi_i^- \mathcal{E}^-_i +\sum_{i\in I_1} \xi_{i,1} \mathcal{E}_{i,1}  + \xi_{i,2} \mathcal{E}_{i,2} + \sum_{i\in I_2} \xi_{i,3} \mathcal{E}_{i,3}  + \varepsilon_1^2 \xi_{i,4} \mathcal{E}_{i,4} .
\end{align*} 
Proceeding as in \eqref{eq-15661}--\eqref{equ:42} and applying Lemma~\ref{lem:sqrt-lipschitz}, we have \begin{align*}
X_-'\leq (\gamma_*-\varepsilon_1)X_-+Y_-
\end{align*}
in the integral sense. By the Cauchy-Schwarz inequality and \eqref{equ:error_exp},  $|Y_-(t)|=O(e^{(\gamma_*-\varepsilon_0)t})$. Integrating 
$$\frac{d}{dt}\left(e^{-(\gamma_*-\varepsilon_1)t}X_-(t)  \right) \leq  e^{-(\gamma_*-\varepsilon_1)t}Y_-(t),$$ 
\begin{align*}
X_-(t)\leq e^{(\gamma_*-\varepsilon_1)t}\bigg( X_-(0)+ \int_0^\infty e^{-(\gamma_*-\varepsilon_1)\tau} |Y_-(\tau)| d\tau\bigg )=O(e^{(\gamma_*-\varepsilon_1)t}).
\end{align*}
Combining \eqref{equ:X+} and the above,  $\Vert q(t) \Vert_G=O(e^{(\gamma_*-\varepsilon_1)t})$. From the elliptic regularity, Lemma~\ref{lem:regularity_e}, we then have $\Vert u \Vert_{C^1}(t)=O(e^{(\gamma_*-\varepsilon_1)t})$. This contradicts the definition of $\gamma_*$. 
\end{proof}

\begin{lemma}\label{lem:gamma_star_m}
Suppose $\gamma_*=2^{-1}m$. Then there exist $d_i \in\mathbb{C} $ for $i\in I_1$ and $c_{i,3}, c_{i,4}\in\mathbb{R} $ for $i\in I_2$ such that the following holds. For
\begin{align*}
\hat{q}(t):= q(t) - e^{\gamma_* t} \bigg(\sum_{i\in I_1 } \textup{Re}\big( d_i e^{\mathbf{i}\beta_i t}  \big) \psi_{i,1}+\textup{Im}\big( d_i e^{\mathbf{i}\beta_i t}  \big) \psi_{i,2} + \sum_{i\in I_2 } c_{i,3}\psi_{i,3}+(tc_{i,3}+c_{i,4})\psi_{i,4}\bigg)    
\end{align*}
there exists $\varepsilon >0$ such that $\Vert \hat{q}(t) \Vert_G= O(e^{(\gamma_*-\varepsilon)t})$.
\end{lemma}
\begin{proof}
Define non-negative functions $X_\pm(t)$ by
\begin{align*}
X^2_+=\sum_{i:\gamma^+_i>\gamma_*} |\xi_i^+|^2+\sum_{i:\gamma^-_i>\gamma_*} |\xi_i^-|^2,  \quad X^2_-=\sum_{i:\gamma^+_i<\gamma_*} |\xi_i^+|^2+\sum_{i:\gamma^-_i<\gamma_*} |\xi_i^-|^2.
\end{align*}
By shrinking the value of $\varepsilon_0$ if necessary, we may assume $\varepsilon_0<|2^{-1}m-\gamma^\pm_i|$ for all $i\in I_3\cup I_4$. An argument similar to the one in the proof of Lemma~\ref{lem:gamma_star} shows $X_\pm(t)=O(e^{(\gamma_*-\varepsilon_0)t})$. From \eqref{equ:odegroup1_exp}, we derive for $i\in I_1$
\begin{align*}
\frac{d}{dt}\left( e^{-(\gamma_* +\mathbf{i}\beta_i)t} \left(\xi_{i,1}(t)+\mathbf{i}\xi_{i,2}(t) \right)\right)=  e^{-(\gamma_* +\mathbf{i}\beta_i)t} \left(\mathcal{E}_{i,1}(t)+\mathbf{i}\mathcal{E}_{i,2}(t) \right) .
\end{align*}
Integrating the above from $0$ to $t$ and using \eqref{equ:error_exp} yield 
\begin{align*}
\xi_{i,1}(t)+\mathbf{i}\xi_{i,2}(t)= d_i e^{(\gamma_*+\mathbf{i}\beta_i)t}+O(e^{(\gamma_*-\varepsilon_0)t})
\end{align*}
for some $d_i\in\mathbb{C}$. A similar argument on $\frac{d}{dt}\left( e^{-\gamma_* t} \xi_{i,3}(t)\right)= e^{-\gamma_* t} \mathcal{E}_{i,3}(t)$ gives 
$$\xi_{i,3}(t)= c_{i,3} e^{\gamma_* t}+O(e^{(\gamma_*-\varepsilon_0)t})$$ 
for some $c_{i,3}\in\mathbb{R}$. Lastly, from  
\begin{align*}
\frac{d}{dt}\left( e^{-\gamma_* t} \xi_{i,4}(t)\right) =& e^{-\gamma_* t} \xi_{i,3}(t) +e^{-\gamma_* t} \mathcal{E}_{i,4}(t)=  c_{i,3}+O(e^{  -\varepsilon_0t}), 
\end{align*}
we derive $\xi_{i,4}(t)=(tc_{i,3}+c_{i,4})e^{\gamma_*t}+O(e^{(\gamma_*-\varepsilon_0)t})$ for some $c_{i,4}\in\mathbb{R}$.
\end{proof}

For $\gamma_*=\gamma^+_i$ or $\gamma^-_i$, an analogous result holds. We omit the proof because it is similar to and simpler than the one for Lemma~\ref{lem:gamma_star_m}.  
\begin{lemma}\label{lem:gamma_star_pm}
Suppose $\gamma_*=\gamma^+_i$ for some $i\in I_3\cup I_4$ and let $N$ be the multiplicity of $\lambda_i$. We may assume $\lambda_{i}=\lambda_{i+1}=\dots =\lambda_{i+N-1}$. Then there exists $a_1,a_2,\dots, a_N\in\mathbb{R}$ and $\varepsilon>0$ such that
\begin{align*}
\Vert q(t)- e^{\gamma_* t} \sum_{j=1}^N a_j \psi^+_{i+j-1} \Vert_G=O(e^{(\gamma_*-\varepsilon) t}).
\end{align*}
Similar result holds when $\gamma_*=\gamma^-_i$ for some $i\in I_3\cup I_4$. 
\end{lemma}
We are ready to prove Theorem~\ref{thm:general_exponential_e}.
\begin{proof}[Proof of Theorem~\ref{thm:general_exponential_e}]
Suppose $\gamma_*=\gamma^+_i$ or $\gamma^-_i$. From Lemma~\ref{lem:gamma_star_pm} and Corollary~\ref{cor:qE}, there exists $v$ with $\mathcal{L}_{\Sigma} v=\lambda_i v$ and $\varepsilon>0$ such that $
\Vert u(t)- e^{\gamma_* t}v \Vert_{H^1}= O(e^{(\gamma_*-\varepsilon ) t})$. Let $w=v/\Vert v\Vert_{L^2}$. Then
\begin{align*}
\lim_{t\to\infty} u(t)/\Vert u(t) \Vert_{L^2}=w\ \textup{in}\ H^1(\Sigma,\mathbf{V}).
\end{align*}
We may upgrade the convergence to $C^\infty(\Sigma,\mathbf{V})$ by taking $\mathbf{L}$ derivatives to \eqref{eq-vectoru}. As a result, case (1) in Theorem~\ref{thm:general_exponential_e} holds. Suppose $\gamma_*=2^{-1}m$. From Lemma~\ref{lem:gamma_star_m} and Corollary~\ref{cor:qE}, there exist $d_i\in\mathbb{C}$  and $c_{i,3}, c_{i,4}\in \mathbb{R}$ such that \begin{align*}
 \left\Vert  u(t) - e^{\gamma_*t}\sum_{i\in I_1}\textup{Re} \left(d_ie^{\mathbf{i}\beta_i t}  \right)\varphi_i- e^{\gamma_*t}\sum_{i\in I_2}(tc_{i,3}+c_{i,4})  \varphi_i\right\Vert_{H^1} = O(e^{(\gamma_*-\varepsilon ) t}),
\end{align*}
and this proves (2) and (3) depending on if $c_{i,3}=0$ for all $i\in I_2$ or not.  
\end{proof}

\section{Slowly decaying solutions to elliptic equation}\label{sec:elliptic}

In this section, we show that if a solution to \eqref{equ:main} decays slowly, then the neutral mode, the projection of $q(u)$ onto the $0$-eigenspace of $\mathbf{L}$, dominates the solution. Moreover, the neutral mode evolves by a gradient flow up to a small error. That is the content of Proposition~\ref{prop-neutral-dynamics}. 
\vspace{0.1cm}

Let $u\in C^\infty(Q_{0,\infty},\widetilde{ \mathbf{V}})$ be a solution to \eqref{equ:main} with $\|u\|_{C^1}(t)=o(1)$ as $t\to \infty$. From the elliptic regularity, Lemma~\ref{lem:regularity_e}, $\|u\|_{C^s}(t)=o(1)$ for all $s\in\mathbb{N}$. We further assume that $u$ does not decay exponentially. Namely, for any $\varepsilon>0$, 
\begin{equation}\label{equ:nonexpdecay}
\limsup_{t\to\infty} e^{\varepsilon t}\|u\|_{C^1}(t)=\infty.
\end{equation}

Recall that we rewrote \eqref{equ:main} as an ODE system \eqref{equ:odegroup1_exp}-\eqref{equ:odegroup_exp}. For brevity, we assume throughout this section that $I_2=\varnothing$. With notational changes, the proof can be readily extended to cover the case where $I_2\neq \varnothing$. Since $I_2=\varnothing,$ the ODE system consists of \eqref{equ:odegroup1_exp} and \eqref{equ:odegroup_exp}. It is convenient to relabel the coefficients $\{\xi^\pm_i\}_{i\in I_3\cup I_4}$ in \eqref{equ:odegroup_exp}. For $\{\psi^\pm_i\}_{i\in I_4}$, we set 
\begin{align*}
& \{\Psi_i\}_{i\in\mathbb{N}}= \{\psi^+_i\, :\, i\in I_4\ \textup{and}\ \gamma^{+}_i>0\}\cup \{\psi^-_i\, :\, i\in I_4\ \textup{and}\ \gamma^{-}_i>0\},\\
 &  \{\Psi_i\}_{i\in-\mathbb{N}}=\{\psi^+_i\, :\, i\in I_4\ \textup{and}\ \gamma^{+}_i<0\}\cup \{\psi^-_i\, :\, i\in I_4\ \textup{and}\ \gamma^{-}_i<0\},
\end{align*}
and define $\Gamma_i$, a relabelling of $\gamma^\pm_i $ for $i\in I_4$, by
\begin{align}\label{equ:Psi}
\mathbf{L}\Psi_i=\mathbf{L}^\dagger\Psi_i=\Gamma_i\Psi_i. 
\end{align}
Recall that $I_3=\{\iota+1,\dots, \iota+J\}$. If $m>0$, we set, for $1\leq j\leq J$, $ \Upsilon_j =\psi^-_{\iota+j} $ and $ \overline{\Upsilon}_j = \psi_{\iota+j}^+$. If $m<0$, we set $ \Upsilon_j =\psi^+_{\iota+j} $ and $ \overline{\Upsilon}_j = \psi_{\iota+j}^-$. This arrangement ensures that
\begin{equation}\label{equ:Ups}
\begin{split}
\mathbf{L}\Upsilon_j=\mathbf{L}^\dagger\Upsilon_j=0,\ \mathbf{L}\overline{\Upsilon}_j=\mathbf{L}^\dagger\overline{\Upsilon}_j=m\overline{\Upsilon}_j.  
\end{split}
\end{equation}
Let
\begin{equation}\label{def:coef}
\begin{split}
&\xi_{i}(t):= G(q(t),\Psi_{i})\ \textup{for}\ i\in \mathbb{Z}\setminus\{0\},\\
&z_j(t):= G(q(t),\Upsilon_j),\ \bar{z}_j(t):=G(q(t),\overline{\Upsilon}_j) \ \textup{for}\ 1\leq j\leq J,
\end{split}
\end{equation}
and
\begin{equation}\label{def:error}
\begin{split}
&\mathcal{E}_{i}(t):= G(\mathcal{E}(t),\Psi_{i})\ \textup{for}\ i\in \mathbb{Z}\setminus\{0\},\\
&\mathcal{W}_j(t):= G(\mathcal{E}(t),\Upsilon_j), \ \overline{\mathcal{W}}_j(t):=G(\mathcal{E}(t),\overline{\Upsilon}_j) \ \textup{for}\ 1\leq j\leq J.
\end{split}
\end{equation}
We rewrite \eqref{equ:odegroup_exp} as \eqref{equ:odegroup2} and \eqref{equ:odegroup3} below. For $i\in\mathbb{Z}\setminus\{0\}$,
\begin{equation}\label{equ:odegroup2}
\frac{d}{dt} \xi_{i}-\Gamma_i \xi_{i } =\mathcal{E}_{i }, 
\end{equation}
and for $1\leq j\leq J$,
\begin{equation}\label{equ:odegroup3}
\begin{split}
&\frac{d}{dt} z_j =\mathcal{W}_j,\ \frac{d}{dt} \bar{z}_j-m \bar{z}_j =\overline{\mathcal{W}}_j.
\end{split}
\end{equation} 

We denote $z(t):=(z_1(t),\dots z_{J}(t))$ and $\bar{z}(t):=(\bar{z}_1(t),\dots \bar{z}_{J}(t))$, and use $|z(t)|$ and $|\bar{z}(t)|$ to denote their Euclidean norms respectively. Recall the reduced functional $f$ is introduced in Proposition~\ref{pro:AS}. The main result of this section proves that $z(t)$ evolves by the gradient flow of potential function $m^{-1}f(z)$ modulo a small error, which will be shown to be negligible in the proof of Thom's gradient conjecture in Section~\ref{sec:7}.
\begin{proposition}\label{prop-neutral-dynamics}  {
Let $u$ be a slowly decaying solution to \eqref{equ:main}.
\begin{enumerate}[$(1)$]
\item  For given $\rho \in(0,1)$ and $N\in \mathbb{N}$, there exists  $C=C(\mathcal{M}_\Sigma, N_1,\rho, N )<\infty $ and $t_0=t_0(u,N,\rho  )$ such that, for $t\ge t_0$, 
\begin{align}\label{equ:zgradient}
 |z'(t)+m^{-1} \nabla f(z(t))| \leq C   (|z(t)|^{\rho }|\nabla f(z(t))|+|z(t)|^N ).
\end{align} 
\item  There exist $0<D_1, D_2<\infty$, $\alpha_2 >0$, which are constants of $\mathcal{M}_\Sigma$, $N_1$ such that, for $t\ge t_0=t_0(u)$,   \be \label{eq-algebraicratebound-5} D_1 t^{-1} \le |z(t)|\le D_2 t^{-\alpha_2}.\ee
\end{enumerate}
}
\end{proposition} 

The remainder of this section is dedicated to proving Proposition~\ref{prop-neutral-dynamics}. {Note that the time $t_0$ in the proposition may depend on the solution itself. As the equation and the limit are regarded as given, for simplicity we will omit mentioning the dependency of the constant $C$ on  $\mathcal{M}_\Sigma$ and $N_1$ in the remaining of this section.} The proof involves two main parts. In Section~\ref{sec:5.1}, we demonstrate (as shown in Corollary~\ref{cor:sobolev}) that any $C^s$ norm of $u$ can be bounded by $|z(t)|$. In Section~\ref{sec:5.2}, we obtain an enhanced decay rate in Lemma~\ref{lem-iterate2} through the decomposition \eqref{def:ut}. Proposition~\ref{prop-neutral-dynamics} (1) is then a simple consequence of Lemma~\ref{lem-iterate2}. 

\subsection{Bounding $\Vert u \Vert_{C^s}$}\label{sec:5.1} To estimate $C^s$ norms of $u$, we need higher-derivative versions of the ODE system. Fix $k,\ell\in\mathbb{N}_0$. Recall that $q^{(k,\ell)}$ and $\mathcal{E}^{(k,\ell)}$ are given in Definition~\ref{def:qEkl}. Let 
\begin{equation}\label{def:coefkl} 
\begin{split}
&\xi^{(k,\ell)}_{i,1}(t):= G(q^{(k,\ell)}(t),\psi_{i,1}),\ \xi^{(k,\ell)}_{i,2}(t):=G(q^{(k,\ell)}(t),\psi_{i,2})\ \textup{for}\ i\in I_1,\\ 
&\xi^{(k,\ell)}_{i}(t):= G(q^{(k,\ell)}(t),\Psi_{i})\ \textup{for}\ i\in \mathbb{Z}\setminus\{0\},\\
&z^{(k,\ell)}_j(t):= G(q^{(k,\ell)}(t),\Upsilon_j),\ \bar{z}^{(k,\ell)}_j(t):=G(q^{(k,\ell)}(t),\overline{\Upsilon}_j) \ \textup{for}\ 1\leq j\leq J.
\end{split}
\end{equation}
Also, let
\begin{equation}\label{def:errorkl} 
\begin{split}
&\mathcal{E}^{(k,\ell)}_{i,1}(t):= G(\mathcal{E}^{(k,\ell)}(t),\psi_{i,1}),\ \mathcal{E}^{(k,\ell)}_{i,2}(t):= G(\mathcal{E}^{(k,\ell)}(t),\psi_{i,2})\ \textup{for}\ i\in I_1,\\
&\mathcal{E}^{(k,\ell)}_{i}(t):= G(\mathcal{E}^{(k,\ell)}(t),\Psi_{i})\ \textup{for}\ i\in \mathbb{Z}\setminus\{0\},\\
&\mathcal{W}^{(k,\ell)}_j(t):= G(\mathcal{E}^{(k,\ell)}(t),\Upsilon_j), \ \overline{\mathcal{W}}^{(k,\ell)}_j(t):=G(\mathcal{E}^{(k,\ell)}(t),\overline{\Upsilon}_j) \ \textup{for}\ 1\leq j\leq J.
\end{split}
\end{equation}
Then for $i\in I_1$,
\begin{equation}\label{equ:odegroup1kl} 
\begin{split}
&\frac{d}{dt} \xi^{(k,\ell)}_{i,1}-2^{-1}m \xi^{(k,\ell)}_{i,1}+\beta_i\xi^{(k,\ell)}_{i,2}=\mathcal{E}^{(k,\ell)}_{i,1},\\ 
&\frac{d}{dt} \xi^{(k,\ell)}_{i,2}-2^{-1}m \xi^{(k,\ell)}_{i,2}-\beta_i\xi^{(k,\ell)}_{i,1}=\mathcal{E}^{(k,\ell)}_{i,2}, 
\end{split}
\end{equation}
for $i\in\mathbb{Z}\setminus\{0\}$,
\begin{equation}\label{equ:odegroup2kl} 
\frac{d}{dt} \xi^{(k,\ell)}_{i}-\Gamma_i \xi^{(k,\ell)}_{i } =\mathcal{E}^{(k,\ell)}_{i }, 
\end{equation}
and for $1\leq j\leq J$,
\begin{equation}\label{equ:odegroup3kl} 
\begin{split}
&\frac{d}{dt} z^{(k,\ell)}_j =\mathcal{W}^{(k,\ell)}_j,\ \frac{d}{dt} \bar{z}^{(k,\ell)}_j-m \bar{z}^{(k,\ell)}_j =\overline{\mathcal{W}}^{(k,\ell)}_j.
\end{split}
\end{equation}

In the next lemma, we use the Merle-Zaag ODE lemma \cite{MZ} (see Lemma \ref{lem-MZODE}) to show that $|z(t)|$ dominates the other coefficients.

\begin{lemma}[dominance of neutral mode]\label{lem:ODE}
For any $s\in\mathbb{N}_0$, as $t\to \infty$, 
\begin{align*}
\sum_{k+\ell\leq s}\bigg [ |z^{(k,\ell)} |^2+|\bar{z}^{(k,\ell)} |^2+\sum_{i\in I_1}\left(|\xi^{(k,\ell)}_{i,1} |^2+|\xi^{(k,\ell)}_{i,2} |^2 \right)+\sum_{i\neq 0}|\xi^{(k,\ell)}_i |^2 \bigg ] =(1+o(1))|z |^2.
\end{align*}
\end{lemma}
 
\begin{proof} Fix $s\in\mathbb{N}_0$. We give the proof for the case $m>0$. The argument for $m<0$ is similar. Define three non-negative functions $X_+(t)$, $X_0(t)$ and $X_-(t)$ by
\begin{align*}X^2_+&=\sum_{k+\ell\leq s}\bigg (\sum_{i\in I_1}\left( |\xi^{(k,\ell)}_{i,1}|^2+|\xi^{(k,\ell)}_{i,2}|^2 \right)+\sum_{i\in\mathbb{N}}|\xi^{(k,\ell)}_i|^2+\sum_{1\leq j\leq J}|\bar{z}^{(k,\ell)}_j|^2 \bigg)\\
X^2_0&=\sum_{k+\ell\leq s}\sum_{1\leq j\leq J}|z^{(k,\ell)}_j|^2,\quad  X^2_-=\sum_{k+\ell\leq s}\sum_{i\in -\mathbb{N}}|\xi^{(k,\ell)}_i|^2.
\end{align*}
From \eqref{equ:odegroup1kl}, \eqref{equ:odegroup2kl} and \eqref{equ:odegroup3kl}, $\frac{1}{2}(X^2_+)'$ becomes 
\begin{align*}
&\sum_{k+\ell\leq s}\bigg [ \frac m2  \sum_{i\in I_1}\left(  |\xi^{(k,\ell)}_{i,1} |^2+|\xi^{(k,\ell)}_{i,2} |^2 \right)+\sum_{i\in\mathbb{N}} \Gamma_i |\xi^{(k,\ell)}_i|^2+m\sum_{1\leq j\leq J}|\bar{z}^{(k,\ell)}_j |^2 \\
&\quad +   \sum_{i\in I_1}\left(   \xi^{(k,\ell)}_{i,1}\mathcal{E}^{(k,\ell)}_{i,1}  +\xi^{(k,\ell)}_{i,2}\mathcal{E}^{(k,\ell)}_{i,2}  \right)+\sum_{i\in\mathbb{N}}   \xi^{(k,\ell)}_i\mathcal{E}^{(k,\ell)}_i + \sum_{1\leq j\leq J} \bar{z}^{(k,\ell)}_j \overline{\mathcal{W}}^{(k,\ell)}_j  \bigg ].
\end{align*}
Let us define $Y_+(t)$ by equating the three terms on the second line as $X_+(t)Y_+(t)$ and $Y_+(t_0)=0$ if $X_+(t_0)=0$. Let $b$ be the minimum among $2^{-1}m$ and $|\Gamma_i|,\ i\in\mathbb{Z}\setminus\{0\}$. 
Proceeding as in \eqref{eq-15661}--\eqref{equ:42} and applying Lemma~\ref{lem:sqrt-lipschitz}, we obtain\begin{align*}
X_+'-bX_+\geq Y_+
\end{align*}
in the integral sense.
Similarly, define $Y_0(t)$ and $Y_-(t)$ by
\begin{align*}
X_0Y_0= \sum_{k+\ell\leq s}\sum_{1\leq j\leq J} z^{(k,\ell)}_j\mathcal{W}^{(k,\ell)}_j,\quad 
X_-Y_-= \sum_{k+\ell\leq s}\sum_{i\in-\mathbb{N} } \xi^{(k,\ell)}_i\mathcal{E}^{(k,\ell)}_i,
\end{align*}respectively. If $X_0(t_0)=0$ or $X_-(t_0)=0$, we set $Y_0(t_0)=0$ or $Y_-(t_0)=0$ respectively.
It holds that

\begin{align} \label{eq-neutralstable}
X_0'= Y_0,\quad  X_-'+bX_-\leq Y_-
\end{align}
in the integral sense.

We now compare $|X_+|^2+|X_0|^2+|X_-|^2$ and $|Y_+|^2+|Y_0|^2+|Y_-|^2$. Observe $|X_+|^2+|X_0|^2+|X_-|^2= \sum_{k+\ell\leq s} \| q^{(k,\ell)} \|^2_G$.
From the Cauchy-Schwarz inequality and \eqref{def:errorkl}, $|Y_+|^2+|Y_0|^2+|Y_-|^2$ 
 can be bounded by $\sum_{k+\ell\leq s} \| \mathcal{E}^{(k,\ell)}\|^2_G$.
By \eqref{equ:Ekl}, $|Y_+|^2+|Y_0|^2+|Y_-|^2=  o(1)\big( |X_+|^2+|X_0|^2+|X_-|^2 \big).$ Then we apply Lemma~\ref{lem-MZODE}. In view of \eqref{eq:mz.ode.cor.B1}, the slow decay assumption \eqref{equ:nonexpdecay} rules out the possibility that $X_-(t)$ dominates. Hence  
\begin{equation}\label{equ:ODEmiddle1}
|X_+|^2+|X_0|^2+|X_-|^2=(1+o(1))|X_0|^2. 
\end{equation}
It remains to show that 
\begin{equation}\label{equ:ODEmiddle2}
 |X_0|^2=(1+o(1))|z|^2. 
\end{equation}
For $\ell\geq 1$,
$
z^{(k,\ell)}_j =G(\mathbf{L}q^{k,\ell-1},\Upsilon_j)=G(q^{k,\ell-1},\mathbf{L}^\dagger\Upsilon_j)=0.
$
For $k\geq 1$,
$
z^{(k,\ell)}_j =\frac{d}{dt} z^{(k-1,\ell)}_j=\mathcal{W}^{(k-1,\ell)}_j.
$
These imply
\begin{align*}
 |X_0|^2&\leq |z|^2+\sum_{k\leq s-1 }  |\mathcal{W}^{(k,0)} |^2 \leq |z|^2+o(1)\left(|X_+|^2+|X_0|^2+|X_-|^2\right),
\end{align*}
and \eqref{equ:ODEmiddle2} follows by \eqref{equ:ODEmiddle1}.
\end{proof}

In view of Corollary~\ref{cor:qE} and Lemma~\ref{lem:ODE}, for any $s\in\mathbb{N}_0$,
\begin{align*}
\sum_{k+\ell \le s}\Vert \partial_t^k \slashed{\nabla}^{\ell  }u\Vert_{L^2}(t)=O(1)|z(t)|. 
\end{align*}
Applying the Sobolev embedding, bounds on any $C^s$-norm of $u$ follow.
\begin{corollary}[control on higher derivatives]\label{cor:sobolev} For any $s \in \mathbb{N}$, there exists {$C=C(s)<\infty $} such that $\Vert u \Vert_{C^{s}}(t) \leq C   |z(t) | $ {for large time $t\ge t_0=t_0(s ,u)$.}  
\end{corollary}

\begin{corollary} [algebraic lower bound] \label{cor-Xi^+_0} For slowly decaying solution $u(t)$, there is $C=C(\mathcal{M}_\Sigma, N_1)<\infty$ such that $|z(t)|\ge (Ct)^{-1}$ for $t\ge t_0$.
 
\begin{proof}
The assertion directly follows from $ |z'(t)|  \leq C  |z(t)|^2$ for large time $t\ge t_0$ and $\lim_{t\to\infty}|z(t)|=0$.   

\end{proof}  
\end{corollary}

The next elementary result will be used to show that the dynamics of non-neutral part will be driven by the inhomogeneous source term $\mathcal{E}(t)$. This will be one of key ingredients to obtain improved decay rate in Section~\ref{sec:5.2}.

\begin{lemma}\label{lemma:integrate} Let $\eta(t)$, defined for $t\ge t'$, be a positive continuous function that satisfies 
\begin{enumerate}[$(a)$]
\item  $\eta(t)\to 0 $ as $t\to \infty$, and  
\item for given $\e>0$, $e^{-\e t}\eta(t)$ is non-increasing and $e^{\e t} \eta(t)$ is non-decreasing for large time $t\ge t_0=t_0(\eta,\e)$. 
\end{enumerate}
 Suppose $y_i(t)$ is a sequence of functions that solves 
\begin{align}\label{equ:ODE_y}
y'_i(t)-m_iy_i(t)=R_i(t). 
\end{align}
Here,  $m_i\in\mathbb{R}$ are non-zero constants and  $R_i(t)$ are functions such that $\sum_{i=1}^\infty |R_i(t)|^2\leq  M^2\eta^2(t)$ for some $M>0$. Then the followings hold:
\begin{enumerate}[$(1)$]
\item If $\sup_{i\in\mathbb{N}}m_i=-b<0$ and $\sum_{i=1}^\infty |y_i(t')|^2<\infty$, then there exists  $C<\infty $  such that $\sum_{i=1}^\infty |y_i(t)|^2\leq C M^2 \eta^2(t) $ for large time $t\ge t_1$ . 
\item If $\inf_{i\in\mathbb{N}}m_i=b>0$ and $\lim_{t\to\infty} e^{-m_i t}y_i(t)=0$ for all $i\in\mathbb{N}$, then there exists $C<\infty $ such that $\sum_{i=1}^\infty |y_i(t)|^2\leq CM^2\eta^2(t) $ for large time $t\ge t_1$.
\end{enumerate}Here $C=C(b)$, 
$t_1=t_1(t',\eta,M,\sum_{i\in\mathbb N}|y_i(t')|^2,b)$
in (1), and $t_1=t_1(t',\eta,b)$ in (2). 
\begin{proof} Without loss of generality, we assume $t'=0$ in the proof as in the other case we may replace $\eta(t-t')$ by $\eta(t)$. In the proof $C$ represents a positive constant of stated dependency whose value may vary from one line to another. 

Let us prove (1). From \eqref{equ:ODE_y},
\begin{align*}
y_i(t)=e^{m_i t} y_i(0)+\int_0^t  e^{m_i(t-\tau)} R_i (\tau)\, d\tau.
\end{align*}
For any sequence $\{a_i\}_{i\in \mathbb{N}}$ with $\sum_{i\in \mathbb{N}} |a_i|^2=1$, we estimate
\begin{align*}
\left|\sum_{i\in  \mathbb{N}} y_i(t)a_i\right|
\leq & e^{-bt} \sum_{i\in  \mathbb{N}}\left| y_i(0)a_i\right|+ \sum_{i\in \mathbb{N}}   \int_0^t e^{-b(t-\tau)} \left| R_i (\tau)a_i \right|\, d\tau.
\end{align*}
Observe first that   {$e^{-bt} \sum_{i\in \mathbb{N}}\left| y_i(0)a_i\right|   \le e^{-bt} (\sum_{i\in \mathbb{N}} |y_i(0)|^2)^{\frac12} \le M\eta(t)$ for large $t\ge t_1$.} From the assumption and the Cauchy-Schwarz inequality,
\begin{align*}
\sum_{i\in \mathbb{N}}   \int_0^t e^{-b(t-\tau)} \left| R_i (\tau)a_i \right|\, d\tau  \leq  M\int_0^t e^{-b(t-\tau)}  \eta(\tau)\, d\tau.
\end{align*}
Let $t_0=t_0(\eta,b/2)$ be the large time provided by  (b) when $\eps=\frac{b}2$. For $t\ge t_0$,
\begin{align*}
\int_0^t e^{-b(t-\tau)} \eta(\tau)\, d\tau=&\int_0^{t_0} e^{-b(t-\tau)}\eta(\tau )\, d\tau +\int_{t_0}^t e^{-\frac b2 (t-\tau)} \left(e^{-\frac b2 (t-\tau)} \eta(\tau) \right) \, d\tau.
\end{align*}
We first claim that $ \int_0^{t_0} e^{-b(t-\tau)}\eta(\tau )\, d\tau\le \eta(t)$ for large $t\ge t_1$. Indeed, since $e^{-b(t-\tau)}$ is increasing in $\tau$, and ${e^{\frac b2 (t-t_0)}\eta(t)}\ge {\eta(t_0)}$ for $t\ge t_0$,
\be \label{eq-t5} \int_0^{t_0} e^{-b(t-\tau)}\eta(\tau )\, d\tau\le e^{-\frac b2 (t-t_0)} \frac{t_0 \Vert  \eta\Vert _{L^\infty ([0,\infty))}}{e^{\frac b2 (t-t_0)}\eta(t)}\eta (t)\le e^{-\frac b2 (t-t_0)} \frac{t_0 \Vert  \eta\Vert _{L^\infty ([0,\infty))}}{\eta(t_0)}\eta (t) .\ee Now the claim follows by choosing large $t_1$. Next, 
\[\int_{t_0}^t e^{-2^{-1}b(t-\tau)} \left(e^{-2^{-1}b(t-\tau)} \eta(\tau) \right) \, d\tau \le \eta(t) \int_{t_0}^t e^{-2^{-1}b(t-\tau)}\, d\tau \le \frac{2}{b} \eta(t).\]
Consequently, $\left|\sum_{i\in  \mathbb{N}}  y_i(t)a_i\right| \leq CM \eta(t)$, for $t\ge t_1$ for every sequence $\{a_i\}$ with $\sum_{i\in \mathbb{N}} |a_i|^2=1$. This implies $\sum_{i=1}^\infty |y_i(t)|^2\leq CM^2 \eta^2(t)$.

The proof of (2) is similar but simpler. Indeed, by the condition
$\lim_{s\to\infty} e^{-m_i s}y_i(s)=0$, variation of constants gives
\[
y_i(t)=-\int_t^\infty e^{m_i(t-\tau)}R_i(\tau)\,d\tau.
\]
Hence for every $\{a_i\}\in \ell^2$ with $\sum_i |a_i|^2=1$,
\[
\left|\sum_i y_i(t)a_i\right|
\le M\int_t^\infty e^{-b(\tau-t)}\eta(\tau)\,d\tau.
\]
Choosing $t\ge t_0(\eta,b/2)$ and using
$\eta(\tau)\le e^{\frac b2(\tau-t)}\eta(t)$ for $\tau\ge t$,
we obtain
\[
\left|\sum_i y_i(t)a_i\right|
\le M\eta(t)\int_t^\infty e^{-\frac b2(\tau-t)}\,d\tau
\le \frac{2M}{b}\eta(t).
\]
Taking the supremum over all such $\{a_i\}$ yields
$\sum_i |y_i(t)|^2\le CM^2\eta^2(t)$ for all sufficiently large $t$.
\end{proof}
\end{lemma}

\subsection{Enhanced decay rate}\label{sec:5.2} Recall that the decomposition of  $u$ in  \eqref{def:ut}  
\begin{equation}\label{def:ut2}
u=u^T+H(u^T)+\tilde{u}^\perp. 
\end{equation} 
Let us introduce an auxiliary quantity
\begin{align}\label{def:Q}
{Q(t):= |z(t)||\bar{z}(t)|+|z(t)|\left\Vert  {\dot u}\right\Vert _{C^2}(t)+ |z(t)|\Vert \tilde u^\perp \Vert _{C^3}(t). }
\end{align}
In this section, as suggested in \eqref{equ:identification}, we recall that the domain of reduced functional $f$ is identified with an open ball in the Euclidean space of dimension $\dim\ker \mathcal{L}_\Sigma$.

\begin{lemma}\label{lem:5.5}
There exists  {$C<\infty$} such that, for large time $t\ge t_0=t_0(u)$,  
\begin{align}\label{equ:MQ}
 | \mathcal{M}_\Sigma (u)+\nabla f&(z(t))-\mathcal{L}_{\Sigma} \tilde{u}^\perp | \leq C   Q(t), \text{ and}
\\\label{equ:NQ}
&|N_1(u)| \leq C  Q(t).
\end{align}
\begin{proof}

Note that the coefficients of $u^T\in \ker \mathcal{L}_{\Sigma}$ are given by $z_j(t)-\bar{z}_j(t)$.
\begin{align*}
u^T(t)= \sum_{j=1}^{J} ( z_j(t) -  \bar{z}_j(t) ) \varphi_{\iota+j}. 
\end{align*}
Fix $\alpha=1/2$. From Lemma~\ref{lem:MSigma} $| \mathcal{M}_\Sigma (u)+\nabla f(z)-\mathcal{L}_{\Sigma} \tilde{u}^\perp   |$ is bounded by 
\begin{align*} C  \big(  \Vert u\Vert_{C^{2,\alpha}(\Sigma)}\Vert \tilde{u}^\perp \Vert_{C^{2,\alpha}(\Sigma)}+  |\nabla f(z-\bar{z})-\nabla f(z) |\big).
\end{align*}
Clearly $\Vert u(t) \Vert_{C^{2,\alpha}(\Sigma)}  \leq C  \Vert u \Vert_{C^{3}}(t)$ and $\Vert \tilde{u}^{\perp} (t) \Vert_{C^{2,\alpha}(\Sigma)} \leq C  \Vert \tilde{u}^{\perp} \Vert_{C^{3}}(t)$. From Corollary~\ref{cor:sobolev}, $\Vert u \Vert_{C^{3} }(t) \leq C   |z(t)|$. Therefore,
$\Vert u(t) \Vert_{C^{2,\alpha}(\Sigma)}\Vert \tilde{u}^\perp(t)  \Vert_{C^{2,\alpha}(\Sigma)}  \leq C  Q(t).$
Because $f-f(0)$ vanishes at the origin of degree $p$, $|\nabla^2 f(x)| \leq C  |x|^{p-2} \le C|x|$ near the origin. Together with $|\bar{z}(t)| \leq C  |z(t)|$, 
$
|\nabla f(z(t)-\bar{z}(t))-\nabla f(z(t)) |  \leq C   Q(t).
$
Hence \eqref{equ:MQ} holds.

Recall that, in \eqref{equ:N}, $N_1(u)=a_1\cdot D\dot u+a_2\cdot\dot u+a_3\cdot\mathcal{M}_\Sigma(u).$ By Corollary~\ref{cor:sobolev} and the structural assumptions on the
coefficients written after \eqref{equ:N}, for large time, $|a_1|+|a_2|+|a_3|\leq C|z(t)|.$
Using \eqref{equ:main} to write $\mathcal{M}_\Sigma(u)=N_1(u)-\ddot u+m\dot u$,
we obtain
\[
(I-a_3)N_1
=
a_1\cdot D\dot u+a_2\cdot\dot u-a_3\cdot\ddot u
+m a_3\cdot\dot u.
\]
Since $a_3\to0$, $I-a_3$ is uniformly invertible for large time.
Therefore,
\[
|N_1(u)|
\leq C|z(t)|\|\dot u\|_{C^1}(t)
\leq C Q(t),
\]
which proves \eqref{equ:NQ}.
\end{proof}
\end{lemma}
 
\begin{lemma}\label{lem:WbarW} There exists a constant $C<\infty $ such that, for $t\ge t_0(u)$,
\begin{align}\label{equ:zprimeQ}
&|z'(t)+m^{-1}\nabla f(z(t))| \leq C  Q(t), \text{ and }
\\\label{equ:zbarQ}
&|\bar{z}'(t)-m\bar{z}(t)+m^{-1}\nabla f(z(t))|  \leq C  Q(t).
\end{align}
\begin{proof}
In view of \eqref{equ:odegroup3}, it suffices to show
\begin{align*}
\big|\mathcal{W}_{j}(t)+ m^{-1} \frac{\partial f}{\partial x^j}(z(t)) \big| \leq C  Q(t)\ \textup{and}\ \big|\overline{\mathcal{W}}_j(t)+m^{-1}\frac{\partial f}{\partial x^j}(z(t)) \big| \leq C Q(t).
\end{align*}
From \eqref{def:error} and $\mathcal{E}(u)= (0,E_1(u))$, we actually have $\mathcal{W}_{j}(t)=\overline{\mathcal{W}}_j(t)$. Recall
\begin{align*}
\mathcal{W}_j=G((0,E_1(u)) ;  ( \varphi_{\iota+j},-2^{-1}m\varphi_{\iota+j}) )=-m^{-1} \int_\Sigma E_1(u)\varphi_{\iota+j}\, d\mu,
\end{align*}
where ${E}_1(u) =N_1(u)-\mathcal{M}_{\Sigma} (u) +\mathcal{L}_{\Sigma}u  $. Because $\varphi_{\iota+j}\in\ker \mathcal{L}_{\Sigma}$,\begin{align*}
\mathcal{W}_j=-m^{-1} \int_\Sigma (N_1(u)-\mathcal{M}_\Sigma(u))\varphi_{\iota+j}\, d\mu,
\end{align*}
Then the assertion follows from \eqref{equ:MQ} and \eqref{equ:NQ}.
\end{proof}
\end{lemma}

Let us define for $\lambda\ge2$, the function $\eta_\lambda(t)$ by \[\eta _\lambda(t) := |\nabla f(z(t))|+|z(t)|^\lambda .\]
\begin{lemma} \label{lemma-etalambda} For $\lambda \ge 2$, suppose $Q(t)\le M\eta_\lambda(t) $ for some $M<\infty$ for large time $t\ge t'$. Then $\eta(t):= \eta_\lambda(t)$ satisfies the condition of Lemma~\ref{lemma:integrate}.   
\begin{proof}
It is clear that $\eta_\lambda(t) \to 0$ as $t\to \infty$. Next, $\frac{d}{dt}|\nabla f(z(t))|^2 = 2 \nabla^2f\,  [{\nabla f}, z'(t)] $ implies $|\frac{d}{dt}|\nabla f(z(t))|| \le |\nabla^2 f ||z'(t)| \le C |z(t)|^{p-2} |z'(t)|$. Since $|z'(t)| \le C(|\nabla f(z(t))|+ Q(t)) \le C \eta_\lambda(t)$ and $p\ge 3$, 
\bea \big |  \frac{d}{dt } \eta_{\lambda}(t) \big| & \le C (|z(t)|^{p-2} +|z(t)|^{\lambda-1})|z'(t)| \le C |z(t)| \eta_\lambda(t).  \eea 
As $|z(t)|\to 0$, this shows $e^{\varepsilon t} \eta_\lambda(t)$ and $e^ { -\varepsilon t}\eta_\lambda(t)$ are eventually increasing and decreasing, respectively.  

\end{proof}

\end{lemma}

\begin{corollary}\label{cor:source1} For $\lambda \ge2$,
suppose $ \left| Q(t)\right|\leq M \eta_\lambda(t) $ for large time $t\ge t'$ for some $M<\infty$. Then there exists $C=C(M)<\infty $ such that $$ \left|{z}'(t)\right|  \leq C \eta_\lambda(t)\ \textup{and}\ \left| \bar{z}(t)\right|  \leq C\eta_\lambda(t), $$ for further large time $t\ge t''$. 
\begin{proof}
The bound for $z'(t)$ follows from \eqref{equ:zprimeQ}. In view of Lemma \ref{lemma-etalambda}, the bound for $\bar{z}(t)$ follows by {applying Lemma~\ref{lemma:integrate}} to \eqref{equ:zbarQ} with choices $\eta=\eta_\lambda$, $y_i=\bar z_i$ for $1\le i\le J$ and $y_j=0$ for $j>J$. 
\end{proof}
\end{corollary}

We vectorize $\tilde{u}^\perp$ and perform the projection. Set $\tilde{q}:=q(\tilde{u}^\perp)$ 
and
\begin{equation*} 
\begin{split}
&\tilde\xi_{i,1}(t):= G(\tilde q(t),\psi_{i,1}),\ \tilde\xi_{i,2}(t):=G(\tilde q(t),\psi_{i,2})\ \textup{for}\ i\in I_1,\\ 
 &\tilde\xi_{i}(t):= G(\tilde q(t),\Psi_{i})\ \textup{for}\ i\in \mathbb{Z}\setminus\{0\},
\end{split}
\end{equation*}
Note that because $\tilde{u}^\perp$ is orthogonal to $\ker \mathcal{L}_{\Sigma}$, those coefficients completely characterize $\tilde{q}$. The projections of higher order derivatives of $\tilde q$, namely $\tilde q ^{(k,\ell)}={\partial_t^k}\mathbf{L}^\ell \tilde q$, are defined similarly.  In the lemma below, we show that $\{\tilde{\xi}_{i,1},\tilde{\xi}_{i,2}\}_{i\in I_1}$, $\{\tilde{\xi}_i\}_{i\in\mathbb{Z}\setminus\{0\}}$ and the higher order coefficients are bounded by $|z(t)|$.
\begin{lemma}[control on higher derivatives of $\tilde q$] \label{lemma-iterate1}
For any $k,\ell\in\mathbb{N}_0$, there exists a positive constant $C=C(u,m,\ell+k)$ such that
 \begin{align*}
 \sum_{i\in I_1}\left(|\tilde{\xi}^{(k,\ell)}_{i,1}(t)|^2+|\tilde{\xi}^{(k,\ell)}_{i,2}(t)|^2 \right)+ \sum_{i\neq 0}|\tilde{\xi}^{(k,\ell)}_i(t)|^2  \leq C |z(t)|^2.
 \end{align*}

\begin{proof} 
Let $s=\ell+k$.  
\begin{align*}
\sum_{i\in I_1}\left(|\tilde{\xi}^{(k,\ell)}_{i,1}(t)|^2+|\tilde{\xi}^{(k,\ell)}_{i,2}(t)|^2 \right)+ \sum_{i\neq 0}|\tilde{\xi}^{(k,\ell)}_i(t)|^2=\|\tilde{q}^{(k,\ell)}\|^2_G(t)   \leq C  \|\tilde{u}^\perp\|^2_{C^{s+1}}(t).
\end{align*} 
The assertion then follows from Lemma~\ref{lem:utup} (boundedness of decomposition) and Corollary~\ref{cor:sobolev} (control on higher derivatives). 
\end{proof}
\end{lemma}

Let us define
\begin{align}\label{def:tildeE}
&\tilde{E}(u):= \left(\tilde u^\perp\right)''- m\left(\tilde u^\perp\right)' +  \mathcal{L}_{\Sigma} \tilde u^\perp,\text{ and } \tilde{\mathcal{E}}:=(0,\tilde{E}(u)).
\end{align}  
Set $\{\tilde{\mathcal{E}}_{i,1},\tilde{\mathcal{E}}_{i,2}\}_{i\in I_1}$ and $\{\tilde{\mathcal{E}}_i\}_{i\in\mathbb{Z}\setminus\{0\}}$ be the coefficients of $\tilde{\mathcal{E}}$. Namely, 
\begin{equation}\label{def:tildeE_coef} 
\begin{split}
&\tilde{ \mathcal{E}}_{i,1}(t):= G( \tilde{ \mathcal{E}}(t),\psi_{i,1}),\ \tilde{ \mathcal{E}}_{i,2}(t):= G(\tilde{ \mathcal{E}}(t),\psi_{i,2})\ \textup{for}\ i\in I_1,\\
&\tilde{ \mathcal{E}}_{i}(t):= G(\tilde{ \mathcal{E}}(t),\Psi_{i})\ \textup{for}\ i\in \mathbb{Z}\setminus\{0\}.
\end{split}
\end{equation}

Equation \eqref{def:tildeE} can be rewritten as
\[
\tilde{q}'=\mathbf{L}\tilde{q}+\tilde{\mathcal{E}}.
\]
This is analogous to deriving \eqref{eq-vectoru} from \eqref{equ:minimalODE}. Projecting it onto $\psi_{i,1}$, $\psi_{i,2}$, and $\Psi_i$, respectively, yields the ODEs for the coefficients appearing in \eqref{def:tildeE}. 
More precisely, for $i\in I_1$,
\begin{equation}\label{equ:odegroup1t}
\frac{d}{dt} \tilde{\xi}_{i,1}-2^{-1}m \tilde\xi_{i,1}+\beta_i\tilde\xi_{i,2}
=\tilde{\mathcal{E}}_{i,1}, \quad
\frac{d}{dt} \tilde\xi_{i,2}-2^{-1}m \tilde\xi_{i,2}-\beta_i\tilde\xi_{i,1}
=\tilde{\mathcal{E}}_{i,2},
\end{equation}
and, for $i\in\mathbb{Z}\setminus\{0\}$,
\begin{equation}\label{equ:odegroup2t}
\frac{d}{dt} \tilde\xi_i-\Gamma_i \tilde\xi_i=\tilde{\mathcal{E}}_i .
\end{equation}
This is the same computation used to derive \eqref{equ:odegroup1_exp} and
\eqref{equ:odegroup_exp} from \eqref{equ:minimalODE}, or equivalently from
\eqref{eq-vectoru}.

\begin{lemma}\label{lem:tildeE}
 There exists a positive constant {$C$} such that
$
\|\tilde{E}(u)\|_{L^2}(t)  \leq C  Q(t) $,
{for large time $t\ge t_0=t_0(u)$.}\end{lemma}
\begin{proof} Observe $\tilde E(u)$ becomes 
\begin{align*}
 & \Pi ^\perp \bigg [  E_1(u)- \left( H(u^T)\right)''  +m\left( H(u^T)\right)'  -   \mathcal{L}_{\Sigma} H(u^T) \bigg] \\
&= \Pi ^\perp \bigg [ -\mathcal{M}_{\Sigma} (u) +\mathcal{L}_{\Sigma} u  +N_1(u) - \left( H(u^T)\right)''  +m\left( H(u^T)\right)'  -   \mathcal{L}_{\Sigma} H(u^T) \bigg].  
\end{align*}
We decompose $\tilde E(u)= \textsc{I}+\textsc{II}$, where  \begin{align*}
 \textsc{I}&= -\Pi^\perp \mathcal{M}_\Sigma(u) + \mathcal{L}_{\Sigma} \tilde{u} ^\perp+\Pi ^\perp N_1(u),\ \textsc{II} =  - \left( H(u^T)\right)''  +m\left( H(u^T)\right)' .
 \end{align*} 
From \eqref{equ:MQ}-\eqref{equ:NQ},  $\|\textsc{I}\|_{L^2}(t)    \leq C    Q(t).$ Also, note that 
\begin{align*}
\textsc{II}= -D^2H(z-\bar{z})[ (z-\bar{z})', (z-\bar{z})']-DH(z-\bar{z})[(z-\bar{z})'']+m DH(z-\bar{z})[(z-\bar{z})'].
\end{align*}	
Since $DH(0)=0$, the above is bounded by $$  C |(z-\bar{z})'|^2+ C | z-\bar{z}  ||(z-\bar{z})''|+ C | z-\bar{z}  ||(z-\bar{z})'|   \leq C  \Vert u^T \Vert _{C^2}(t) \Vert \partial _t u ^T \Vert_{C^1}(t).$$
Together with Corollary~\ref{cor:sobolev}, $\|\textsc{II}\|_{L^2}(t)   \leq C  |z(t)| \Vert  {\dot u}\Vert_{C^1 }(t)   \leq C  Q(t).$
\end{proof}

\begin{corollary}\label{cor:source}
Suppose $Q(t) \leq M \eta_\lambda(t) $ for some $M<\infty$ and $\lambda\ge 2$ for large time. Then there is {$C=C(M)<\infty$} such that, for further large time,
\begin{align*}
\sum_{i\in I_1}|\tilde{\xi}_{i,1}(t)|^2+|\tilde{\xi}_{i,2}(t)|^2+\sum_{i\in\mathbb{Z}\setminus\{0\}}|\tilde{\xi}_{i}(t)|^2  \leq C \eta^2_\lambda(t) .
\end{align*} 
\end{corollary}
\begin{proof}
{From \eqref{def:tildeE}, \eqref{def:tildeE_coef}, Lemma~\ref{lem:tildeE} and the assumption, 
\begin{align*}
\sum_{i\in I_1}\big (|\tilde{\mathcal{E}}_{i,1} |^2+|\tilde{\mathcal{E}}_{i,2} |^2 \big )+\sum_{i\in\mathbb{Z}\setminus\{0\}}|\tilde{\mathcal{E}}_{i} |^2= \Vert \tilde{\mathcal{E}}\Vert^2_G  = \tfrac{2}{m^2}\|\tilde{E}(u)\|^2_{L^2}   \leq \tfrac{2}{m^2} C^2M^2  \eta_\lambda^2 ,
\end{align*}
where $C$ is the constant from Lemma~\ref{lem:tildeE}. Note that $\inf_{i\in\mathbb{Z}\setminus\{0\}} |\Gamma_i|=b>0$ for some $b=b(m,\mathcal{M}_\Sigma)$. Hence the bound for $\sum_{i\in\mathbb{Z}\setminus\{0\}}|\tilde{\xi}_{i}(t)|^2$ follows by applying Lemma~\ref{lemma:integrate} to \eqref{equ:odegroup2t}. From \eqref{equ:odegroup1t}, \begin{align*}
\left|\frac{d}{dt} \sqrt{ |\tilde{\xi}_{i,1} |^2+|\tilde{\xi}_{i,2} |^2}-\frac{m}{2}\sqrt{ |\tilde{\xi}_{i,1} |^2+|\tilde{\xi}_{i,2} |^2}   \right|\leq  \sqrt{ |\tilde{\mathcal{E}}_{i,1} |^2+|\tilde{\mathcal{E}}_{i,2} |^2}.
\end{align*}
Applying Lemma~\ref{lemma:integrate} to the above, we derive the bound for remaining terms.}
\end{proof}

\begin{lemma}[improvement in decay] \label{lem-iterate2} For a given $ \lambda \ge 2 $, suppose 
\be\label{equ:iterateQ} Q(t) \le M \eta_{\lambda}(t),\ee 
{for large time $t\ge t'$.} For given $\rho \in (0,1)$,  there holds\begin{align}\label{equ:zprimebound}
Q(t) \leq C  |z(t)|^{\rho  }\eta_\lambda(t) \le C \eta_{\lambda+\rho  }(t),
\end{align}   
for further large time $t\ge t_0=t_0(u, \rho,\lambda,M )$ and some {$C=C(\rho,\lambda,M)<\infty $}.
 
\begin{proof} 
Consider a small number $\varepsilon\in (0,1)$ to be fixed later. From Corollaries \ref{cor:source1} and \ref{cor:source}, for large enough time $t$, 
\begin{align}\label{equ:iterateamp}
| z'  |^2+|\bar{z} |^2 + \sum_{i\in I_1}|\tilde{\xi}_{i,1} |^2+|\tilde{\xi}_{i,2} |^2+ \sum_{i\neq 0}|\tilde{\xi}_i |^2  \leq M   \eta_\lambda^2 ,
\end{align}  Using  \eqref{equ:iterateamp}, Lemma~\ref{lem:ODE}  (dominance of neutral mode) and Lemma~\ref{lemma-interpolation} (interpolation), for all $k,\ell\in\mathbb{N}_0$, \[ \bigg|\frac{d}{dt} z^{(k,\ell)} (t)\bigg|^2+|\bar{z}^{(k,\ell)}(t)|^2 \leq C_{k,\ell,\eps }  \big( \sup_{s\in[t-1,t+1]}[ \eta_\lambda(s)]  \big)^{2(1-\varepsilon)} \big(\sup_{s\in[t-1,t+1]} \Vert u \Vert_{C^N}(s)\big)^{2\eps}.\] 
Here $N$ is a large integer that depends on $\eps>0$, $k$ and $\ell$. Note $\eta(t)$ satisfies the condition of Lemma \ref{lemma:integrate} by Lemma \ref{lemma-etalambda}. Therefore, $\sup_{s\in[t-1,t+1]} \eta_\lambda(s) \le 2\eta_{\lambda}(t)$ for large $t$. Utilizing $\Vert u\Vert_{C^N}(s)\to 0$, for given $\eps>0$,
\begin{equation}\label{eq-419tilde}
\left|\frac{d}{dt} z^{(k,\ell)} (t)\right|^2+|\bar{z}^{(k,\ell)}(t)|^2 \leq C_{k,\ell,\eps }[ \eta_{\lambda}(t)]^{2(1-\eps)},
\end{equation}
for $t\ge t_0=t_0(u,\eps)$. Similarly, applying Lemma \ref{lemma-interpolation} and Sobolev interpolation inequality (for functions on $\Sigma$) with Lemma~\ref{lemma-iterate1}, 
\be  \label{eq-418tilde}
\sum_{i\in I_1}\left(|\tilde{\xi}^{(k,\ell)}_{i,1}(t)|^2+|\tilde{\xi}^{(k,\ell)}_{i,2}(t)|^2 \right)+ \sum_{i\neq 0}|\tilde{\xi}^{(k,\ell)}_i(t)|^2 \leq C_{k,\ell,\eps }  [ \eta_\lambda(t)]^{2(1-\varepsilon)},
\ee for $t\ge t_0$. 
In view of Corollary~\ref{cor:qE}, estimates \eqref{eq-419tilde} and \eqref{eq-418tilde} imply that the $L^2$ norms of all derivatives of $(u^T)'$ and $\tilde{u}^{\perp}$ are bounded by $C[\eta_\lambda(t)]^{1-\varepsilon}$, where the constant $C$ may depend on the order of differentiation.
By the Sobolev embedding, $\|(u^T)'\|_{C^3 }(t)+\|\tilde{u}^\perp\|_{C^3 }(t)\leq C [ \eta_\lambda(t)]^{1-\varepsilon}$. From \eqref{def:ut2}, we also have $\| {\dot u} \|_{C^3 }(t) \leq  [ \eta_\lambda(t)]^{(1-\varepsilon)}$. In view of the definition of $Q(t)$ in \eqref{def:Q}, 
$$Q(t) \leq   C |z(t)|[ \eta_\lambda(t)]^{(1-\varepsilon)} \leq C (|z(t)||\nabla f (z(t))|^{1-\e}+ |z(t)|^{\lambda (1-\varepsilon)+1}).$$  In view of Young's inequality
\[|z(t)||\nabla f (z(t))|^{1-\e} \le (1-\e)|z(t)|^{\rho }|\nabla f (z(t))|+\e |z(t)|^{ \frac{(1-\rho)+\eps \rho }{\eps} }, \]
we obtain \eqref{equ:zprimebound} once we choose $\eps$  small depending on $\lambda$ and $\rho $.
\end{proof}

\end{lemma}
{Now we are ready to give the proof of Proposition~\ref{prop-neutral-dynamics}. The assertion (1) follows by series of lemmas shown in this section. Except its lower bound, it should be noted that the assertion (2) follows by a rather independent argument based on \L ojasiewicz-Simon inequality.
\begin{proof}[Proof of Proposition~\ref{prop-neutral-dynamics} (1)]
 To show \eqref{equ:zgradient}, we may assume $N\ge2$ as $|z(t)|\to 0$. From Lemma~\ref{lem:ODE} (dominance of neutral mode), Lemma~\ref{lem:utup} (boundedness of decomposition) and Corollary~\ref{cor:sobolev} (control on higher derivatives), the assumption~\eqref{equ:iterateQ} in Lemma~\ref{lem-iterate2} (improvement in decay) holds for $\lambda=2$. By iterating Lemma~\ref{lem-iterate2},  
\be \label{eq-166}Q(t)\le C (|z(t)|^{\rho } |\nabla f(z(t)) | + |z(t)|^N ),\ee 
for some $C=C(\rho,N)$ for large time $t\ge t_0=t_0(u,\rho,N)$. Combining this with Lemma~\ref{lem:WbarW} yields \eqref{equ:zgradient}. 
\end{proof}

\begin{proof}[Proof of Proposition~\ref{prop-neutral-dynamics} (2)]
As the lower bound was shown in Corollary~\ref{cor-Xi^+_0}, we proceed with the upper bound. The equation~\eqref{equ:main} can be written as 
\be \label{eq-mainremind} \dot u  - m^{-1}\mathcal{M}_{\Sigma}u =  m^{-1}\ddot u - m^{-1}{N}_1(u) .\ee 
It was proved in \cite[(6.58)-(6.59)]{S0} by Simon that for every small $\eta >0$, there exists $t_0=t_0(u,\eta)$ such that 
\be \label{equ:nigligiblesecondorder}  \Vert  \ddot u \Vert _{L^2} \le \eta \Vert \dot  u \Vert _{L^2},\text{ for } t\ge t_0.\ee  Moreover, there exist $C=C(\mathcal{M}_\Sigma, N_1)$ and $t_0=t_0(u)$ such that 
\[ \quad \Vert \slashed{\nabla} \dot u\Vert _{L^2} \le C \Vert \dot u \Vert_{L^2}  ,\text{ for } t\ge t_0.\]  
In view of the structure of $N_1$ in \eqref{equ:N}, 
\[\Vert  \dot u -m^{-1}\mathcal{M}_{\Sigma} u \Vert _{L^2} \le \eta (\Vert \mathcal{M}_\Sigma u\Vert _{L^2}+ \Vert \dot u \Vert _{L^2}), \text{ for } t\ge t_0=t_0(u,\eta).  \]
By taking $\eta$ small, we may assume
 $ \Vert \dot u \Vert_{L^2}\ge \frac12 \Vert   m^{-1}\mathcal{M}_\Sigma u\Vert _{L^2}$ and $  m^{-1} \mathcal{M}_{\Sigma} u\cdot \dot u   \ge \frac{1}{2}  \Vert \dot u \Vert _{L^2} \Vert m^{-1} \mathcal{M}_\Sigma u \Vert _{L^2}$. Therefore,
\bea \label{eq-6891} \frac{d}{d t} m^{-1}\mathcal{F}(u(t)) = - m^{-1}\int_{\Sigma}  \mathcal{M}_\Sigma u \cdot \dot u \, & \le  -\frac12 \Vert \dot u \Vert _{L^2} \Vert m^{-1}\mathcal{M}_\Sigma u \Vert _{L^2}. \eea 
Note that the analytic function $f(x)$ satisfies the \L ojasiewicz inequality  $|\nabla f(x)| \ge c |f(x)- f(0) |^\theta$ for some exponent $\theta \in [1/2,1)$ near $0$.
From \cite[Theorem 3]{S0}, $\mathcal{F}$ satisfies a \L ojasiewicz-Simon inequality for the same exponent $\theta$ near the origin: 
\be  \label{eq-simonloj0}\Vert \mathcal{M}_\Sigma u \Vert _{L^2}\ge c|\mathcal{F}(u)-\mathcal{F}(0)|^\theta, \ee  for some possibly smaller $c>0$.  Take $\bar{\mathcal{F}}=m^{-1}\mathcal{F}$. By the Lojasiewicz-Simon inequality \eqref{eq-simonloj0}, 
\bea \frac{d}{d t} \bar{\mathcal{F}}(u(t))  \le -\frac12 \Vert \dot u \Vert _{L^2} \Vert m^{-1}\mathcal{M}_\Sigma u \Vert _{L^2} \le  -\frac c2 \Vert \dot  u \Vert_{L^2}| \bar{\mathcal{F}}(u(t))-\bar{\mathcal{F}}(0)|^\theta . \eea  
This shows $\bar{\mathcal{F}}(u(t))$ is non-increasing for large times and 
\[\sigma(t):=\int _{t}^\infty \Vert \dot u \Vert_{L^2}(s)  ds \le \frac{2}{c(1-\theta)} |\bar{\mathcal{F}}(u(t))-\bar{\mathcal{F}}(0)|^{1-\theta},  \]for sufficiently large times. Therefore, again with the \L ojasiewicz inquality, for $t\ge t_0$, there exist $\theta\in[1/2,1)$ and $c'>0$ such that 
\[-\dot \sigma(t) =\Vert \dot u \Vert_{L^2}\ge \frac12 \Vert   m^{-1}\mathcal{M}_\Sigma u\Vert _{L^2} \ge \frac{c}{2}|\bar{\mathcal{F}} (u(t))-\bar{\mathcal{F}}(0)|^{\theta} \ge c' \sigma^{\frac{\theta}{1-\theta}}(t).\] 
If $\theta=1/2$, then by integrating the inequality above, $\sigma(t) \le \sigma(t_0) e^{-c'(t-t_0)} $. If $\theta\in(1/2,1)$, then $\sigma(t) \le c''t ^{\frac{1-\theta}{1-2\theta}}$. 
Since $u(t)\to 0 $, $\Vert u(t)-0\Vert_{L^2}  \le \int_t^\infty \Vert \partial _t u\Vert_{L^2}  ds =\sigma(t)$ and hence $\Vert u(t) \Vert_{L^2} $ has desired upper bound.

\end{proof}

}

\section{Slowly decaying solutions to parabolic equation} \label{sec:slowparabolic}
We consider the slowly decaying solutions to the parabolic equation~\eqref{equ:parabolic}. The main goal is to prove Proposition~\ref{prop-neutral-dynamics-p1}, an analogous result of Proposition~\ref{prop-neutral-dynamics} for elliptic equation~\eqref{equ:main}.
\vspace{0.1cm}

Let $u\in C^\infty(Q_{0,\infty},\widetilde{ \mathbf{V}})$ be a solution to \eqref{equ:parabolic} with $\|u\|_{H^{n+4}}(t)=o(1)$ as $t\to \infty$. From Lemma~\ref{lem:regularity} (parabolic regularity), $\|u\|_{C^s}(t)=o(1)$ for all $s\in\mathbb{N}$. We further assume that $u$ does not decay exponentially (See Theorem \ref{thm:general_exponential_p} and Section \ref{sec:exp}.)  Namely, for any $\varepsilon>0$, 
\begin{equation}\label{equ:nonexpdecay_p1}
\limsup_{t\to \infty} e^{ \varepsilon t}\|u\|_{C^1}(t)=\infty.
\end{equation}
We project $u$ onto the eigensections $\varphi_i$. Set
\begin{align}\label{def:parabolic_xi}
\xi_i(t):=\int_\Sigma  \left\langle u,\varphi_{i}\right\rangle \, d\mu. 
\end{align}
Recall that $\{i\in\mathbb{N}\, :\, \lambda_i=0\}=\{\iota+1,\iota+2,\dots, \iota+J\}$. The projection onto the neutral mode will play a crucial role as the role of $z(t)$ in the elliptic case, and here we denote it by
\begin{align}\label{def:parabolic_x}
x_j(t):= \xi_{\iota+j}(t).
\end{align}

\begin{proposition}\label{prop-neutral-dynamics-p1}   
Let $u$ be a slowly decaying solution to \eqref{equ:parabolic}. 
\begin{enumerate}[$(1)$]
\item  For given $\rho \in(0,1)$ and $N\in \mathbb{N}$, there exists a positive constant $C=C(\mathcal{M}_\Sigma, N_2,\rho )<\infty $ such that
\begin{align}\label{equ:xgradient}
 |x' +\nabla f(x )| \leq C   (|x |^{\rho }|\nabla f(x )|+|x |^N ), \text{ for time } t\ge t_0=t_0(u,N,\rho).
\end{align} 
\item  There exist $0<D_1, D_2<\infty$, $\alpha_2 >0$, which are constants of $\mathcal{M}_\Sigma$, $N_2$ such that, for $t\ge t_0=t_0(u)$,   \be \label{eq-algebraicratebound-5p} D_1 t^{-1} \le |x(t)|\le D_2 t^{-\alpha_2}.\ee\end{enumerate} 
\end{proposition}

The proof of Proposition~\ref{prop-neutral-dynamics-p1} follows a series of arguments similar to those of Proposition~\ref{prop-neutral-dynamics}. The first part is to show (in Corollary~\ref{cor:parabolic}) that $|x(t)|$ dominates any $C^s$ norm of $u$. 

\bigskip

Let $E_2(u)=N_2(u)+\mathcal{M}_\Sigma (u)-\mathcal{L}_{\Sigma} u$. Then \eqref{equ:parabolic} becomes
\begin{equation}\label{equ:linear_p}
{\dot u}-\mathcal{L}_{\Sigma} u=E_2(u).
\end{equation}
From \eqref{equ:quasilinear} and \eqref{equ:N}, the error term $E_2(u)$ has the structure 
\begin{align} \label{equ:E_2str}
E_2(u)=\sum_{j=0}^2 b_{2,j}\cdot \slashed{\nabla}^j u,
\end{align}
where $b_{2,j}=b_{2,j}(\omega, u,\slashed{\nabla}u)$ are smooth with $b_{2,j}(\omega,0,0)=0$. With $\|u\|_{C^s}(t)=o(1)$, the quadratic nature of $E_2$ allows us to bound higher derivatives of $E_2$ in terms of higher derivatives of $u$.
\begin{lemma}\label{lem:error_p}
For any integer $s\geq 2$, 
\begin{align*}
\sum_{2k+\ell\leq 2s}\| \partial^k_t\slashed{\nabla}^\ell E_2(u) \|_{L^2}(t)=o(1)\sum_{2k+\ell\leq 2s}\| \partial^k_t\slashed{\nabla}^\ell u \|_{L^2}(t).
\end{align*}
\begin{proof}
We present the proof for
\begin{align*}
\sum_{2k+\ell\leq 2s}\left\| \partial^k_t\slashed{\nabla}^\ell \left( b_{2,2}\cdot\slashed{\nabla}^2 u \right)  \right\|_{L^2}(t)=o(1)\sum_{2k+\ell\leq 2s}\| \partial^k_t\slashed{\nabla}^\ell u \|_{L^2}(t).
\end{align*} 
The other terms can be treated similarly. For simplicity, we write $b(\omega,u,\slashed{\nabla}u )$ for $b_{2,2}(\omega, u,\slashed{\nabla}u )$. For $m_0,m_1,m_2 \in\mathbb{N}_0$, we write $b^{(m_0, m_1,m_2 )}$ for the partial derivative of $b$ of order $(m_0,m_1,m_2 )$. Fix $k_0,\ell_0$ with $2k_0+\ell_0\leq 2s$. Then the terms in the expansion of $\partial^{k_0}_t\slashed{\nabla}^{\ell_0} \left( b \cdot\slashed{\nabla}^2 u \right)$ are of the form
\begin{align*}
b^{(m_0,m_1,m_2 )}\cdot\left(\partial^{k_1}_t\slashed{\nabla}^{\ell_1} u \ast \partial^{k_2}_t\slashed{\nabla}^{\ell_2} u \ast\dots \partial^{k_N}_t\slashed{\nabla}^{\ell_N} u \right),
\end{align*} 
where $N=m_1+m_2 +1$, $\sum_{i=1}^N k_i=k_0 $ and $\sum_{i=1}^N \ell_i=\ell_0+m_2+2$. It suffices to show the pointwise bound
$$b^{(m_0,m_1,m_2 )}\cdot\left( \big(\partial^{k_1}_t\slashed{\nabla}^{\ell_1} u  \big)\cdot \big(\partial^{k_2}_t\slashed{\nabla}^{\ell_2} u \big)  \cdot \ldots \cdot \big( \partial^{k_N}_t\slashed{\nabla}^{\ell_N} u\big) \right)=o(1)\sum_{2k+\ell\leq 2s} |\partial^k_t\slashed{\nabla}^\ell u|.$$ We first discuss the case $m_1=m_2=0$. The above becomes $b^{(m_0,0,0)}\cdot \partial^{k_0}_t\slashed{\nabla}^{\ell_0-m_0+2} u $. From $$|b^{(m_0,0,0)}| \leq C ( |u|+|\slashed{\nabla }u|)\ \textup{and}\  |\partial^{k_0}_t\slashed{\nabla}^{\ell_0-m_0+2} u |=o(1),$$ the assertion holds. 

Next, we consider the case $m_1+m_2 \geq 1$. It suffices to show $\min_i (2k_i+\ell_i)\leq 2s$ as the other terms can be bounded by  $o(1)$. Suppose this fails. In other words, $2k_i+\ell_i\geq 2s+1$ for all $1\leq i\leq N$. Then
\begin{align*}
2Ns+N\leq   \sum_{i=1}^N (2k_i+\ell_i)=2k_0+\ell_0+m_2+2\leq 2s+2N. 
\end{align*} 
In view of $N\geq 2$ and $s\geq 2$, this is a contradiction. 
\end{proof}
\end{lemma}
Now we project \eqref{equ:linear_p} onto the eigensections $\varphi_i$. Let
\begin{align*}
\xi^{(k,\ell)}_i:=\int_\Sigma  \left\langle \mathcal{L}_{\Sigma}^{\ell}\partial^{k}_t u,\varphi_{i}\right\rangle\, d\mu,\ \mathcal{E}^{(k,\ell)}_i:=\int_\Sigma  \left\langle \mathcal{L}_{\Sigma}^{\ell}\partial^{k}_t E_2(u),\varphi_{i}\right\rangle\, d\mu
\end{align*}
Then \eqref{equ:linear_p} becomes
\begin{align}\label{equ:ODE_sys_p}
\frac{d}{dt}\xi^{(k,\ell)}_i-\lambda_i \xi^{(k,\ell)}_i=\mathcal{E}^{(k,\ell)}_i.
\end{align}
For the neutral mode, we denote 
$
x^{(k,\ell)}_{j}:= \xi^{(k,\ell)}_{\iota+j}$ and $\mathcal{W}^{(k,\ell)}_{j}:=\mathcal{E}^{(k,\ell)}_{\iota+j}
$.
\begin{lemma}[\textit{c.f.} Lemma \ref{lem:ODE}]\label{lem:ODE_p}
For any $s\geq 2$, as $t\to \infty$, 
\begin{align*}
\sum_{k+\ell\leq s}\sum_{i=1}^\infty  |\xi^{(k,\ell)}_i|^2=(1+o(1))|x|^2.
\end{align*}
\begin{proof}Define three non-negative functions $X_+(t)$, $X_0(t)$ and $X_-(t)$ by
\begin{align*}
X^2_0 =&\sum_{k+\ell\leq s}\sum_{1\leq j\leq J}|x^{(k,\ell)}_j |^2,\quad  X^2_\pm =\sum_{k+\ell\leq s}\sum_{i: \pm\lambda_i>0 }|\xi^{(k,\ell)}_i |^2. 
\end{align*}
We note that these coefficients are grouped together according to the sign of the eigenvalues. From \eqref{equ:ODE_sys_p}, we compute $X_+X_+'$ and obtain 
\begin{align*}
\sum_{k+\ell\leq s}\sum_{i:\lambda_i>0} \left(\lambda_i  |\xi^{(k,\ell)}_i |^2+\xi^{(k,\ell)}_i \mathcal{E}^{(k,\ell)}_i \right)= \bigg (\sum_{k+\ell\leq s}\sum_{i:\lambda_i>0} \lambda_i  |\xi^{(k,\ell)}_i |^2 \bigg)+X_+Y_+.
\end{align*}
Here $X_+(t)Y_+(t)$ is defined by the last equality and $Y_+(t_0)=0$ if $X_+(t_0)=0$. Let $b=\min\{|\lambda_i|\, :\, \lambda_i\neq 0\}$. 

Proceeding as in \eqref{eq-15661}--\eqref{equ:42} and applying Lemma~\ref{lem:sqrt-lipschitz}, we have
$
X_+'-bX_+\geq Y_+, 
$ in the integral sense.

Similarly, define $Y_0(t)$ and $Y_-(t)$ by
\begin{align*}
X_0 Y_0 = \sum_{k+\ell\leq s}\sum_{1\leq j\leq J} x^{(k,\ell)}_j  \mathcal{W}^{(k,\ell)}_j ,\quad 
X_- Y_- = \sum_{k+\ell\leq s}\sum_{i:\lambda_i<0 } \xi^{(k,\ell)}_i \mathcal{E}^{(k,\ell)}_i .
\end{align*}If $X_0(t_0)=0$ or $X_-(t_0)=0$, we set $Y_0(t_0)=0$ or $Y_-(t_0)=0$, respectively.

It holds that $X_0'= Y_0$ and $X_-'+bX_-\leq Y_-$ in the integral sense.

Next, recall $|X_+ |^2+|X_0 |^2+|X_- |^2=\sum_{k+\ell\leq s} \| \mathcal{L}_{\Sigma}^{\ell}\partial^{k}_t u  \|_{L^2}^2$. From the Cauchy-Schwarz inequality, $ |Y_+ |^2+|Y_0 |^2+|Y_- |^2 \leq  \sum_{k+\ell\leq s}\sum_{i=1}^\infty |\mathcal{E}^{k,\ell}_i |^2   = \sum_{k+\ell\leq s} \| \mathcal{L}_{\Sigma}^{\ell}\partial^{k}_t E_2(u)  \|_{L^2}^2 .$
In view of Lemma~\ref{lem:error_p}, $|Y_+ |^2+|Y_0 |^2+|Y_- |^2=  o(1)\big( |X_+ |^2+|X_0 |^2+|X_- |^2 \big).$ Then we apply Lemma~\ref{lem-MZODE}. As the slow decay assumption \eqref{equ:nonexpdecay_p1} rules out the possibility that $X_- $ dominates,   \begin{equation}\label{equ:ODE_pmiddle1}
|X_+ |^2+|X_0 |^2+|X_- |^2=(1+o(1))|X_0 |^2. 
\end{equation}
It remains to show $
 |X_0 |^2=(1+o(1))|x |^2. 
$  If $\ell\geq 1$, then 
$
x^{(k,\ell)}_j = \int_\Sigma \left\langle \mathcal{L}_{\Sigma}^{\ell}\partial^{k}_t u,\varphi_{\iota+j}\right\rangle\, d\mu=\int_\Sigma \left\langle \partial^{k}_t u, \mathcal{L}_{\Sigma}^{\ell}\varphi_{\iota+j}\right\rangle\, d\mu=0
$, and  if $k\geq 1$, then 
$
x^{(k,\ell)}_j =\frac{d}{dt} x^{(k-1,\ell)}_j=\mathcal{W}^{(k-1,\ell)}_j.
$
These imply
\begin{align*}
 |X_0 |^2&\leq |x |^2+\sum_{k \leq s-1}  |\mathcal{W}^{(k,0)}  |^2 \leq |x |^2+o(1)\left(|X_+ |^2+|X_0 |^2+|X_- |^2\right).\end{align*}
Finally $
 |X_0 |^2=(1+o(1))|x |^2 
$ follows by \eqref{equ:ODE_pmiddle1}.
\end{proof}
\end{lemma}
\begin{corollary}[\textit{c.f.} Corollary \ref{cor:sobolev}]\label{cor:parabolic}
For any $\gamma \in \mathbb{N}$, there exists $C=C(\gamma )<\infty$ such that $\Vert u \Vert_{C^{\gamma }}(t) \leq C   |x(t) |,$ for $t\ge t_0=t_0(u,\gamma)$.  
\end{corollary}

\begin{corollary} [\textit{c.f.} Corollary \ref{cor-Xi^+_0}]\label{cor-Xi^+_0-p} For slowly decaying solution $u(t)$, there is $C=C(\mathcal{M}_\Sigma, N_2)<\infty$ such that $|x(t)|\ge (Ct)^{-1}$ for $t\ge t_0=t_0(u)$. 
 
 \end{corollary} 

The rest of the arguments are simpler than ones in the elliptic case. We omit the details and only provide the main steps. Their elliptic counterparts can be found in Section~\ref{sec:5.2}. Let us define
{
\begin{align*} Q(t)& :=|x(t)|\left\Vert  {\dot u}\right\Vert _{C^2}(t)+ |x(t)|\Vert \tilde u^\perp \Vert _{C^3}(t) , \text{ and}\\
\eta _\lambda(t) & := |\nabla f(x(t))|+|x(t)|^\lambda .\end{align*}}\begin{lemma}[\textit{c.f.} Lemma~\ref{lem:5.5}]\label{lem:3.5}There exists  $C=C(\mathcal{M}_\Sigma,N_{2} )<\infty $ such that, for $t\ge t_0=t_0(u)$, $$|\mathcal{M}_\Sigma(u)+\nabla f(x(t))-\mathcal{L}_\Sigma\tilde{u}^\perp|+|N_2(u)| \leq C  Q(t).$$ \end{lemma} 

\begin{lemma}[\textit{c.f.} Lemmas~\ref{lem:WbarW} and \ref{lem:tildeE}]\label{lem:3.6}There exists a positive constant $C=C(\mathcal{M}_\Sigma,N_2 )$ such that, for $t\ge t_0=t_0(u)$, there holds
$$|x'(t) +\nabla f(x(t))|+\left \Vert (\tilde{u}^\perp )'-\mathcal{L}_\Sigma\tilde{u}^\perp \right\Vert_{L^2}(t) \leq C Q(t).$$  
\end{lemma}
\begin{corollary}[\textit{c.f.} Corollaries~\ref{cor:source1} and \ref{cor:source}]\label{cor:3.7}
Suppose $Q(t)\leq M \eta_{\lambda}(t)$ for some $M<\infty$ and $\lambda\geq 2$. Then there exists a positive constant $C=C(M,\mathcal{M}_\Sigma,N_2 )$ such that $|x'(t)| \leq C \eta_\lambda(t)$ and $\Vert\tilde{u}^\perp  \Vert_{L^2}(t)  \leq C   \eta_\lambda(t)$, for $t\ge t_0=t_0(u,M,\lambda)$. 
\end{corollary}

Using the above, we obtain a crucial iterative Lemma. 
\begin{lemma}[\textit{c.f.} Lemma~\ref{lem-iterate2}]\label{lem:3.10} 
For a given $ \lambda \ge 2 $, suppose 
\be \label{eq-6y7}  Q(t) \le M \eta_{\lambda}(t),\ee  
{for large time $t\ge t'$.} 
Then for given $\rho\in(0,1)$ there is $C=C(M,\rho,\lambda,  \mathcal{M}_{\Sigma}, N_2)<\infty $ such that\begin{align}Q(t) \leq C   |x(t)|^{\rho }\eta_\lambda(t) \le C \eta_{\lambda+\rho }(t),
\end{align}   
{for further large time $t\ge t_0=t_0(M,u,\rho , \lambda )$.} 
\end{lemma}
Proposition~\ref{prop-neutral-dynamics-p1} follows exactly as Proposition~\ref{prop-neutral-dynamics}.
\begin{proof}[Proof of Proposition~\ref{prop-neutral-dynamics-p1}(1)] 
First by Corollary~\ref{cor:parabolic} and Lemma~\ref{lem:utup}, the solution satisfies the assumption  \eqref{eq-6y7} for $\lambda=2$. The proof follows by iterating Lemma \ref{lem:3.10}. 
\end{proof}

\begin{proof}[Proof of Proposition~\ref{prop-neutral-dynamics-p1} (2)] 
Recalling the equation $\dot  u = \mathcal{M}_{\Sigma}u + N_2(u)$, the structural assumption on $N_2$, and the convergence of $u$ to $0$, we may assume\[\Vert \dot u -\mathcal{M}_\Sigma {u} \Vert_{L^2} = \Vert N_2 (u) \Vert _{L^2} \le \eta  \Vert \mathcal{M}_\Sigma u \Vert _{L^2} , \]for $t\ge t_0=t_0(u,\eta)$, and thus
\[({1-\eta}) \Vert \mathcal{M}_\Sigma u \Vert _{L^2} \le \Vert \dot  u \Vert _{L^2}\le ({1+\eta}) \Vert \mathcal{M}_\Sigma u \Vert _{L^2} .\]
With two estimates above, the evolution of $u$ becomes a gradient flow of $\mathcal{F}_\Sigma$ with a small perturbation and the rest of the proof becomes a straightforward application of the \L ojasiewicz-Simon inequality near the origin: $\Vert \mathcal{M}_\Sigma u \Vert _{L^2}\ge c|\mathcal{F}(u)-\mathcal{F}(0)|^\theta.$
In particular, it is exactly the same as the argument from \eqref{eq-6891} to the end of the proof of Proposition \ref{prop-neutral-dynamics} $(2)$. 
\end{proof}

\section{Perturbed gradient flow in finite dimension}\label{sec:7}

Let $0\in U\subset\mathbb{R}^n$ be an open set, $g:U\to\mathbb{R}$ be an analytic function with $g(0)=0$ and $\nabla g(0)=0$. There exists $p\geq 2$ such that $g=\sum_{j\geq p} g_j$ , where $g_j$ are homogeneous polynomials of degree $j$ and $g_p\not\equiv 0$. We write $\hat{g}_p$ for the restriction of $g_p$ on the unit sphere.  From {\L}ojasiewicz gradient inequality, there exists $\theta\in (0,1)$ such that, for $|y|$ small, $|\nabla g(y)|\geq c|g(y)|^\theta$. Let $y:[1,\infty)\to U$ be a continuously differentiable curve that satisfies

\noindent \textbf{A1}: There exist $0<\alpha_2\leq 1$ and $D_1,D_2>0$ such that for all $t\geq 1$, 
\begin{equation}\label{equ:alpha12}
D_1 t^{-1}\leq |y(t)|\leq D_2 t^{-\alpha_2};
\end{equation}

\noindent \textbf{A2}: Let $\textup{Err}(t)=y'(t)+\nabla g(y(t)) .$ For all $N\in\mathbb{N}$, there exists $b_N>0$ such that for all $t\geq 1,$
\begin{equation}\label{equ:Ebound}
|\textup{Err}(t)|\leq b_N\left( |y(t)|^\rho|\nabla g(y(t))|+|y(t)|^N \right),
\end{equation}
where $\rho:=\max\{\theta,\frac{3}{4}\}$.

\begin{theorem}\label{thm:gradflow} {
There is a finite set $\mathcal{Z}=\mathcal{Z}(g)\subset \mathbb{Q}_{\ge 3}\times \mathbb{R}_{>0}$ such that if $y(t)$ satisfies \textbf{A1}--\textbf{A2}, then 

\noindent (1) The secant $ {y(t)}/{|y(t)|}$ converges to some $\theta_*\in\mathbb{S}^{n-1}$, which is a critical point of $\hat{g}_p$ with $\hat{g}_p(\theta_*)\geq 0$.

\noindent (2) There exists $(\ell_* ,\alpha_0)\in \mathcal{Z}$ such that $ \lim_{t\to\infty} t^{-1}|y(t)|^{2-\ell_*}=\alpha_0\ell_*(\ell_*-2)$. Moreover, $\ell_*=p$ holds if and only if $\hat{g}_p(\theta_*)>0$. In that case, $\alpha_0=\hat{g}_p(\theta_*)$. }
 
\end{theorem}

\begin{proof}[Proof of Theorems~\ref{thm:general_p} and \ref{thm:general_e}]
We present the proof for Theorem~\ref{thm:general_e}. The argument for Theorem~\ref{thm:general_p} is identical. In view of Proposition~\ref{prop-neutral-dynamics}, Theorem~\ref{thm:gradflow} applies with $y(t)=z(t)$ and $g=m^{-1}{(f-f(0))}$. This ensures $z(t)/ |z(t)|$ converges to $\theta_*$, a critical point of $\hat{f}_p$ with $m^{-1} \hat{f}_p(\theta_*)\geq 0$. Moreover, $ \lim_{t\to\infty} t^{1/(\ell_* -2)} |z(t)|=\big( \alpha_0  \ell_*(\ell_*-2) \big)^{-1/(\ell_*-2)}.$ 
This implies $$\lim_{t\to\infty} t^{1/(\ell_*-2)}\Vert u(t) \Vert_{L^2}  =\big( \alpha_0  \ell_*(\ell_*-2) \big)^{-1/(\ell_*-2)}.$$ 
Let $w\in\ker \mathcal{L}_{\Sigma}$ be the section corresponding to $\theta_*$ through \eqref{equ:identification}. Clearly $ u(t)/\Vert u(t) \Vert_{L^2}$ converges to $w$ in $L^2(\Sigma;\mathbf{V})$. Corollary~\ref{cor:sobolev} then ensures the convergence of $ u(t)/\Vert u(t) \Vert_{L^2}$ in $C^\infty(\Sigma;\mathbf{V})$.
\end{proof}

 In Section~\ref{sec:perturbed gradient flow}, we show that the error term $\textup{Err}(t)$ is much smaller than $\nabla g(y(t))$. Then we prove (1) and (2) in Theorem~\ref{thm:gradflow} in Sections~\ref{sec:limitofsecant} and \ref{sec:rateofconvergence} respectively. From now on we assume \textbf{A1} and \textbf{A2} hold true. Moreover, we may assume $|y(t)|\leq 1$ for all $t$. We rely on some notation and results in \cite{KMP} which we record below. 

\vspace*{0.3cm}

Let $\nabla' g:= \nabla g- \partial_rg \partial r$ be the spherical component of $\nabla g$. For any $q>0$, define $G_q$ on $U\setminus\{0\}\to\mathbb{R}$ by $G_q(y):= \frac{g(y)}{|y|^q}$. For any $\varepsilon>0$ and $r>0$, define
\begin{equation*}
W(\varepsilon,r):=\left\{ y\in U\, |\, \vert y\vert\leq r,\ g(y)\neq 0\ \textup{and}\ \varepsilon|\nabla'g(y)|\leq \left| \partial_r g (y) \right| \right\}.
\end{equation*}
For any $\omega>0$ and $q>0$, define
\begin{equation}\label{equ:def_W_4}
W(\varepsilon,r,\omega,q):=\left\{ y\in W(\varepsilon,r)\, |\, \left|1 -\frac{q g(y)}{|y|\partial_r g(y)} \right|\leq  \frac{1}{2 } |y|^{2\omega}  \right\}.
\end{equation}

{Roughly, $W(\eps,r)\subset \mathbb{R}^n$ corresponds to a good region where the spherical component of $\nabla g$ is controlled by the radial component so that $y(t)$ moves radially. $W(\eps,r,\omega,q)$ is the subregion where $\frac{r\partial_rg}{g}$ is asymptotic to $q$ and thus the gradient flow $y(t)$ moves as if it slides under a potential function of form $C |y|^q$. A limit of $\frac{r\partial_r g(y)}{g(y)}$ for $ y\to 0$ and $y\in W(\eps,r)$ is called a characteristic exponent of $g$ (see \cite{KMP}). Moreover, \cite{KMP} showed there are only finitely many characteristic exponents, which are all rational.} \begin{proposition}[Proposition 4.2 in \cite{KMP}]\label{pro:1}
For any $\varepsilon>0$, there exist $r_1=r_1(\varepsilon)>0$ and $\omega_1=\omega_1(\varepsilon)\in (0,4^{-1})$ such that {$W(\eps,r_1)$ can be decomposed into pairwise disjoint unions} \begin{equation*}
W(\varepsilon,r_1)=\bigsqcup_{\ell\in L} W(\varepsilon,r_1,\omega_1,\ell),
\end{equation*} where $L\subset{\mathbb{Q}_{>0}}$ is a finite set and is independent of $\varepsilon$. Moreover, there exists $C_1=C_1(\varepsilon)\geq 1 $ such that $\textup{for all}\ y\in W(\varepsilon,r_1,\omega_1,\ell)\ \textup{and}\ \ell\in L,$
\begin{equation}\label{equ:C_1}
\frac{1}{C_1}\leq |G_\ell(y)| \leq C_1.
\end{equation}
\end{proposition}
To simplify notation, in the sequel we write $V(\varepsilon)$ and $V(\varepsilon,\ell)$ for $W(\varepsilon,r_1(\varepsilon))$ and $W(\varepsilon,r_1(\varepsilon),\omega_1(\varepsilon),\ell)$ respectively. In the next lemma, we record direct consequences of {\L}ojasiewicz gradient inequality and Bochnak-{\L}ojasiewicz inequality in \cite[Lemma~4.3]{KMP}.
\begin{lemma}\label{lem:2}
There exist $T_1\geq 1$, $c_2, c_3>0$ such that for all $t\geq T_1$,
\begin{equation}\label{equ:Lo}
|\nabla g(y(t))|\geq c_2 |g(y(t))|^{\rho}, \text{ and}
\end{equation}
\begin{equation}\label{equ:BLo}
|y(t)||\nabla g(y(t))|\geq c_3 |g(y(t))|.
\end{equation}
\end{lemma}

\subsection{bounding the error term}\label{sec:perturbed gradient flow}
{In this subsection, the goal is to prove  $|y|^N$ in \eqref{equ:Ebound} can be absorbed to the other term and hence $\textup{Err}(t)$ is small compared to $\nabla g(y(t))$ as follows:
}

\begin{proposition}[smallness of error]\label{thm:perturb}
There exists $b_0>0$ and $T_0\geq 1$ such that for all $t\geq T_0$,
\begin{equation}\label{equ:Ebound_0}
| \textup{Err}(t)| \leq b_0  |y(t)|^\rho|\nabla g(y(t))|. 
\end{equation}
\end{proposition} 
The proof will be given once necessary lemmas are presented. 
 
\begin{lemma}\label{lem:4}
For any $d>0$ and $M\geq 4$, there exists $c =c (d,M)$ such that, for $t$ large enough, $  dt^{-M}+g(y(t))\geq 0 $. Moreover, it holds that
\begin{equation}\label{equ:L4_0}
\frac{d}{dt}\left( dt^{-M}+g(y(t)) \right)\leq -c  \left( dt^{-M}+g(y(t)) \right)^{\rho}|\nabla g(y(t))|.
\end{equation}
\begin{proof}
For a given $N\in \mathbb{N}$, by \eqref{equ:alpha12}-\eqref{equ:Ebound}, $\frac{d}{dt}( dt^{-M}+g(y(t)))$ is bounded by 
\begin{align*}
-\frac{dM}{t^{M+1}}- (1-b_N|y|^\rho) |\nabla g(y)|^2+b_N(D_2 t^{-\alpha_2})^N|\nabla g(y)|. 
\end{align*}
Let $T$ be a large constant to be fixed later. Since $y(t)\to 0$, we may choose $N=N(M,\alpha_2)$ and $T$ large so that, for $t \ge T$, there hold
\begin{align}\label{equ:L4_1}
\frac{d}{dt}\left({d}{t^{-M}}+g(y) \right) \leq &-\frac{dM}{t^{M+1}}-  \frac{|\nabla g(y)|^2}{2}+\frac{b_ND_2^N}{ t^{M/2+1 }}|\nabla g(y)|,
\end{align}
and $(dMt)^{1/2}\geq 2b_ND_2^N.$ From the Cauchy-Schwarz inequality, 
\begin{equation}\label{equ:L4_2}
dMt^{-M-1}+4^{-1}|\nabla g(y)|^2\geq 2b_ND_2^N t^{-M/2-1}|\nabla g(y)|.  
\end{equation}
Thus, 
$
\frac{d}{dt}\left( dt^{-M}+g(y) \right) \leq  - 4^{-1}  |\nabla g(y)|^2-b_ND_2^N t^{-M/2-1 }|\nabla g(y)|,
$
and this ensures $dt^{-M}+g(y(t))\geq 0$. By \eqref{equ:Lo}, $M\geq 4$ and $\rho\geq \frac{3}{4}$, 
\begin{align*} 
\frac{d}{dt}\left( dt^{-M}+g(y) \right) \leq &\left( - 4^{-1}c_2 |g(y)|^{\rho}-b_ND_2^Nd^{-\rho}  (dt^{-M})^{\rho}\right)  |\nabla g(y)|\\
\leq & -c'\max\{dt^{-M},| g(y)| \}^{\rho}|\nabla g(y)|,
\end{align*}
where $c'=\min\{4^{-1}c_2,b_ND_2^Nd^{-\rho}\}$. Using $\max\{|a|,|b|\}\geq \frac{ a+b}{2}$, \eqref{equ:L4_0} holds with $c =2^{-\rho}c'$.
\end{proof}
\end{lemma} 

\begin{lemma}\label{lem:5}
For any $M\geq 4$, there exists $C=C(M)$ such that, for $t$ large enough, $  t^{-M}+g(y(t))\geq 0 $. Moreover, 
\begin{equation*}
\int_{t}^\infty |{y'} (\tau)|\,d\tau \leq C \left(  t^{-M}+g(y(t)) \right)^{1-\rho}.
\end{equation*}
\begin{proof}
From \eqref{equ:alpha12}, \eqref{equ:Ebound} and $|y(t)|\leq 1$, $|y'(t)|\leq (1+b_N)|\nabla g(y(t))|+b_ND_2^N t^{-N\alpha_2}.$ Here $N$ is an integer to be chosen later. 
Therefore,
\begin{align*}
\int_t^\infty |y'(\tau)|\, d\tau\leq &(1+b_N) \int_t^\infty |\nabla g(y(\tau))|\, d\tau+\frac{b_ND_2^N}{N\alpha_2-1} t^{-N\alpha_2+1}.
\end{align*}
Let $h(t)=t^{-M}+g(y(t))$ and let $c=c(1,M)$ be the constant given in Lemma~\ref{lem:4}. From Lemma~\ref{lem:4}, for $t$ large enough,
\begin{align*}
\int_t^\infty |\nabla g(y(\tau))|\, d\tau\leq \frac{1}{c(1-\rho)} h(t)^{1-\rho}.
\end{align*} {Applying Lemma \ref{lem:4} for $d=1/2$, we notice} $h(t)\geq 2^{-1}t^{-M}$ for $t$ large enough. As a result, after choosing sufficiently large $N$ so that $-N\alpha_2 +1 \le -(1-\rho)M$, we conclude $\frac{b_ND_2^N}{N\alpha_2-1} t^{-N\alpha_2+1}\leq C h(t)^{1-\rho}.$ The lemma follows by combining two estimates above.  \end{proof} 
\end{lemma} 
\begin{lemma}\label{lem:6}
For any $d>0$, $M\geq 4$ and $q>0$, there exists $T_2=T_2(d,M,q)$ such that the following statement holds. Set $\varepsilon=c_3/(2q)$ where {$c_3$ the constant is from  the Bochnak-{\L}ojasiewicz inequality \eqref{equ:BLo}}. Then 
\begin{align*}
t\geq T_2\ \textup{and}\ y(t)\notin  V(\varepsilon)\Longrightarrow  \frac{d}{dt}\left( dt^{-M}+G_q(y(t))\right)\leq 0.
\end{align*}
\begin{proof} 
Set $r_1=r_1(\varepsilon)$ be given in Proposition~\ref{pro:1}. Let $N$ and $T_2$ be sufficiently large integer and number to be fixed later, respectively. In the beginning, let us assume $T_2$ is large so that $T_2\ge T_1$ and $|y(t)|\le r_1$ for $t\ge T_2$. Here $T_1\ge 1$ is the time given from Lemma~\ref{lem:2}. 
We aim to show that
\begin{equation}\label{equ:L6_0}
t\geq T_2\ \textup{and}\  \frac{d}{dt}\left( dt^{-M}+G_q(y(t))\right)>0\Longrightarrow y(t)\in  V(\varepsilon) .
\end{equation}

From now on we always assume $t\geq T_2$. 
From the Bochnak-{\L}ojasiewicz inequality \eqref{equ:BLo}, $\frac{d}{dt}\left( dt^{-M}+G_q(y ) \right)$ is bounded by 
\begin{equation}\label{equ:L6_1}
\begin{split}
 &-\frac{dM}{t^{M+1}}-\frac{|\nabla g(y )|^2}{|y |^q}+\frac{q g(y )\partial_r g(y ) }{|y |^{q+1}} +(1+q/c_3)  \frac{|\nabla g(y )|}{|y |^q}|\textup{Err}(t)|.
\end{split}
\end{equation}
By the Cauchy-Schwarz inequality and \eqref{equ:alpha12},
\begin{equation}
\ba   \frac{|\nabla g(y )|}{|y |^q}|\textup{Err}(t)|&\le b_N(|y |^\rho + |y |^{\rho})\frac{|\nabla g(y )|^2}{|y |^q} +\frac{b_N}{4}(D_2t^{-\alpha_2})^{2N-q-{\rho }}, \ea 
\end{equation} Fixing $T_2$ and $N$ sufficiently large, we can make $2b_N(1+q/c_3)|y(t)|^\rho\le 1/2$ and  $4^{-1}b_N(1+q/c_3) (D_2t^{-\alpha_2})^{2N-q-1}\le dMt^{-M-1}$ and  for $t\ge T_2$.  Combining this with the bound \eqref{equ:L6_1},
\begin{align}\label{equ:L6_4}
\frac{d}{dt}\left( dt^{-M}+G_q(y ) \right)\leq - 2^{-1}\frac{|\nabla g(y )|^2}{|y |^q}+\frac{q g(y )\partial_r g(y ) }{|y |^{q+1}} .
\end{align} 
In view of \eqref{equ:L6_4},  $\frac{d}{dt}\left( dt^{-M}+G_q(y ) \right)> 0$ implies   $\left|\partial_r g (y ) \right|> \frac{|y |}{2q |g(y )|}|\nabla g(y )|^2.$
From $T_2\geq T_1$, we may apply the Bochnak-{\L}ojasiewicz inequality \eqref{equ:BLo} to obtain
$
\left| \partial_r g(y ) \right|\geq \frac{c_3}{2q}|\nabla g(y )|\geq \varepsilon |\nabla' g(y )|.
$
Together with $|y(t)|\leq r_1(\varepsilon)$, $y(t)\in V(\varepsilon)$. 
\end{proof}
\end{lemma}
Let us fix
\begin{align} \label{equ:choices}
q_*=\max_{\ell\in L}\left\{ \frac{2}{1-\rho},\ell\right\},\ \varepsilon_*=\frac{c_3}{2q_*},\ C_{1,*}=C_1(\varepsilon_*),\ \omega_*=\omega_1(\varepsilon_*).
\end{align}
\begin{lemma}\label{lem:7}
For any $t\geq 1$, there exists $t'\geq t$ such that $y(t')\in V(\varepsilon_*)$.
\begin{proof}
Suppose the assertion fails. There exists $\bar{t}\geq 1$ such that $y(t)\notin V(\varepsilon_*)$ for all $t\geq \bar{t}$. Take $d= 1/ D_1^{q_*}$ and define $h(t)=dt^{-4}+G_{q_*}(y(t)).$ From Lemma~\ref{lem:6}, $h(t)$ is non-increasing for $t$ large enough and hence $h(t)\leq C'$ for some $C'$. Equivalently, $d|y(t)|^{q_*}t^{-4}+g(y(t))\leq C'|y(t)|^{q_*}.$ From {the lower bound in} \eqref{equ:alpha12}, \begin{align*}
t^{-4-q_*}+g(y(t)) = dD_1^{q_*}t^{-4-q_*}+g(y(t)) \leq C'|y(t)|^{q_*}.
\end{align*}
From Lemma~\ref{lem:5} with $M=4+q_*$, for $t$ large enough,
\begin{align*}
|y(t)|\leq \int_t^\infty |y'(\tau)|\, d\tau\leq C(t^{-4-q_*}+g(y(t)) )^{1-\rho}\leq  C(C')^{1-\rho} |y(t)|^{2}.  
\end{align*}
Here we used $q_*(1-\rho)\geq 2$. This contradicts $|y(t)|\to 0$.
\end{proof}
\end{lemma}
\begin{lemma}\label{lem:8}
There exists $T_3\geq 1 $ such that the following statements hold.
\begin{enumerate}
\item There does not exist $\bar{t}_2\geq\bar{t}_1\geq T_3$ and $\ell\in L$ such that
\begin{align*}
G_{\ell}(y(\bar{t}_1))= \frac{1}{2C_{1,*}} \ \textup{and}\ G_{\ell}(y(\bar{t}_2))= \frac{1}{C_{1,*}}  .
\end{align*}
\item There does not exist $\bar{t}_2\geq\bar{t}_1\geq T_3$ and $\ell\in L$ such that
\begin{align*}
G_{\ell}(y(\bar{t}_1))= C_{1,*} \ \textup{and}\ G_{\ell}(y(\bar{t}_2))= 2C_{1,*} .
\end{align*}
\end{enumerate}
\begin{proof} We will eventually fix $T_3$ later. Let us first assume  $T_3$ is large so that 
\begin{equation}\label{def:T_3-1}
\begin{split}
&\left\{ \left( \frac{1}{2C_{1,*}}|y(t)|^\ell, 2C_{1,*}   |y(t)|^\ell\right)\right\}_{\ell\in L}\ \textup{are disjoint}\ \textup{for}\ t\geq T_3, \text{ and }
\end{split}
\end{equation}
\begin{equation} \label{def:T_3-2}
 T_3 \geq T_2(1,4,\ell)\ \textup{for all}\ \ell\in L,
\end{equation}
where $T_2(1,4,\ell)$ is given in  Lemma \ref{lem:6}. Suppose (1) fails. There exists $t_2'\geq t_1'\geq T_3$ and $\ell\in L$ such that $G_{\ell}(y(  t'_1))= \frac{1}{2C_{1,*}} ,\ G_{\ell}(y( t'_2))= \frac{1}{C_{1,*}},$ and $G_{\ell}(y( t)) \in \left(    \frac{1}{2C_{1,*}},\frac{1}{C_{1,*}}\right) \ \textup{for all}\ t\in (t'_1,t'_2).$ From \eqref{def:T_3-1}, the above implies
\begin{align*}
|G_{\ell'}(y( t))|\notin \left[\frac{1}{C_{1,*}}, C_{1,*}   \right]\ \textup{for all}\ t\in (t'_1,t'_2)\ \textup{and}\ \ell'\in L.
\end{align*}
From \eqref{equ:C_1}, $y(t)\notin V(\varepsilon_*)$. From \eqref{def:T_3-2}, $\varepsilon_*\leq \frac{c_3}{2\ell}$ and Lemma~\ref{lem:6}, we infer the function $t^{-4}+G_{\ell}(y(t))$ is non-increasing for $t\in (t_1',t_2')$. In particular,
\begin{align*}
(t_2')^{-4}+ {G_{\ell}(y(t_2'))} \leq (t_1')^{-4}+ {G_{\ell}(y(t_1'))}. 
\end{align*} 
This implies $\frac{1}{2C_{1,*}}\leq (t_1')^{-4}-(t_2')^{-4}\leq T_3^{-4}.$ A contradiction follows by fixing $T_3$ large so that $T_3^{-4}\leq 3^{-1}C^{-1}_{1,*}$. This finishes the case (1).  {The other case (2) follows by a similar (even easier) argument.}
\end{proof}
\end{lemma}
\begin{lemma}\label{lem:9}
There exist  $\ell_*\in L$ and $T_4 \geq 1$ such that for all $t\geq T_4$,
\begin{align*}
G_{\ell_*}(y(t))\in \left(   \frac{1}{2C_{1,*}},2C_{1,*} \right) .
\end{align*} 
\begin{proof}
The uniqueness of $\ell_*$ readily follows if $T_4$ is chosen large so that the same condition in \eqref{def:T_3-1} holds for $t\ge T_4$. It suffices to show the existence. From Lemma~\ref{lem:7} and Proposition~\ref{pro:1}, there exists $\ell_*\in L$ and ${t}'_i\to\infty$ such that $y( {t}'_i)\in V(\varepsilon_*,{ \ell_*})$. From Proposition~\ref{pro:1}, $\frac{1}{C_{1,*}} \leq |G_{\ell_*}(y( {t}'_i))|\leq C_{1,*}$. By Lemma~\ref{lem:4} and {the lower bound in} \eqref{equ:alpha12}, $G_{\ell_*}(y( {t}'_i))>0$ for $i$ large enough. From these observations, we can find $i_0 \in \mathbb{N}$ such that for $i\ge i_0$, there hold $\frac{1}{C_{1,*}} \leq G_{\ell_*}(y( {t}'_i)) \leq C_{1,*}$  and $t'_i\geq T_3$. Set $T_4=t'_{i_0}$. Then Lemma~\ref{lem:8} implies that, for all $t\geq T_4$, $\frac{1}{2C_{1,*}}  \leq  G_{\ell_*}(y(t))\leq 2C_{1,*}$.
\end{proof}
\end{lemma}
Combining Lemma~\ref{lem:9}, \eqref{equ:C_1} and $|y(t)|\to 0$, we have the following:
\begin{corollary}\label{cor:10} 
Let $\ell_*\in L$ be {the exponent} given in Lemma~\ref{lem:9}. Then for $t$ large enough, $y(t)\in V(\varepsilon_*)$ implies $y(t)\in V(\varepsilon_*,\ell_*)$.
\end{corollary}

From now on, we fix $\ell_*\in L$ in Lemma~\ref{lem:9} and denote $G_{\ell_*}=G_*$. 

\begin{proof}[Proof of Proposition~\ref{thm:perturb}]
Set $T_0= \max\{ T_1,T_4\}$. From Lemma~\ref{lem:9} and the Bochnak-{\L}ojasiewicz inequality \eqref{equ:BLo}, for $t\geq T_0$, $
|\nabla g(y)|\geq \frac{c_3}{2C_{1,*}}|y|^{\ell_*-1}$. 
Take $N=\lceil \ell_* \rceil$. Then
\begin{align*}
|y|^N\leq |y|^{\ell_*}\leq \frac{2C_{1,*}}{c_3}|y||\nabla g(y)|\leq \frac{2C_{1,*}}{c_3}|y|^\rho|\nabla g(y)|.
\end{align*}
In view of \eqref{equ:Ebound}, \eqref{equ:Ebound_0} holds with $b_0=b_N(1+  2C_{1,*} / c_3 )$.
\end{proof}

\subsection{Convergence of secant}\label{sec:limitofsecant}

{
To explain the main idea, let us momentarily assume $\textup{Err}(t)\equiv 0$. \`A la  \L  ojasiewicz argument, the key boils down to finding a control function $h(t)\ge 0 $ that satisfies  an inequality
\bea  \label{equ:controlineq} h'(t) \le - c h(t)^\rho \frac{|\nabla 'g(y(t))|}{|y(t)|} ,\eea for some $c>0$ and $\rho \in (0,1)$ for $t$  large enough time along the solution $y(t)$.  Since it is immediate to  check $\frac{ \nabla' g(y(t))}{|y(t)|}=-\frac{d}{dt} \frac{y(t)}{|y(t)|},$
by integrating \eqref{equ:controlineq}, one obtains that the trajectory of secant has a finite length  $$\int_t^\infty \left | \frac{d}{dt} \frac{y(\tau)}{|y(\tau )|} \right| d\tau  \le \frac{c}{1-\rho}h^{1-\rho}(t) <\infty ,$$ and thereby shows the convergence of the secant. Following the argument in \cite{KMP}, we  choose $$h(t)=(G_*(y(t))-\alpha_0)+|y(t)|^{\omega_*/2}.$$
  Roughly, \eqref{equ:controlineq} means $h(t)$ does not decrease so slowly along the change of the secant.  Two terms in $h(t)$ contribute to it for different situations: the change of the term $|y|^{\omega_*/2}$ will drive the control inequality \eqref{equ:controlineq} when the radial motion of $y(t)$ dominates the spherical motion (see Lemma~\ref{lem:2.9}) and the change of normalized potential function $(G_*(y)-\alpha_0)$ will do this in the other case (see Lemma~\ref{lem:2.8}). 

}

\begin{lemma}\label{lem:Fprime}
For $t$ large enough, \begin{equation}\label{eqn:F_levol}
\frac{d}{dt} G_*(y(t))\leq -\frac{3}{4}\frac{|\nabla g(y(t))|^2}{|y(t)|^{\ell_*}}+\frac{\ell_* g(y(t))\partial_rg(y(t)) }{|y(t)|^{\ell_*+1}}.
\end{equation}
If we further assume $y(t)\in V(\varepsilon_*,\ell_*)$, then it holds that
\begin{equation}\label{eqn:F_levol_2} 
\frac{d}{dt}G_*(y(t))\leq -2^{-1}\frac{|\nabla' g(y(t))|^2}{|y(t)|^{\ell_*}}+\frac{|y(t)|^{2\omega_*}}{|y(t)|^{\ell_* }}   |\partial_r g(y(t))|^2.
\end{equation}
\begin{proof}
Observe 
\begin{align*}
\frac{d}{dt} G_*(y )=&- \frac{|\nabla g(y )|^2}{|y |^{\ell_*}}+\frac{\ell_* g(y )\partial_r g(y ) }{|y |^{\ell_*+1}}+\frac{\left\langle \textup{Err} ,\nabla g(y ) \right\rangle}{|y |^{\ell_*}}-\frac{\ell_* g(y )}{|y |^{\ell_*+1}}\left\langle \textup{Err} ,\frac{\partial}{\partial r}  \right\rangle.
\end{align*}
 Applying Proposition \ref{thm:perturb} (smallness of error) and the Bochnak-{\L}ojasiewicz inequality \eqref{equ:BLo} to the above,\begin{equation}\label{eqn:F_levol_0}
\begin{split}
\left| \frac{d}{dt} G_*(y )+\frac{|\nabla g(y )|^2}{|y |^{\ell_*}}-\frac{\ell_* g(y )\partial_rg(y ) }{|y |^{\ell_*+1}} \right|\leq  C |y |^\rho \frac{|\nabla g(y )|^2}{|y |^{\ell_*}} ,
\end{split}
\end{equation}
for $C=b_0(1+q_*/c_3)$ and for large time $t\ge \max \{T_0,T_1\}$.

The first assertion \eqref{eqn:F_levol} follows from $|y(t)|\to 0$. 
 If $y(t)\in V(\varepsilon_*,\ell_*)$,  then combining \eqref{eqn:F_levol_0} and \eqref{equ:def_W_4},
\begin{align*}
\frac{d}{dt} G_*(y )\leq & -\frac{|\nabla' g(y )|^2}{|y |^{\ell_*}}+\frac{|y |^{2\omega_*}}{2|y |^{\ell_*}}   |\partial_r g(y )|^2+C |y |^\rho \frac{|\nabla g(y )|^2}{|y |^{\ell_*}}\\
=&-(1-C|y |^\rho)\frac{|\nabla' g(y )|^2}{|y |^{\ell_*}}+\left( \frac{1}{2}+C|y |^{\rho-2\omega_*} \right) \frac{|y |^{2\omega_*}}{ |y |^{\ell_*}}   |\partial_r g(y )|^2 .
\end{align*}
Recall that $\rho\geq \frac{3}{4}$ and $\omega_*\leq \frac{1}{4}$. This ensures $\rho-2\omega_*\geq \frac{1}{4}$. Then \eqref{eqn:F_levol_2} follows by taking $t$ large enough.   
\end{proof}
\end{lemma}

\begin{lemma}\label{lem:2.2}
There exists $c >0$ such that for $t$ large enough 
\begin{align*}
y(t)\notin V(\varepsilon_*,\ell_*)\ \Longrightarrow\ \frac{d}{dt} G_*(y(t))\leq -c |y(t)|^{-1}|\nabla g(y(t))|.
\end{align*}
\begin{proof}
We assume $y(t)\notin V(\varepsilon_*,\ell_*)$ and $t$ is large enough such that Lemmas~\ref{lem:2},~\ref{lem:9}, \ref{lem:Fprime} and Corollary~\ref{cor:10} hold. From Corollary~\ref{cor:10}, $y(t) \notin V(\varepsilon_*)$. Therefore, $|\nabla g(y(t))|\geq \varepsilon_*^{-1}|\partial_r g(y(t))|$. Together with Lemma~\ref{lem:2} and the choice of $\varepsilon_*$ in \eqref{equ:choices}, 
\begin{align*}
|y||\nabla g(y)|^2=|y||\nabla g(y)|\cdot |\nabla g(y)|\geq c_3|g(y)|\cdot \varepsilon_*^{-1} |\partial_r g(y)|\geq 2\ell_* g(y)\partial_r g(y).  
\end{align*}
Combined with \eqref{eqn:F_levol}, we deduce
\begin{equation*} 
\begin{split}
\frac{d}{dt} G_*(y)\leq  -\frac{3}{4}  \frac{|\nabla g(y)|^2}{|y|^{\ell_*}}+\frac{\ell_* g(y)\partial_r g(y) }{|y|^{\ell_*+1}}\leq -4^{-1}\frac{|\nabla g(y)|^2}{|y|^{\ell_*}}.
\end{split}
\end{equation*}
Lemma~\ref{lem:2} and Lemma~\ref{lem:9} imply
$$|\nabla g(y)|\geq c_3|y|^{-1}|g(y)|\geq \frac{c_3}{2C_{1,*} }|y|^{\ell_*-1}.$$ 
Combining above, the assertion holds for $c=\frac{c_3}{8C_{1,*}}$.
\end{proof}
\end{lemma}
\begin{lemma}\label{lem:5.4}
There exists $C_{4}\geq 1$ such that for $t$ large enough
\begin{align*}
y(t)\in V(\varepsilon_*,\ell_*)\ \Longrightarrow\ |\nabla g(y(t))|,  \partial_r g(y(t))\in \left[ \frac{1}{C_{4}},C_{4} \right]|y(t)|^{\ell_*-1} .
\end{align*}
\begin{proof}
Let $t$ be large enough such that Lemma~\ref{lem:9} holds. For $y(t)\in V(\varepsilon_*,\ell_*)$, $|\nabla g(y(t))|$ is comparable to $|\partial_r g(y(t))|$. Moreover, \eqref{equ:def_W_4} implies $\partial_r g(y(t))$ is comparable to $ g(y(t))/|y(t)|$, which is comparable to $|y(t)|^{\ell_*-1}$ in view of Lemma~\ref{lem:9}.
\end{proof}
\end{lemma}
\begin{lemma}\label{lem:limFstar}
{The} function $G_*(y(t))+|y(t)|^{\omega_*}$ is non-increasing for $t$ large.
\begin{proof}
First, we consider the case $y(t)\in V(\varepsilon_*,\ell_*)$. Assume $t$ is large enough such that Proposition~\ref{thm:perturb} (smallness of error), Lemmas~\ref{lem:Fprime} and~\ref{lem:5.4} hold. Then \eqref{eqn:F_levol_2} implies 
\begin{align*}
\frac{d}{dt}G_*(y )  \leq &\frac{|y |^{2\omega_*}}{|y |^{\ell_*}}    |\partial_r g(y )|^2 \leq \frac{|y |^{2\omega_*}}{|y |^{\ell_*}} |\nabla g(y )|^2\leq   C_{4} |y |^{2\omega_*-1} |\nabla g(y )|.
\end{align*}
We compute
\begin{align*}
\frac{d}{dt}  |y |^{\omega_*} =& \omega_*|y |^{\omega_*-1}(-\partial_r g(y )+\left\langle \textup{Err} ,\partial_r \right\rangle)
\leq -2^{-1}\omega_* C_{4}^{-2} |y |^{\omega_*-1} |\nabla g(y )|.
\end{align*}
As $|y(t)|^{\omega_*}$ becomes arbitrarily small for large time, the assertion follows. 

Next, we turn to the case  $y(t)\notin V(\varepsilon_*,\ell_*)$. From Lemma~\ref{lem:2.2}, 
\begin{align*}
\frac{d}{dt}G_*(y)\leq -c |y|^{-1}|\nabla g(y)|.
\end{align*}
Together with
\begin{align*}
\left| \frac{d}{dt} |y|^{\omega_*}  \right|=\left|  \omega_*|y|^{\omega_*-1}(-\partial_r g(y)+\left\langle \textup{Err},\partial_r \right\rangle) \right| \leq 2\omega_* |y|^{\omega_*-1}|\nabla g(y)|,
\end{align*}
we deduce the assertion holds for $t$ large as $|y|^{\omega_*}$ becomes arbitrarily small. 
\end{proof}
\end{lemma} 
Lemma~\ref{lem:limFstar} ensures the limit of $G_*(t)$ exists. Let
\begin{align}\label{def:alpha_0}
\alpha_0:=\lim_{t\to\infty} G_*(y(t)).
\end{align}
In view of Lemma~\ref{lem:9}, $\alpha_0>0$, which immediately implies 
\begin{equation}\label{equ:ellandp}
\ell_*\geq p,
\end{equation}
where $p$ is the vanishing degree of $g$ at the origin.  {Following \cite[Section 5]{KMP}, we say $\alpha\in \mathbb{R}$ is an asymptotic critical value of $F$ at the origin if there exists a sequence $x_i\to 0$ and $x_i \in U$ such that
\begin{enumerate}
\item $|x_i||\nabla' F(x_i)|\to 0$ as $i\to \infty$, and 
\item $F(x_i)\to \alpha$ as $i\to \infty$. 
\end{enumerate}
Here $F$ is a  $C^1$-subanalytic function defined on an open subanalytic set $U\subset \mathbb{R}^n$ such that $0 \in \bar U$. Note that $G_*$ is an analytic function defined on $\mathbb{R}^n \setminus \{0\}$ and it satisfies the assumption of $F$.}

\begin{lemma}
$\alpha_0$ is an asymptotic critical value of $G_*(y)$.
\begin{proof}
Assume the assertion fails. Then there exists $c'>0$ such that $$|y ||\nabla 'G_*(y )|= |y |^{1-\ell_*} |\nabla 'g(y )| \ge c'$$ for large $t$. We first show that there exists $c>0$ such that for large $t$, 
\begin{align}\label{equ:5.60}
|\nabla' g(y)|\geq c|\partial_r g(y)|.
\end{align}
If $y(t)\notin V(\varepsilon_*,\ell_*)$, then by Corollary~\ref{cor:10}, \eqref{equ:5.60} holds for $c\leq \varepsilon_*^{-1}$. If $y(t)\in V(\varepsilon_*,\ell _*)$, $|y|^{1-\ell_*}|\partial _r g(y)| \le  C_{4}  $ from Lemma~\ref{lem:5.4}. Hence \eqref{equ:5.60} is satisfied for $c\leq   c'/C_4$.

Next, there exists $\tilde{c}>0$ such that for $t$ large enough,
\begin{equation}\label{equ:5.6}
\frac{d}{dt} G_*(y )\leq -\tilde{c} |y |^{-1}|\nabla g(y )|.
\end{equation}
If $y(t)\notin V(\varepsilon_*,\ell_*)$, then \eqref{equ:5.6} follows from Lemma~\ref{lem:2.2}.  Suppose $y(t)\in V(\varepsilon_*,\ell_*)$. Then \eqref{equ:5.6} follows by combining \eqref{eqn:F_levol_2}, \eqref{equ:5.60} and Lemma~\ref{lem:5.4}. We claim, however, that $\int_{t}^\infty |y(s)|^{-1}|\nabla g(y(s))|ds=\infty$, and thus yielding $\alpha_0-G_*(y(t))\leq -\infty$. This is impossible and hence $\alpha_0$ is an asymptotic critical value.

Finally, the proof of claim follows by the perturbed gradient flow structure. By \textbf{A1} and \textbf{A2}, we have that, for sufficiently large $s\ge s_0$,  
\[|\nabla g(y(s))|\ge \frac12 |y'(s)|- |y(s)|s^{-2} \ge \frac12 ||y|'(s)|-|y(s)|s^{-2}  .\] 
Observing $\int_{s_0}^{\infty} |y|^{-1} ||y|'|ds \ge  |\int _{s_0}^\infty |y|^{-1} |y|' ds | = \infty $ since $y(s)\to 0$, and $\int_{s_0}^{\infty}  s^{-2}ds <\infty $, the claim follows.
\end{proof} 
\end{lemma}
The following is Proposition~5.3 in \cite{KMP}.
\begin{proposition}\label{pro:subanalytic}
There exists $e_1, e_2>0$ and $\bar{\rho}\in [2^{-1},1)$ such that $|G_*(y)-\alpha_0|\leq e_1$ and $|\nabla' G_*(y)|\geq |y|^{-\omega_*/2}|\partial_r G_*(y)|$ imply $|y||\nabla G_* (y)|\geq e_2|G_*(y)-\alpha_0|^{\bar{\rho}}$.
\end{proposition}
In \cite[Proposition~5.3]{KMP}, $\bar{\rho}\in (0,1)$ but here we require $\bar{\rho}\geq 2^{-1}$ for notational convenience. For $\eta>0$, define
\begin{align*}
V(\varepsilon,\eta,\ell):=\{y\in V(\varepsilon,\ell)\, |\, \vert y \vert ^{-\eta}|\nabla' g(y)|\leq |\partial_r g(y)|\}.
\end{align*} 
\begin{lemma}\label{lem:2.6}
Let $\bar{\rho}$ be the exponent given in Proposition~\ref{pro:subanalytic}. There exists $c>0$ such that, for $t$ large enough, 
\begin{align*}
y(t)\notin V(\varepsilon_*,2^{-1}\omega_*,\ell_*)\Longrightarrow \frac{d}{dt} G_*(y(t))\leq - c|G_*(y(t))-\alpha_0|^{\bar{\rho}}\cdot\frac{|\nabla'g(y(t))|}{|y(t)|}.
\end{align*}
\begin{proof}
First, consider the case $y(t)\notin V(\varepsilon_*,\ell_*)$. Let $c'>0$ be the constant in Lemma~\ref{lem:2.2}. Then $
\frac{d}{dt} G_*(y)\leq -c'{|\nabla g(y)|}/{|y|}\leq -c'{|\nabla'g(y)|}/{|y|}.$
By requiring $t$ is large so that $|G_*(y)-\alpha_0|\leq 1$, the assertion holds for $c\leq c'$.

Next, consider the case $y(t)\in V(\varepsilon_*,\ell_*)\setminus V(\varepsilon_*,2^{-1}\omega_*,\ell_*)$. Since $y\notin V(\varepsilon_*,2^{-1}\omega_*,\ell_*)$,  $|\nabla'g(y)|>|y|^{\omega_*/2}|\partial_r g(y)|$. By Lemma~\ref{lem:5.4}, $|\partial_r g(y)|\geq C_{4}^{-1}|y|^{\ell_*-1}$ for $t$ large. Therefore, $|\nabla' G_*(y)|={|\nabla'g(y)|}/{|y|^{\ell_*}}\geq C_{4}^{-1} |y|^{\omega_*/2-1}.$
Combining \eqref{equ:def_W_4} and Lemma~\ref{lem:5.4}, 
\begin{align*}
\left|\partial_r G_*(y )  \right|= \left| \frac{\partial_r g(y )}{|y |^{\ell_*}}\left( 1-\frac{\ell_* g(y )}{|y |\partial_r g(y )} \right) \right|\leq 2^{-1}C_{4}|y |^{2\omega_*-1}.
\end{align*}
By two estimates above, for $t$ large enough,
$
|\nabla' G_*(y )|\geq |y |^{-\omega_*/2}\left|\partial_r G_*(y )  \right|$.
From Proposition~\ref{pro:subanalytic}, 
\begin{align}\label{equ:L13_1}
|y(t)||\nabla G_* (y(t))|\geq e_2|G_*(y(t))-\alpha_0|^{\bar{\rho}}.
\end{align}
By a direct computation, for $y\in V(\varepsilon_*,\ell_*)\setminus V(\varepsilon_*,2^{-1}\omega_*,\ell_*)$,
\begin{align*}
|\nabla G_*(y)|\leq &\frac{|\nabla' g(y)|}{|y|^{\ell_*}}+\left|1-\frac{\ell_*g(y)}{|y|\partial_r g(y)}  \right| \frac{|\partial_r g(y)|}{|y|^{\ell_*}} 
\leq  (1+2^{-1}|y|^{\frac{3\omega_*}{2}})\frac{|\nabla' g(y)|}{|y|^{\ell_*}}
\end{align*}
Therefore, for $t$ large enough,
\begin{align}\label{equ:L13_2}
|\nabla G_*(y)|\leq 2 \frac{|\nabla' g(y)|}{|y|^{\ell_*}}.
\end{align}
Also, combining \eqref{eqn:F_levol_2} and $|\nabla'g(y)|>|y|^{\omega_*/2}|\partial_r g(y)|$, 
\begin{align}\label{equ:L13_3}
\frac{d}{dt}G_*(y)\leq  &(-2^{-1}+ |y|^{\omega_*} )\frac{|\nabla' g(y)|^2}{|y|^{\ell_*}}\leq - \frac 13 \frac{|\nabla' g(y)|^2}{|y|^{\ell_*}}, 
\end{align}
for $t$ large enough. Combining \eqref{equ:L13_1}, \eqref{equ:L13_2} and \eqref{equ:L13_3}, \bea \label{eq-secondpart}
\frac{d}{dt}G_*(y )\leq &-\frac16  |y ||\nabla G_*(y )|\cdot\frac{|\nabla' g(y )| }{|y | }\leq  -\frac{e_2}6 |G_*(y )-\alpha_0|^{\bar{\rho}}\cdot\frac{|\nabla' g(y )| }{|y | }.
\eea 
Therefore, the assertion holds for $c\leq e_2/6$.
\end{proof}
\end{lemma}
\begin{lemma}\label{lem:2.7}
There exists $c>0$ such that for $t$ large enough,
\begin{align*}
y(t)\in V(\varepsilon_*,4^{-1}\omega_*,\ell_*)\Longrightarrow \frac{d}{dt}  |y(t)|^{\omega_*/2} \leq -c(|y(t)|^{\omega_*/2})^{\bar{\rho}}\frac{|\nabla'g(y(t))|}{|y(t)|}.
\end{align*}
\begin{proof} By Lemma~\ref{lem:5.4}, for large $t$, $\partial_r g(y(t))>0$. Hence
\begin{align*}
\frac{d}{dt} |y(t)|^{\frac{\omega_*}{2}} 
\leq &- \frac12 \omega_*|y(t)|^{\frac{\omega_*}{2}-1} \left( |\partial_r g(y(t))|- b_0|y(t)|^\rho|\nabla g(y(t))| \right).
\end{align*}
Since $y(t)\in V(\varepsilon_*,\ell_*)$, $|\nabla g(y)|\leq   (1+\varepsilon_*^{-1})|\partial_r g(y)|.$ Thus, for large $t$, $\frac{d}{dt} |y|^{\frac{\omega_*}{2}} \leq   - \frac 14\omega_*|y|^{\frac{\omega_*}{2}-1}   |\partial_r g(y)|.$ Together with $|\partial_r g(y)|\geq |y|^{-\frac{\omega_*}{4}}|\nabla' g(y)|$, there holds 

\begin{align*}
\frac{d}{dt} |y|^{\frac{\omega_*}{2}} \leq& - \frac14 \omega_*|y|^{\frac{\omega_*}{4}} \frac{|\nabla'g(y)|}{|y|}\\
=&- \frac14 \omega_*(|y|^{\omega_*/2})^{\bar{\rho}}|y|^{-\omega_*(\bar{\rho}-\frac{1}{2})/2} \frac{|\nabla'g(y)|}{|y|}.
\end{align*}
Since $\bar{\rho}\geq\frac{1}{2}$ and $|y|\leq 1$, we have $|y|^{-\omega_*(\bar{\rho}-\frac{1}{2})/2}\geq 1$. 
The assertion then holds for $c=4^{-1}\omega_*$.

\end{proof}
\end{lemma}
Define the control function
\begin{align*}
h(t):=(G_*(y(t))-\alpha_0)+|y(t)|^{\omega_*/2}.
\end{align*}
\begin{lemma}[control under non-radial motion]\label{lem:2.8}
There exists $c>0$ such that for $t$ large enough,
\begin{align*}
y(t)\notin V(\varepsilon_*,\omega_*/2,\ell_*)\Longrightarrow h'(t)\leq -c|h(t)|^{\bar{\rho}}\cdot \frac{|\nabla'g(y(t))|}{|y(t)|}.
\end{align*}
\begin{proof}
We first discuss the case $y(t)\notin V(\varepsilon_*,\ell_*).$ Let $c'$ be the constant in Lemma~\ref{lem:2.2}. Then from Lemma~\ref{lem:2.2}, for $t$ large enough,
\begin{align}\label{equ:L16_1}
\frac{d}{dt} G_*(y )\leq -c'|y |^{-1} |\nabla g(y )|.
\end{align}
From Proposition \ref{thm:perturb} (smallness of error), for $t$ large enough, 
\begin{equation}\label{equ:L16_2}
\begin{split}
\frac{d}{dt} |y |^{\omega_*/2}  \leq &   2^{-1}\omega_*|y |^{\omega_*/2-1}\cdot\left( |\partial_r g(y )|+b_0|y |^\rho|\nabla g(y )|  \right)\\
\leq &  \omega_*|y |^{\omega_*/2-1 }  |\nabla g(y )|.
\end{split}
\end{equation}
For $t$ large enough such that $|h(t)|\leq 1$ and \eqref{equ:L16_1} and \eqref{equ:L16_2} hold,  
\begin{align*}
h' \leq -2^{-1}c'|y |^{-1} |\nabla g(y )| \leq -2^{-1}c'|h |^{\bar{\rho}} \cdot\frac{|\nabla' g(y )|}{|y |}.
\end{align*}
Hence the assertion holds by requiring $c\leq 2^{-1}c'$.

Next, we consider the case $y(t)\in V(\varepsilon_*,\ell_*)\setminus V(\varepsilon_*,4^{-1}\omega_*,\ell_*)$. Note that $$V(\varepsilon_*,\ell_*)\setminus V(\varepsilon_*,4^{-1}\omega_*,\ell_*)\subset V(\varepsilon_*,\ell_*)\setminus V(\varepsilon_*,2^{-1}\omega_*,\ell_*).$$ From the first inequality of \eqref{eq-secondpart} in the proof of Lemma~\ref{lem:2.6},
\begin{align*}
\frac{d}{dt} G_*(y) \leq & -6^{-1}|y||\nabla' G_*(y)| \frac{|\nabla'g(y)|}{|y|}=-6^{-1}|y|^{1-\ell_*}|\nabla' g(y)| \frac{|\nabla'g(y)|}{|y|}.
\end{align*}
Because $y(t)\notin V(\varepsilon_*,4^{-1}\omega_*,\ell_*)$, $|\nabla'g(y )|\geq |y |^{\frac{\omega_*}{4}}|\partial_r g(y )|$. Therefore,
$
\frac{d}{dt} G_*(y )\leq -6^{-1}|y |^{1-\ell_*+\frac{\omega_*}{4}}|\partial_r g(y )| \frac{|\nabla'g(y )|}{|y |}.
$
From Lemma~\ref{lem:5.4}, for $t$ large enough,
\begin{equation}\label{equ:L16_3}
\frac{d}{dt} G_*(y )\leq -(6C_{4})^{-1} |y |^{\frac{\omega_*}{4}} \frac{|\nabla'g(y )|}{|y |}.
\end{equation}
On the other hand,  from Lemma~\ref{lem:2.6}, for $t$ large enough,
\begin{align}\label{equ:L16_4}
\frac{d}{dt} G_*(y )\leq -\tilde{c} |G_*(y )-\alpha_0|^{\bar{\rho}}\cdot\frac{|\nabla'g(y )|}{|y |}.
\end{align} 
Here $\tilde{c}$ is the constant in Lemma~\ref{lem:2.6}. Next,
from Lemma~\ref{lem:5.4}, $\partial_rg(y(t))>0$ for $t$ large enough. Hence 
\begin{align*}
\frac{d}{dt} |y |^{\frac{\omega_*}{2}}  \leq &  -2^{-1}\omega_*|y |^{\frac{\omega_*}{2}-1}\left( |\partial_r g(y )|-b_0|y |^\rho|\nabla g(y )|  \right).
\end{align*}
Because $y(t)\in V(\varepsilon_*,\ell_*)\subset V(\varepsilon_*)$, $|\nabla g(y )|\leq (1+\varepsilon_*^{-1})|\partial_r g(y )|$. Thus, for $t$ large enough, $\frac{d}{dt} |y |^{\frac{\omega_*}{2}}<0$.
Combining this with \eqref{equ:L16_3} and \eqref{equ:L16_4},
\begin{align*}
h' \leq -\left(2^{-1}\tilde{c} |G_*(y )-\alpha_0|^{\bar{\rho}}+ (12C_{4})^{-1} |y|^{\frac{\omega_*}{4}}\right)\frac{|\nabla'g(y )|}{|y |}.
\end{align*}
Since $\bar{\rho}\geq 2^{-1}$, $
2(|y |^{\frac{\omega_*}{4}}+|G_*(y )-\alpha_0|^{\bar{\rho}})\geq |h |^{\bar{\rho}}$.
Then the assertion holds for small $c>0$.

Lastly, we suppose $y(t)\in V(\varepsilon_*,4^{-1}\omega_*,\ell_*)\setminus V(\varepsilon_*,2^{-1}\omega_*,\ell_*)$. In view of  Lemmas~\ref{lem:2.6} and \ref{lem:2.7}, for $t$ large enough,
\begin{align*}
h' \leq -c( |G_*(y )-\alpha_0|^{\bar{\rho}}+(|y |^{\omega_*/2})^{\bar{\rho}})\cdot \frac{|\nabla'g(y )|}{|y |}, 
\end{align*}
for some $c>0$ and then the assertion follows. 
\end{proof}
\end{lemma}
\begin{lemma}[control under radial motion]\label{lem:2.9}
There exists $c>0$ such that for $t$ large enough,
\begin{align*}
y(t)\in V(\varepsilon_*,\omega_*/2,\ell_*)\Longrightarrow h'(t)\leq -c|h(t)|^{\bar{\rho}}\cdot \frac{|\nabla'g(y(t))|}{|y(t)|}. 
\end{align*}
\begin{proof}
From \eqref{eqn:F_levol_2} and Lemma~\ref{lem:5.4}, for $t$ large enough,
\begin{equation}\label{equ:L17_1}
\frac{d}{dt}G_*(y )\leq  |y |^{2\omega_*-\ell_*}|\partial_r g(y )|^2\leq  C_{4}|y |^{2\omega_*-1}|\partial_r g(y )|.
\end{equation}
From Proposition~\ref{thm:perturb} (smallness of error), for $t$ large enough, 
\begin{align*}
\frac{d}{dt} |y |^{\frac{\omega_*}{2}}\leq &  2^{-1}\omega_*|y |^{\frac{\omega_*}{2}-1}\left( -|\partial_r g(y )|+b_0|y |^\rho|\nabla g(y )|  \right).
\end{align*}
Since $y(t)\in V(\varepsilon_*,\omega_*/2,\ell_*)\subset V(\varepsilon_*)$, $|\nabla g(y )|\leq (1+ \varepsilon_*^{-1})|\partial_r g(y )|$. Hence for $t$ large enough,
\begin{equation}\label{equ:L17_2}
\begin{split}
\frac{d}{dt} |y |^{\frac{\omega_*}{2}} \leq& -4^{-1}\omega_*|y |^{\frac{\omega_*}{2}-1}|\partial_r g(y )| .
\end{split}
\end{equation} 
Combining \eqref{equ:L17_1} and \eqref{equ:L17_2}, we get $h' \leq -8^{-1}\omega_*|y |^{\frac{\omega_*}{2}-1}|\partial_r g(y )|.$ Since $y(t)\in V(\varepsilon_*,\omega_*/2,\ell_*)$, $|\nabla'g(y )|\leq |y |^{\frac{\omega_*}{2}}|\partial_r g(y )|$. Therefore, when $t$ is large enough such that $|h(t)|\leq 1$, 
\begin{align*}
h'  \leq -8^{-1}\omega_*   \frac{|\nabla'g(y )|}{|y |}\leq -8^{-1}\omega_* |h |^{\bar{\rho}}  \frac{|\nabla'g(y )|}{|y |},
\end{align*}
and the assertion follows.
\end{proof}
\end{lemma}
{Now, we prove Theorem~\ref{thm:gradflow} (1) and (2) separately.}
\begin{proof}[Proof of Theorem~\ref{thm:gradflow} (1)]
From Proposition~\ref{thm:perturb} (smallness of error), for $t$ large enough,
\begin{align*}
|y|\left|\frac{d}{dt}\left( \frac{y}{|y|} \right)\right|= |y'-\left\langle y',\partial_r \right\rangle\partial_r|
\leq &|\nabla'g(y)|+b_0|y|^\rho|\nabla g(y)|.
\end{align*}
Hence
\begin{align*}
\int_t^\infty \left|\frac{d}{d\tau }\left( \frac{y(\tau )}{|y(\tau)|} \right)\right|\, d\tau\leq \int_t^\infty \frac{|\nabla'g(y(\tau))|}{|y(\tau)|}\, d\tau+ b_0 \int_t^\infty   \frac{|\nabla g(y(\tau))|}{|y(\tau)|^{1-\rho}}\, d\tau.
\end{align*}
Momentarily, we assume the both claims, namely $\int_t^\infty \frac{|\nabla'g(y(\tau))|}{|y(\tau)|}\, d\tau$ and $\int_t^\infty   \frac{|\nabla g(y(\tau))|}{|y(\tau)|^{1-\rho}}\, d\tau$ are finite. Note the claims imply that $y(t)/|y(t)|$, as a trajectory in $\mathbb{S}^{n-1}$, has a finite length, and hence it has a limit $\theta_*\in\mathbb{S}^{n-1}$.  To proceed with the remaining assertion, recall the expansion $g=\sum_{j\geq p}g_j$ and that $\hat{g}_p$ is the restriction of $g_p$ on the unit sphere. Then the first claim implies $\theta_*$ is a critical point of $\hat{g}_p$. From \eqref{def:alpha_0} and \eqref{equ:ellandp},
\begin{align}\label{equ:alpha_0equg}
\hat{g}_p(\theta_*)=\lim_{t\to\infty}  |y(t)|^{\ell_*-p} G_*(y(t))=\lim_{t\to\infty}  |y(t)|^{\ell_*-p} \alpha_0\geq 0.
\end{align}

It remains to prove the claims. From Lemma~\ref{lem:2.8} and Lemma~\ref{lem:2.9}, for $t$ large enough $\frac{|\nabla'g(y(\tau))|}{|y(\tau)|}\leq -c^{-1}(1-\bar{\rho})^{-1}\frac{d}{d\tau}(h(\tau))^{1-\bar{\rho}}.
$ Integrating the above yields the first claim. Next, combining Proposition~\ref{thm:perturb} (smallness of error) and \eqref{equ:Lo} in Lemma~\ref{lem:2}, for $\tau$ large enough, 
\begin{align*}
\frac{d}{d\tau} g(y(\tau))=&-|\nabla g(y(\tau))|^2+\left\langle \textup{Err}(\tau),\nabla g(y(\tau)) \right\rangle \leq -\frac{1}{2}|\nabla g(y(\tau))|^2\\
\leq& - \frac{c_2}{2}|g(y(\tau))|^\rho |\nabla g(y(\tau))|= - \frac{c_2}{2}|g(y(\tau))|^\rho |y(\tau)|^{1-\rho} \frac{|  \nabla g(y(\tau))|}{|y(\tau)|^{1-\rho}}.
\end{align*}
Since $g(0)=0$, $\nabla g(0)=0$ and $|y(\tau)|\to 0$, $|g(y(\tau))|\leq C|y(\tau)|^2$ for large $\tau$. \begin{align*}
\frac{d}{d\tau} g(y(\tau))\leq -\frac{c_2}{2C^{(1-\rho)/2}} |g(y(\tau))|^{(1+\rho)/2}\frac{|  \nabla g(y(\tau))|}{|y(\tau)|^{1-\rho}},
\end{align*} and the second claim follows.
\end{proof}

We prove the remaining statement (2) in a separate subsection. 
\subsection{Rate of convergence -- Proof of Theorem~\ref{thm:gradflow} (2)}\label{sec:rateofconvergence}
 Recall the set of characteristic exponents of $g$, namely $L\subset \mathbb{Q}_{>0}$, is finite (Proposition \ref{pro:1}). Moreover, the set of asymptotic critical values of $G_\ell = g/|x|^{\ell}$ is finite \cite[Proposition 5.1]{KMP}. Define finite set $\mathcal{Z}\subset \mathbb{Q}_{\ge 3}\times \mathbb{R}_{>0}$ by
\[\mathcal{Z} :=\{(\ell ,\alpha )\,:\, \ell \in L, \, \ell \ge 3, \text{and }\alpha >0\text{ is an asymptotic critical value of } G_{\ell}\} .\] 
 
 We start by proving $\ell_*>2$. Suppose that $\ell_*\leq 2$. Recalling  $\ell_*\geq p$ in \eqref{equ:ellandp}. Since $p\ge 2$, where $p$ is the vanishing degree of $g$ at the origin, $\ell_*=p=2$. By \eqref{def:alpha_0}, $\hat{g}_2(\theta_*)=\alpha_0>0$, where $\theta_*$ is the limit of $y(t)/|y(t)|$. Using \eqref{equ:Ebound_0}, one readily can check {$|y|'\le -\frac{\alpha_0}{2}|y|$ for large time and hence $|y|$} decays exponentially. This contradicts Assumption \textbf{A1} and hence $\ell_*>2$. Next, by \eqref{equ:alpha_0equg}, $\ell_*=p$ if and only if $\hat{g}_p(\theta_*)>0$ and in that case $\hat{g}_p(\theta_*)=\alpha_0$. It remains to show that 
\begin{equation}\label{equ:decayrate}
\lim_{t\to\infty} t^{-1}|y(t)|^{2-\ell_*}=\alpha_0\ell_*(\ell_*-2).
\end{equation}
In view of Assumption \textbf{A1}, this implies $\ell_*\geq 3$ and hence $(\ell_*,\alpha_0)\in \mathcal{Z}$.

Denoting   ${\kappa}(t):=\frac{|y(t)|^{2-\ell_* }}{\alpha_0 \ell_* { (\ell_*-2)} } $, \eqref{equ:decayrate} is equivalent to $\lim_{t\to\infty}\frac{\kappa(t)}{t}=1$. It  reduces to show that ${\kappa}(t)$ satisfies the assumptions of the following lemma. 
\begin{lemma} \label{lem-crutialconvrate} Let ${\kappa}(t)$ be a positive continuously differentiable function defined on $[0,\infty)$ and $\lim_{t\to \infty} {\kappa}(t)=\infty$. Suppose there exist $M_1,M_2>0$ and {an increasing sequence $ t_i\to\infty $ that satisfies ${\kappa}(t_{i+1})=2{\kappa}(t_i)$} with the following significances: for every given $\eps>0$, for $i$ large enough there hold 
\begin{align}  \label{eq-lemrate5} &(1-M_1\eps){\kappa}(t_i)\le  t_{i+1}-t_i \le (1+M_1\eps){\kappa}(t_i), \text{ and}\\ \label{eq-lemrate6} &\int_{[t_i,t_{i+1}]\cap \{ t\, |\, \vert{\kappa}'(t)-1\vert >\eps \}} |{\kappa}'(t)| dt \le \eps M_2 {\kappa}(t_i).   \end{align} 
Then $\lim_{t\to \infty}   {\kappa}(t)/t  =1 . $

\begin{proof}[Proof of Lemma \ref{lem-crutialconvrate}]
Adding \eqref{eq-lemrate5} for $i=1$ to $n$, 
\be (1-M_1\eps) (2^{n}-1)\kappa(t_1)\le t_{n+1}-t_1 \le (1+M_1\eps) (2^{n}-1)\kappa(t_1).\ee 
Since $2^n\kappa(t_{1})=\kappa(t_{n+1})$, 
one immediately infers that $\lim_{i\to\infty} \kappa(t_i)/t_i=1$. For every other $t\in(t_{j},t_{j+1})$, we may use \eqref{eq-lemrate5}, \eqref{eq-lemrate6} to obtain that there is $C>0$ depending on $M_1$, $M_2$ but uniform in $\eps$ and $j$ such that $|{\kappa}(t)-{\kappa}(t_j) - (t-t_j)|\le C \eps {\kappa}(t_j) .$ 
This shows $\lim_{t\to\infty}\kappa(t)/t=1$.  
\end{proof}
\end{lemma}

{Roughly, two lemmas below show excess motion in the spherical part (and radial motion in reversed direction) is negligible compared to the radial motion toward the limit and they will be used in the proof of  \eqref{equ:decayrate}, presented at the end of this subsection.} Let $\sigma(t)$ be the arclength of the trajectory $y(\tau)$ from time $t$ to infinity 
\begin{align*}
\sigma(t):=\int_t^\infty |y'(\tau)|\, d\tau.
\end{align*}

\begin{lemma}\label{lem:arclength}
For any $\delta>0$, $|y(t)|\le \sigma(t) \le (1+\delta)|y(t)|$ for $t$ large enough.
\begin{proof}
It is clear that $|y(t)|\leq \sigma(t)$ and we will focus on proving $\sigma(t) \le (1+\delta)|y(t)|$. It is shown in the proof of \cite[Corollary 6.5]{KMP} that given $\tilde{C}>\tilde{c}>0$ and $\varepsilon>0$, there exists $r>0$ such that
\begin{align*}
|y|\leq r, \tilde{c}\leq G_*(y)\leq \tilde{C}\Longrightarrow |\nabla g(y)|\geq (1-\varepsilon)\ell_*|y|^{-1} g(y) .
\end{align*}
Fix $\varepsilon\in (0,1)$ such that $(1-\varepsilon)^{-5}=1+\delta$. Since $G_*(y(t))$ converges to $\alpha_0$, for $t$ large enough, $|\nabla g(y(t))|\geq (1-\varepsilon) \ell_* |y(t)|^{-1} g(y(t)).$ In view of Proposition~\ref{thm:perturb} (smallness of error), $\frac{d}{dt}g(y(t))\leq -(1-\varepsilon)|\nabla g(y(t))|^2.$ Since $G_*(y(t))$ converges to $\alpha_0$, for $t$ large enough,
\begin{align*}
(1-\varepsilon) \alpha_0^{1/\ell_*} g(y(t))^{-1/\ell_*} \leq  |y(t)|^{-1}\leq (1-\varepsilon)^{-1} \alpha_0^{1/\ell_*} g(y(t))^{-1/\ell_*}.
\end{align*}
Combining the above, we get
\begin{align*}
\frac{d}{dt}g(y)\leq -(1-\varepsilon)|\nabla g(y)|^2\leq -(1-\varepsilon)^3\alpha_0^{1/\ell_*}\ell_*g(y)^{1-1/\ell_*}|\nabla g(y)|.
\end{align*}
Equivalently, $
|\nabla g(y)|\leq -(1-\varepsilon)^{-3}\alpha_0^{-1/\ell_*} \frac{d}{dt} g(y)^{1/\ell_*}$.
Therefore,
\begin{align*}
\int_t^{\infty}|\nabla g(y(\tau))|\, d\tau\leq (1-\varepsilon)^{-3}\alpha_0^{-1/\ell_*} g(y(t))^{1/\ell_*}\leq (1-\varepsilon)^{-4}|y(t)|.
\end{align*}By Proposition~\ref{thm:perturb} (smallness of error),  $|y'(t)|\leq (1-\varepsilon)^{-1} |\nabla g(y(t))|$ for large $t$. As a result, $\sigma(t)= \int_t^\infty |y'(\tau)|\, d\tau \leq (1-\varepsilon)^{-5}|y(t)|=(1+\delta)|y(t)|.$ 
\end{proof}
\end{lemma}
\begin{lemma}\label{lem:bad}
For any $\delta>0$, for $t$ large enough,
\begin{align}
\int_{[t,\infty)\cap\{ \tau\, |\,y(\tau)\notin V(\varepsilon_*) \}} &|\nabla g(y(\tau))|\, d\tau\leq \delta |y(t)|.\label{equ:bad1} 
\end{align}
\begin{proof}
From Proposition~\ref{thm:perturb} (smallness of error), $|\nabla g(y(\tau))|$ and $| y'(\tau) |$ are comparable. Hence it suffices to prove 
\begin{equation}\label{equ:bad2}
\int_{[t,\infty)\cap\{ \tau\, |\,y(\tau)\notin V(\varepsilon_*) \}}  | y'(\tau) |\, d\tau\leq \delta |y(t)|.
\end{equation}
Denote
\begin{align*}
&\overline{A}(t)= \int_{[t,\infty)\cap\{  y(\tau)\in V(\varepsilon_*) \}} \left|y'(\tau) \right| d\tau,\ \overline{B}(t)= \int_{[t,\infty)\cap\{y(\tau)\notin V(\varepsilon_*) \}} \left|y'(\tau) \right| d\tau,\\
& {A}(t)= \int_{[t,\infty)\cap\{ y(\tau)\in V(\varepsilon_*) \}} -\frac{d\left| y(\tau) \right|}{d\tau}d\tau,\  {B}(t)= \int_{[t,\infty)\cap\{ y(\tau)\notin V(\varepsilon_*) \}} -\frac{d\left| y(\tau) \right|}{d\tau}  d\tau.
\end{align*}
Then \eqref{equ:bad2} can be restated as $
\overline{B}(t)\leq \delta|y(t)|$.
 Let $\delta'>0$ be a small number to be determined. From Lemma~\ref{lem:arclength}, for $t$ large enough,
\begin{align}\label{equ:A+B}
\overline{A}(t)+\overline{B}(t)\leq (1+\delta')(A(t)+B(t)). 
\end{align}
We claim for a moment that for $t$ large enough,
\begin{equation}\label{equ:BbarB}
\overline{B}(t)\geq (1-\delta')\sqrt{1+\varepsilon_*^{-2}}|B(t)|.
\end{equation} 
By choosing $\delta'>0$ sufficiently small depending on $\eps_*$, for $t$ large enough, there is $c=c(\eps_*)>0$ such that $c \overline{B}(t)\le\overline{B}(t)-|B(t)|$. Together with \eqref{equ:A+B} and $\overline{A}(t)\geq |A(t)|$, this ensures $c\overline{B}(t)\le \overline{B}(t)-|B(t)|\leq \delta'(A(t)+B(t))=\delta'|y(t)|.$ Then \eqref{equ:bad2} follows by taking $\delta'$ small enough.

Now we prove \eqref{equ:BbarB}. It suffices to show that for $\tau$ large enough, 
\begin{equation}\label{equ:xbarx}
y(\tau)\notin V(\varepsilon_*)\Longrightarrow  \left|y'(\tau) \right|\geq (1-\delta')\sqrt{1+\varepsilon_*^{-2}}\left|\frac{d}{d\tau}| y(\tau)| \right|.
\end{equation}
From Proposition~\ref{thm:perturb} (smallness of error),
\begin{align*}
|y'|\geq (1-b_0|y|^\rho) |\nabla g(y)|,\quad\left|\frac{d}{d\tau}| y| \right|\leq |\partial_r g(y)|+b_0|y|^\rho |\nabla g(y)|.
\end{align*}
Since $|\nabla g(y)|\geq \sqrt{1+\varepsilon_*^{-2}}|\partial_r g(y)|$ for $y\notin V(\varepsilon_*)$, \eqref{equ:xbarx} follows.
\end{proof}

\end{lemma}

\begin{proof}[Proof of {\eqref{equ:decayrate}}]
Note that ${\kappa}(t)>0$  and it diverges to $+\infty$ as $t\to \infty$. We set $t_1=1$ and inductively define
\begin{align*}
t_{i+1}=\inf\{t\geq t_{i}\, |\, {\kappa}(t)\geq 2{\kappa}(t_{i})\}.
\end{align*}
Clearly $\kappa(t_{i+1})=2\kappa(t_i)$ and $t_i\to\infty$. Our goal is to show that \eqref{eq-lemrate5} and \eqref{eq-lemrate6} hold for $i$ large enough. For the rest of the proof, we assume $i$ is large so that  the following (i)-(iii) hold for all $t\ge t_i$:
\begin{enumerate}[(i)]
\item  $\alpha_0^{-1} G_*(y(t)) \in [1-\delta,1 +\delta]$ and $\frac{1}{2} |y(t)|^{2\omega_*}\le \delta $, 
\item   $|\textup{Err}(t)|\leq \delta (1+\varepsilon_*^{-2})^{-1/2} |\nabla g(y(t))|$,
\item   Corollary~\ref{cor:10} and Lemma~\ref{lem:bad} hold.
\end{enumerate} 

First, we show that, for $\delta$ small enough,  
\begin{equation}\label{eq-thmrate_1}
y(t)\in V(\varepsilon_*)\Longrightarrow |{\kappa}'(t)-1|<4\delta.
\end{equation}
From a direct computation, 
\begin{align}\label{eq-thmrate_2}
{\kappa}'(t)=\alpha_0^{-1}\ell_*^{-1}|y(t)|^{-\ell_*+1}\left( \partial_rg(y(t))-\left\langle \textup{Err}(t),\partial_r \right\rangle \right)
\end{align}
Suppose $y(t)\in V(\varepsilon_*)$. From Corollary~\ref{cor:10}, $y(t)\in V(\varepsilon_*,\ell_*)$. Combining (i) and \eqref{equ:def_W_4}, $ \frac{1-\delta}{1+\delta}\leq \frac{\partial_r g (y(t))}{\alpha_0 \ell_* |y(t)|^{\ell_*-1}}\leq  \frac{1+\delta}{1-\delta} $. From (ii), $|\textup{Err}(t)|\leq  \delta(1+\varepsilon_*^{-2})^{-1/2} |\nabla g(y(t))|\leq   \delta |\partial_rg(y(t))|$. Then \eqref{eq-thmrate_1} holds by taking $\delta$ small. 

We are ready to prove \eqref{eq-lemrate6}.  From \eqref{eq-thmrate_2} and (ii),
\begin{align*}
|{\kappa}'(t)|\leq (1+\delta)\alpha_0^{-1}\ell_*^{-1}|y(t)|^{-\ell_*+1}|\nabla g(y(t))|.
\end{align*}
From the definition of $t_{i+1}$, $\kappa(t)\leq 2\kappa(t_i)$ for $t\in [t_i,t_{i+1}]$. Equivalently, $|y(t)|^{-\ell_*+1}\leq 2^{\frac{\ell_*-1}{\ell_*-2}}  |y(t_i)|^{-\ell_*+1}$. Therefore,
$$ |{\kappa}'(t)|\leq (1+\delta) 2^{\frac{\ell_*-1}{\ell_*-2}} \alpha_0^{-1}\ell_*^{-1}|y(t_{i})|^{-\ell_*+1}|\nabla g(y(t))| . $$
Combining the above and Lemma~\ref{lem:bad}, 
\begin{align}\label{eq-thmrate_4}
\int_{[t_i,t_{i+1}]\cap \{t \, |\, y(t)\notin V(\varepsilon_*) \}} |{\kappa}'(t)|\, dt\leq \delta M_2 {\kappa}(t_i), 
\end{align} 
for some $M_2>0$. Then for each sufficiently small $\eps>0$, \eqref{eq-lemrate6} follows from \eqref{eq-thmrate_1} and \eqref{eq-thmrate_4} by choosing $\delta = \eps/4$.

Next, we turn to \eqref{eq-lemrate5}. From \eqref{eq-thmrate_1} and \eqref{eq-thmrate_4},
\begin{align*}
{\kappa}(t_i)= \int_{[t_i,t_{i+1}]} {\kappa}'(t)dt & \le  (1+4\delta) (t_{i+1}-t_i)  + \delta M_2 {\kappa}(t_i).
\end{align*}
Then the first inequality of \eqref{eq-lemrate5} follows by rearranging terms and taking $\delta$ small enough. Lastly, we prove the second inequality of \eqref{eq-lemrate5}. Arguing similarly as above,  ${\kappa}(t_i)  \ge  (1-4\delta)(t_{i+1}-t_i)-|[t_i,t_{i+1}] \cap  \{ y(t) \notin V(\eps_*)\}|   - \delta M_2 {\kappa}(t_i).$ It is then sufficient to prove
\begin{equation}\label{eq-thmrate_7}
|[t_i,t_{i+1}] \cap  \{ y(t) \notin V(\eps_*)\}|\leq \delta M{\kappa}(t_i),
\end{equation} for some $M>0$.
From the Bochnak-{\L}ojasiewicz inequality \eqref{equ:BLo} and (i),  $|\nabla g(y(t))|\geq (1-\delta){c_3}\alpha_0 |y(t)|^{\ell_*-1}\geq  (1-\delta) 2^{-\frac{\ell_*-1}{\ell_*-2}}{c_3}\alpha_0 |y(t_i)|^{\ell_*-1}.$ Together with Lemma~\ref{lem:bad}, we deduce
$$\delta |y(t_i)|\geq (1-\delta)2^{-\frac{\ell_*-1}{\ell_*-2}}{c_3}\alpha_0 |y(t_i)|^{\ell_*-1} |[t_i,t_{i+1}] \cap  \{ y(t) \notin V(\eps_*)\}|. $$
Then rearranging terms yields \eqref{eq-thmrate_7}.
\end{proof}

\section{Examples}\label{sec:B}

We explain some notable examples and situations to which the main results apply. As noted in \cite{S0,S}, the minimal surface and harmonic map near isolated singularities can be described by the elliptic equation \eqref{equ:main}. The $G_2$ manifold \cite{Chen} is another example of a geometric PDE that can be expressed by \eqref{equ:main}. The parabolic equation \eqref{equ:parabolic} covers geometric flows such as the harmonic map heat flow \cite{ChoiMantoulidis} and the normalized Yamabe flow \cite{CCR15}. It is well-known among experts that the (rescaled) mean curvature flow near the stationary solutions is also of the form \eqref{equ:parabolic}.
\bigskip 

Since the MCF and the rescaled MCF are the $L^2$-gradient flows of the area functional and the Gaussian weighted area functional, respectively, solutions near compact critical points (i.e. minimal submanifolds and shrinking solitons) are written in the form \eqref{equ:parabolic}. If a flow $\Sigma_t^n\subset \mathbb{R}^{n+k}$ is close to a compact critical point $\Sigma$, $\Sigma_t$ can be written as a normal graph over $\Sigma$. i.e. $\Sigma_t= \mathrm{graph}(u(t))$ where $u(t)\in \Gamma(\mathbf V)$ and $\mathbf{V}$ is the normal bundle of $\Sigma$.
In terms of $g_u$, the metric on $\Sigma$ induced from the $\mathrm{graph}(u)$, the (local) area functional becomes
\begin{align*}\mathcal{F}_\Sigma(u)=\int _{\Sigma} \sqrt{\det g_u} dx^1\wedge\dots\wedge dx^n.
\end{align*}
Then it is straightforward computation that the MCF is of the form \eqref{equ:parabolic}. The rescaled MCF requires an extra term $e^{-|X(u)|^2/4}$ in the integrand of $\mathcal{F}_\Sigma (u)$. Note, however, that for both cases the term $N_2(u)$ on the right hand side is essential for two reasons: a) our choice of representing $\Sigma_t$ as a graph over $\Sigma$ and b) $\mathcal{M}_{\Sigma}$ is $L^2$-gradient flow with respect to the $L^2$-norm on $\Sigma=\textrm{graph}(0)$ while the MCF is the gradient flow with respect to changing base $\Sigma_t$ and $L^2$-norm on it. This formulation also carries to the flows in non-Euclidean, but analytic ambient Riemannian manifolds.

\bigskip

Although Theorems~\ref{thm:general_p}-\ref{thm:general_exponential_e} assumes $\Sigma$ is compact without boundary, it is generally believed that the following principle holds: a large class of PDEs should behave like a gradient flow of an analytic potential. Illustrative examples include the works of Colding-Minicozzi \cite{CM2015}, Chodosh-Schulze \cite{MR4332673} and Lee-Zhao \cite{lee2023uniqueness}, showing the uniqueness of the tangent flow for cylindrical shrinkers and asymptotically conical shrinkers, respectively. See also \L ojasiewicz-Simon inequality for problems with boundary condition and non-analytic energy functional \cite{MR1800136}, for higher order equations \cite{MR4218365}, and for functionals on Banach spaces other than $L^2$ space \cite{MR4129355,MR3926130}. We finish this section with a non-compact example which is relevant to Thom's gradient conjecture and the next-order asymptotics: the ancient MCFs having bubble-sheet tangent flow in $\mathbb{R}^4$ by Du and Haslhofer \cite{MR4666631,MR4635341}. This work later leads to the classification of bubble-sheet ovals in $\mathbb{R}^4$ by Daskalopoulos, Du, Haslhofer, Sesum, and the first author \cite{CDDHS}. We note that backward asymptotic analysis at $-\infty$ often plays a crucial role in the classification of such solutions: see also ancient mean convex MCF in $\mathbb{R}^3$ \cite{ADS,brendle2019uniqueness}, ancient MCF with low entropy in $\mathbb{R}^3$ \cite{MR4448681} and ancient MCF asymptotically to a cylinder \cite{MR4438354}.


\begin{example}[bubble-sheet tangent flow] Suppose an ancient MCF, $\Sigma^3_t\subset \mathbb{R}^{4}$, has bubble-sheet tangent flow: the rescaled flow $\tilde \Sigma_\tau = e^{\frac\tau2} \Sigma_{ -e^{-\tau} }$ converges to $\mathbb{R}^2 \times \sqrt 2 \mathbb{S}^1$ as $\tau \to -\infty$. Representing $\tilde \Sigma_\tau$ as $\textrm{graph}(u(\tau))$ over  $(y_1,y_2,\theta)\in  \mathbb{R}^2 \times \sqrt 2 \mathbb{S}^1$, $u(\tau)$ smoothly converges to $0$ on compact sets of $(y_1,y_2,\theta)$ as $\tau \to -\infty$. \cite[Theorem 1.1]{MR4666631} and \cite[Theorem 1.3]{MR4635341} show, after a rotation in the $y_1y_2$-plane, $\tilde \Sigma_\tau$ shows one of following behaviors: either\begin{equation}\label{equ:ex_cases}
u(\tau)\asymp \tfrac{y_1^2+y_2^2-4}{2\sqrt{2}\tau},\quad u(\tau)\asymp \tfrac{y_1^2 -2}{2\sqrt{2}\tau},\quad \text{or}\quad u(\tau)=o(\tau^{-1}).
\end{equation} In the last case, $u(\tau)$ actually decays exponentially, corresponding to Theorem~\ref{thm:general_exponential_p}. The first two cases correspond to slowly decaying gradient flow.

Ignoring $\theta$-part as the solution is $\mathbb{S}^1$-symmetric, the kernel of $\mathcal{L}_\Sigma$, the linearized shrinker operator, is {essentially} three-dimensional and it is spanned by $\{y_1^2-2,y_2^2-2,\sqrt{2}y_1y_2\}$. It may be identified with $\mathbb{R}^3$ 
\begin{equation}\label{equ:ex_id}
x=(x_1,x_2,x_3)\mapsto x_1\cdot (y_1^2-2)+x_2\cdot(y_2^2-2)+x_3\cdot \sqrt{2}y_1y_2.
\end{equation}
Let $x(\tau)=(x_1(\tau),x_2(\tau),x_3(\tau))$ be the projection of $u(\tau)$ to $\ker \mathcal{L}_{\Sigma}$. Here, certain delicate truncation is needed as $\tilde{\Sigma}_\tau$ may not be a global graph over  $\mathbb{R}^2\times\sqrt{2}\mathbb{S}^1$. We ignore this so to present a connection in a formal level.  \cite[Proposition 3.1]{MR4666631} shows, up to a small error, $x(\tau)$ is a gradient flow
\begin{equation}\label{equ:ex_grad}
x'(\tau)=-\nabla f(x(\tau)), \text{ where }f(x)=\tfrac{2\sqrt 2}{3}x_1^3+\sqrt{2}(x_1+x_2)x_3^2+\tfrac{2\sqrt 2}{3}x_2^3. 
\end{equation}Let $\hat{f}$ be the restriction of $f$ on $\mathbb{S}^2$. In view of Theorem \ref{thm:general_p} and Remark \ref{rmk:cases} (1), we are interested in the \textbf{non-positive} critical values of $\hat{f}$.  {All critical values of $\hat {f}$ are non-zero and there are two negative ones.}

\noindent \textbf{Case 1:} $(x_1,x_2,x_3)=-(1/\sqrt{2},1/\sqrt{2},0)$ with a critical value $\hat f = -2/3$. The corresponding solution to \eqref{equ:ex_grad} is given by $x(\tau)= \frac{1}{2\tau}\cdot (1/\sqrt{2},1/\sqrt{2},0)$. Through the identification, this is the first case in \eqref{equ:ex_cases}. 

\noindent \textbf{Case 2:} $(x_1,x_2,x_3)=-(\frac{1+\cos\phi}{2} ,\frac{1-\cos\phi}{2},-\frac{\sqrt{2}}{2}\sin\phi)$, for $\phi \in [0,2\pi)$, with $\hat f=-2\sqrt{2}/3$. This results in $x(\tau)= \frac{1}{2\sqrt{2} \tau}\cdot (\frac{1+\cos\phi}{2} ,\frac{1-\cos\phi}{2},-\frac{\sqrt{2}}{2}\sin\phi)$, and it corresponds to   $u(\tau)\sim \frac{1}{2\sqrt{2}\tau}((y_1\cos\frac{\phi}{2}-y_2\sin\frac{\phi}{2})^2-2)$. This is the second alternative in \eqref{equ:ex_cases}.

The critical point in \textbf{Case 1} is isolated, but the critical points in \textbf{Case 2} form a circle in $\mathbb{S}^2$. In \cite{MR4666631}, Du and Haslhofer initially left the convergence of direction in the \textbf{Case 2} unsettled (so the rotation in $y_1y_2$-plane may not be fixed), but resolved this issue in \cite{MR4635341}. This heuristic  also applies in the forward-in-time asymptotics of the rescaled MCF. This is obtained by Sun and Xue \cite[Theorem 1.3]{SunXue} and it is an important step in proving cylindrical singularities are generically isolated.

\end{example}
 
\appendix \section{Tools}\label{sec:A}

\begin{lemma}\label{lem:sqrt-lipschitz} 
For \(T\in(0,\infty)\), let \(\xi\) be an absolutely continuous, nonnegative
function on \([0,T]\), and set
\[
X(t):=\sqrt{\xi(t)},\qquad 
\Omega:=\{t\in[0,T]: \xi(t)>0\}.
\]
Let \(Y\in L^1([0,T])\) satisfy \(Y(t)=0\) for a.e.
\(t\in[0,T]\setminus\Omega\). Suppose that
\[
\frac{\xi'(t)}{2\sqrt{\xi(t)}}\ge Y(t)
\]
for a.e. \(t\in\Omega\). Then $X'\geq Y$ in the integral sense. Precisely, for all \(0\le a<b\le T\),
\[
X(b)-X(a)\ge \int_a^b Y(t)\,dt .
\]
Similarly, if
\[
\frac{\xi'(t)}{2\sqrt{\xi(t)}}\le Y(t)
\]
for a.e. \(t\in\Omega\), then $X'\leq Y$ in the integral sense. Precisely, for all \(0\le a<b\le T\),
\[
X(b)-X(a)\le \int_a^b Y(t)\,dt .
\]
\end{lemma}
\begin{proof}
We prove the first assertion; the second can be proved similarly.

For \(\varepsilon>0\), set
\[
X_\varepsilon(t):=\sqrt{\xi(t)+\varepsilon}.
\]
Then \(X_\varepsilon\) is absolutely continuous on \([0,T]\), and
\[
X_\varepsilon'(t)
=
\frac{\xi'(t)}{2\sqrt{\xi(t)+\varepsilon}}
\]
for a.e. \(t\in[0,T]\). On \(\Omega\), this can be written as
\[
X_\varepsilon'(t)
=
\frac{\sqrt{\xi(t)}}{\sqrt{\xi(t)+\varepsilon}}
\frac{\xi'(t)}{2\sqrt{\xi(t)}}
\ge
\frac{\sqrt{\xi(t)}}{\sqrt{\xi(t)+\varepsilon}}\,Y(t)
\]
for a.e. \(t\in\Omega\). Since \(Y=0\) and \(X_\varepsilon'=0\) a.e. on
\([0,T]\setminus\Omega\), we obtain, for \(0\le a<b\le T\),
\[
X_\varepsilon(b)-X_\varepsilon(a)
\ge
\int_a^b
\frac{\sqrt{\xi(t)}}{\sqrt{\xi(t)+\varepsilon}}\,Y(t)\,dt .
\]
Then the assertion follows by sending $\varepsilon$ to zero.
\end{proof}

The following ODE lemma was proved in \cite{MZ}.
\begin{lemma}[Merle-Zaag ODE lemma]\label{lem-MZODE}
Let $X_+$, $X_0$, $X_- : [0,\infty) \to [0, \infty)$ be absolutely continuous functions such that $X_+(t) + X_0(t) + X_-(t) > 0$ for all $t\geq  0$ and $		\liminf_{t \to \infty} X_+(t) = 0.$ Suppose there exist $b>0$ and functions $Y_+$, $Y_0$, $Y_-$ on $[0,\infty)$ that satisfy $
|Y_+|^2 + |Y_0 |^2 + |Y_-|^2 = o(1)(|X_+|^2 + |X_0 |^2 + |X_-|^2) $ as $t\to \infty$ and 
\begin{equation}  \label{eq:mz.ode.cor}
				X_+' - bX_+ \geq  Y_+, \quad 
			|X_0'|  \leq Y_0, \quad 
			X_-' + bX_-  \leq   Y_- ,
		\end{equation} for $t\ge0$. 
Then one of the following holds: either
\begin{align}\label{eq:mz.ode.cor.A}
&X_+(t)+X_-(t)=o(1)X_0(t), \text{ or}
\\ \label{eq:mz.ode.cor.B}
&X_+(t)+X_0(t)=o(1)X_-(t).
\end{align}
Moreover, suppose \eqref{eq:mz.ode.cor.B} holds true. Then for all $\varepsilon>0$,
\begin{align}\label{eq:mz.ode.cor.B1}
\limsup_{t\to\infty} e^{(b-\varepsilon)t}(X_+(t) + X_0(t) + X_-(t))=0.
\end{align}
\end{lemma} 

\begin{lemma}[interpolation]\label{lemma-interpolation} Suppose a function of single variable $g(t)$ has bounds\[\sum_{m=0}^k \sup_{t\in[-1,1]} \left |\frac{d^m }{dt^m}g(t) \right | \le C_k \quad  \text{ and } \quad \sup_{t\in [-1,1]} |g(t)| \le C_0.\] Then, for every $0 < \ell < k$, there holds a bound 
\[\left |\frac{d^\ell}{dt ^\ell } g (0)\right | \le  \beta \cdot C_0^{\alpha} C_k^{1-\alpha} \]
for some $\alpha=\alpha(\ell,k)>0$ and $\beta= \beta(\ell,k)>0$. Moreover, $\alpha(\ell,k)$ converges to $1$ as $k \to \infty $ while $\ell$ is fixed. 
\begin{proof} We prove the inequality with $\beta\equiv \sqrt{2}$ for functions which are compactly supported in the interior of $(-1,1)$. This is sufficient as one may multiply a general function by a cut-off and then apply the result. By an argument which uses an induction and the integration by parts, for every $ \ell  <  k$,  $\Vert \partial^\ell  _t g \Vert _{L^2} \le \Vert  g \Vert _{L^2}^{1-\frac \ell k} \Vert \partial ^k _t g\Vert_{L^2} ^{\frac{\ell }{k}}   $. Thus by the fundamental theorem of calculus and the Cauchy--Schwarz inequality\[ |\partial ^{\ell}_t g(0)|^2  \le  \Vert \partial ^\ell_t g\Vert_{L^2} \Vert \partial^{\ell+1}_t g\Vert _{L^2}\le 2 C_0^{2\alpha} C_{k}^{2(1-\alpha)},\] for $\alpha=1-(\ell+\tfrac12)/k>0$. For fixed $\ell$, $\alpha (\ell,k)\to 1$ as $k\to \infty$. 
\end{proof}
\end{lemma}

In the next lemma, we record the elliptic regularity in \cite[(1.13)]{S0}.

\begin{lemma}[elliptic regularity]\label{lem:regularity_e}
There exists $\rho_0>0$ small depending on $\mathcal{M}_\Sigma, m$ and $N_1$ such that the following holds. Let $u$ be a solution to \eqref{equ:main} for $t\in [T-2,T+2]$. Assume $\Vert u \Vert_{C^1}(t)\leq \rho_0$ for $t\in [T-2,T+2]$. Then for any $s\in\mathbb{N}$, there exists a positive constant $C=C(m,\mathcal{M}_\Sigma,N_1,s)$ such that
\begin{align*}
\sup_{t\in [T-1,T+1]}\Vert u \Vert^2_{C^s}(t) \leq C  \int_{T-2}^{T+2} \Vert u \Vert_{L^2}^2(t)\, dt.
\end{align*} 
\end{lemma}

{The parabolic regularity below is a straightforward minor extension of the one presented in \cite[Section 4]{S0}.} 
\begin{lemma}[parabolic regularity]\label{lem:regularity}
Let $u\in C^\infty(Q_{0,\infty},\widetilde{\mathbf{V}})$ be a solution to \eqref{equ:parabolic}. Assume $\lim_{t\to \infty} \Vert u \Vert_{H^{n+4}}(t)=0$. Then for all $s\in \mathbb{N}$, $\lim_{t\to \infty} \Vert u \Vert_{C^s}(t)=0$. Moreover, if $\Vert u \Vert_{H^{n+4}}(t)=O(e^{\gamma t})$ for some $\gamma<0$, then $\Vert u \Vert_{C^s}(t)=O(e^{\gamma t})$.
\end{lemma}
 
\bibliography{ChoiHung}
\bibliographystyle{alpha}

\end{document}